\documentclass[12pt]{amsart}

\newcommand{\indentalign}{\hspace{0.3in}&\hspace{-0.3in}}
\newcommand{\la}{\langle}
\newcommand{\ra}{\rangle}

\def\bbbone{{\mathchoice {1\mskip-4mu {\rm{l}}} {1\mskip-4mu {\rm{l}}}
{ 1\mskip-4.5mu {\rm{l}}} { 1\mskip-5mu {\rm{l}}}}}

\newcommand{\CC}{{\mathbb C}}
\newcommand{\RR}{{\mathbb R}}

\newcommand{\sgn}{\operatorname{sgn}}

\newcommand{\sech}{\operatorname{sech}}
\newcommand{\defeq}{\stackrel{\rm{def}}{=}}
\def\squarebox#1{\hbox to #1{\hfill\vbox to #1{\vfill}}}

\newcommand{\spn}{\operatorname{span}}

\usepackage{amssymb}
\usepackage{amsthm}
\usepackage{amsxtra}
\usepackage{graphicx}
\usepackage{color}

\newtheorem{theorem}{Theorem}[section]
\newtheorem{definition}[theorem]{Definition}
\newtheorem{proposition}{Proposition}[section]
\newtheorem{lemma}[proposition]{Lemma}
\newtheorem{corollary}[proposition]{Corollary}

\theoremstyle{remark}
\newtheorem{remark}[proposition]{Remark}

\numberwithin{equation}{section}

\ifx\pdfoutput\undefined
  \DeclareGraphicsExtensions{.pstex, .eps}
\else
  \ifx\pdfoutput\relax
    \DeclareGraphicsExtensions{.pstex, .eps}
  \else
    \ifnum\pdfoutput>0
      \DeclareGraphicsExtensions{.pdf}
    \else
      \DeclareGraphicsExtensions{.pstex, .eps}
    \fi
  \fi
\fi

\title{Effective dynamics of double solitons for perturbed mKdV}
\author{Justin Holmer}
\address{Brown University}
\author{Galina Perelman}
\address{Ecole Polytechnique}
\author{Maciej Zworski}
\address{University of California, Berkeley}

\begin{document}

\begin{abstract}
We consider 
the perturbed mKdV equation 
$\partial_t u = -\partial_x (\partial_x^2u + u^3- b(x,t)u)$ where 
the potential $b(x,t)=b_0(hx,ht)$, $ 0 < h \ll 1 $, is slowly varying
with a double soliton initial data.
On a dynamically interesting time scale the solution is $ {\mathcal O}(h^2) $
close in $H^2$ to a double soliton 
whose position and scale parameters 
follow an effective dynamics, 
a simple system of ordinary differential equations.  
These equations are formally 
obtained as Hamilton's equations for the restriction of the mKdV Hamiltonian to the submanifold of solitons.  
The interplay between algebraic aspects of 
complete integrability of the unperturbed equation
and the analytic ideas related to soliton stability is central in the proof. 

\end{abstract}

\maketitle

\section{Introduction}
\label{S:intro}

We consider $2$-soliton solutions of the modified 
KdV equation with a slowly varying external potential \eqref{E:pmKdV}.
The purpose of 
the paper is to find 
minimal {\em exact} effective dynamics valid for a {
long} time in the semiclassical sense and describing
non-perturbative 2-soliton interaction. 
In standard quantum 
mechanics the natural long  time 
for which the semiclassical approximation is 
valid is the {Ehrenfest time}, $ \log(1/h)/h $ -- see
for instance \cite{BR}.
The semiclassical 
parameter, $ h $,  quantifies the slowly varying nature of
the potential. 

\renewcommand\thefootnote{\dag}%

Unlike in the case of single-particle semiclassical dynamics,
that is, for the linear Schr\"odinger equation with 
a slowly varying potential, the exact effective dynamics
valid for such a long time requires $ h^2 $-size corrections\footnote{A compensation for that comes however at having
the semiclassical propagation accurate for larger values of
$ h$.}.
Those corrections appeared as unspecified $ {\mathcal O} (h^2) $ additions
to Newton's equations (which give the usual semiclassical approximation)
in the work of Fr\"ohlich-Gustafson-Jonsson-Sigal \cite{FrSi}
on $1$-soliton propagation. That paper and its 
symplectic point of view were the starting point for \cite{HZ1, HZ2}.

\begin{figure}
\includegraphics[width=6in]{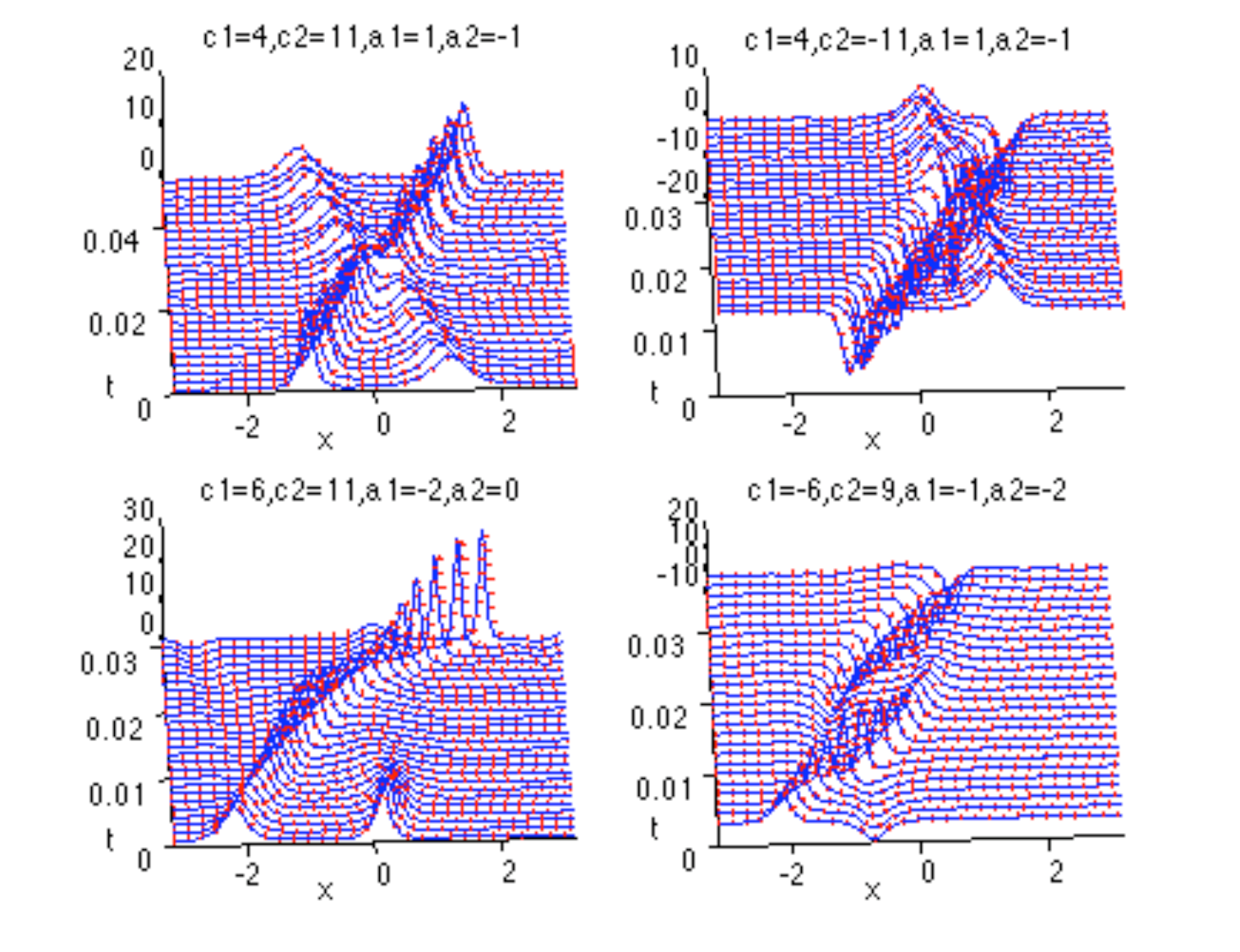}
\caption{A gallery of numerical experiments showing
agreement with the results of the main theorem 
(clockwise from
the left hand corner) for the external fields listed in
\eqref{eq:listex} with the indicated initial data. The continuous
lines are the numerically computed solutions and the dotted lines
follow the evolution given by \eqref{E:eom}. The main theorem 
does not apply to the bottom two figures on the whole interval
of time due to the crossing of $ c_j$'s -- see Fig.\ref{f:cros}.
In the first figure in the second line, \eqref{E:eom} still apply
directy, but in the second one further modification is needed
to account for the signs.}
\label{f:sol1}
\end{figure}

Following the $1$-soliton analysis of \cite{HZ1,HZ2}
the semiclassical dynamics for $2$-solitons considered here is obtained 
by restricting the
Hamiltonian to the symplectic manifold of $2$-solitons and
considering the finite dimensional dynamics there. 
The numerical experiments \cite{Codes} show a remarkable
agreement with the theorem below. However, they also 
reveal an interesting scenario not covered by our theorem:
the velocities of the solitons can almost cross within
exponentially small width in $ h $ and the effective
dynamics remains valid. Any long time analysis involving
multiple interactions of solitons has to explain this
avoided crossing which perhaps could be replaced by 
a direct crossing in a different parametrization. This
seems the most immediate open problem of phenomenological
interest.

The effective dynamics follows a long tradition 
of the use of modulation parameters in soliton propagation
-- see for instance \cite{BP},\cite{MM},\cite{MMT},\cite{P},\cite{RSS}
and the numerous references given there. For 
non-linear dispersive equations with non-constant coefficients
one can consult, in addition to \cite{FrSi}, \cite{AFS},\cite{ZS},\cite{ZW},\cite{KMR},
and references given there.

Here we avoid generality 
and, as described above, the aims are more modest: 
for the physically relevant cubic non-linearity we benefit
from the completely integrable structure and using 
classical methods we can give
a remarkably accurate and phenomenologically 
relevant description of $2$-soliton interaction.
The equation \eqref{E:pmKdV}
shares many features
with the dynamical Gross-Pitaevskii equation,
\[ i \partial_t u = - \partial_x^2 u - |u|^2 u + V ( x) u \,,\]
but is easier to study, mathematically and numerically. 
In a recent numerical study Potter \cite{Po} showed that
the same effective dynamics applies very well to 
$N$-soliton trains in the case of perturbed mKdV and 
NLS. The soliton matter-wave trains created for Bose-Einstein
condensates \cite{phys-paper} were a good testing ground
and our effective dynamics gives an alternative
explanation of the observed phenomena. At the moment it
is not clear how to obtain {\em exact} effective dynamics
for the perturbed NLS.

To state the exact result we recall the perturbed 
mKdV equation \cite{DJ},\cite{DS}:
\begin{gather}
\label{E:pmKdV}
\begin{gathered}
\partial_t u =  - \partial_x( \partial_x^2u - b(x,t)u + 2u^3) \,, \\ 
b( x, t ) =  
b_0(hx,ht)\,, \ \  0 < h \ll 1\,, \ \ \partial^\alpha b_0 \in L^\infty 
( \RR^2 ) \,.
\end{gathered}
\end{gather}
For $ b \equiv 0 $
the equation is completely integrable and has a special class of 
$N$-soliton solutions, $ q_N ( x , a , c ) $, $ a \in \RR^N $, $ c \in 
\RR^N $ -- see \S \ref{dsm} and \S \ref{S:q-properties} below.  For 
$ N = 2 $ we obtain 

\medskip

\noindent
{\bf Theorem.}
{\em Let $ \delta_0 > 0 $ and $ \bar a , \bar c \in 
\RR^n $.   Suppose that $ u ( x , t ) $ solves
 \eqref{E:pmKdV} with 
\begin{equation}
\label{eq:Tin}
u( x , 0 ) = q_2 ( x , \bar a, \bar c ) \,, \ \ \ 
|\bar c_1  
\pm \bar c_2 | > 2\delta_0 > 0 \,, \ \ 2\delta_0< | \bar c_j| < (2\delta_0)^{-1} \,. 
\end{equation}
Then, for $ t < T(h)/h $,
\begin{equation}
\label{eq:Th}
\| u(\cdot,t) - q_2(\cdot,a(t),c(t))\|_{H^2} \leq C h^2 e^{Cht}\,, 
\quad C=C(\delta_0,b_0)>0  \,, 
\end{equation}
where $a(t)$ and $c(t)$ evolve according to the {\em effective} equations
of motion,
\begin{gather}
\label{E:eom}
\begin{gathered}
\dot a_j = c_j^2 - \sgn(c_j )  \partial_{c_j}B(a,c,t) \, \ \ 
\dot c_j = \sgn ( c_j )  \partial_{a_j}B(a,c,t) \\
B(a,c,t) \defeq \frac12 \int b(x,t) q_2(x,a,c)^2 \, dx \,.
\end{gathered}
\end{gather}
The upper bound $ T(h)/h $ for the validity of \eqref{eq:Th} is given in terms of
\begin{equation}
\label{eq:T0h} T ( h ) = \min ( \delta \log( 1/h ) , T_0 ( h ) ) \, , \qquad \delta=\delta(\delta_0,b_0)>0 
\end{equation}
where for $ t < T_0 ( h ) / h $, $ | c_1 ( t ) \pm c_2 ( t ) | > \delta_0 > 0$  and $\delta_0 < | c_j ( t ) | < \delta_0^{-1}$. 
Under the assumption \eqref{eq:Tin} on $ \bar c$, 
$ T_0( h ) > \delta_2  $,
where $ \delta_2 =\delta_2(\delta_0,b_0)>0$ is independent of $ h $ -- see \eqref{eq:ODE}.}

\noindent
{\bf Remarks.} 
{\bf 1.} We expect the same result to be true for all $ N $ with 
$ H^2 $ replaced by $ H^N $. For $ N = 1 $ it follows directly from 
the arguments of \cite{HZ2}. That case is also implicit in this
paper: single soliton dynamics 
describes the propagation away from the interaction 
region. 

\noindent
{\bf 2.} The Ehrenfest time bound, $ T ( h) \leq \delta \log ( 1/ h ) $, 
is probably optimal if we insist on the agreement with 
classical equations of motion \eqref{E:eom}. We expect
that the solution is close to a soliton profile $ q_2 ( x , a , c ) $
for much longer times ($ h^{-\infty} $?) but with a modified
evolution for the parameters. One difficulty is the lack of 
a good description of the long time behaviour of
time dependent linearized evolution with $ b $ present -- see \S \ref{S:correction}. However, the modified equations would lack the transparency of
\eqref{E:eom} and would be harder to implement. The numerical 
study \cite{Po} suggests that for the minimal exact dynamics 
the error bound $ {\mathcal O} (h^2) $ in \eqref{eq:Th} is optimal.

\noindent
{\bf 3.} As shown by the top two plots in Fig.\ref{f:sol1} the agreement
of the approximations given by \eqref{E:eom} and numerical solutions
of \eqref{E:pmKdV} is remarkable. The codes are available at 
\cite{Codes}, see also \S \ref{nume}. Experiments support
the preceding remark.

\noindent
{\bf 4.} The condition that $ | c_1 ( t) \pm c_2 ( t ) | > \delta_1 $,
that is, that the perturbed effective dynamics avoids the 
lines shown in Fig.\ref{f:col}, could most likely 
be relaxed. Allowing that provides more interesting dynamics
as then the solitons can interact multiple times.
As discussed in \S \ref{dyass} and Appendix \ref{A:ode},
we expect avoided crossing after $ \pm c_j ( t )$'s get within
$ \exp ( - c/h)$ of each other -- see Fig.\ref{f:cros}. 
Examples of such evolution, and the comparisons with effective
dynamics, are shown in the lower two plots in Fig.\ref{f:sol1}. 
On closer inspection the agreement between the 
solutions and solitons moving according to effective 
dynamics is not as dramatic as in the
case when $ \pm c_j $'s stay away from each other but for smaller
values of $ h $ the result should still hold. We concentrated on
the simpler case at this early stage. 

\noindent
{\bf 5.} The equation \eqref{E:pmKdV} 
is globally well-posed in $H^k$, $k\geq 1$ 
under even milder 
regularity hypotheses on $b$. This can be shown by modifying the techniques of Kenig-Ponce-Vega \cite{KPV} -- see Appendix \ref{A:gwp}.  Although
for $ k \geq 2 $ more classical methods are available, we opt for 
a self-contained treatment dealing with all $ H^k $'s at once.

\noindent
{\bf 6.} 
Studies of single solitons for 
perturbed KdV, mKdV, and their generalizations
were conducted by Dejak-Jonsson \cite{DJ} and Dejak-Sigal \cite{DS}. 
The perturbative terms, $ b ( x, t ) $, 
were assumed to be not only slow varying but also small in size. 
The mKdV results of \cite{DJ} are improved by following \cite{HZ2}.
For KdV one does not expect the same behaviour as for mKdV and 
the $ {\mathcal O} ( h^2 ) $-approximation similar to \eqref{eq:Th}
is not valid -- see \cite{Munoz, HP} for finer analysis of that case.

\noindent
{\bf 7.} The conditions that $ u(x,0) = q_2 ( x , \bar a , \bar c ) $
can be relaxed by allowing a small perturbation in $ H^2 $ 
-- see \cite{DaVe} for the
adaptation of \cite{HZ2} to that case. Similar statements
are possible here but we prefer the simpler formulation
both in the statement of the theorem and in the proofs.

\medskip

In the remainder of the introduction we will explain the
origins of the effective dynamics \eqref{E:eom}, outline the proof, and 
comment on numerical experiments.

\begin{figure}
\includegraphics[width=6in]{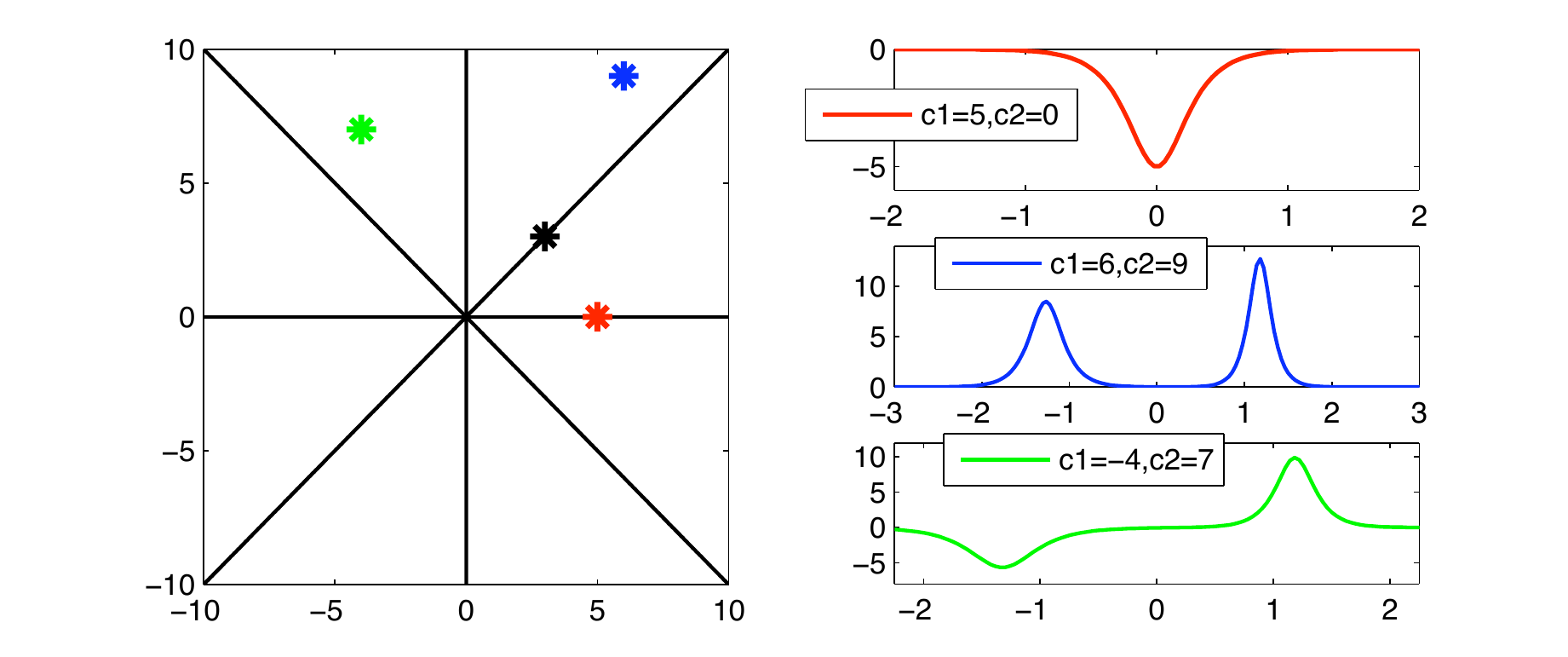}
\caption{On the left we show $ \RR^2 \setminus {\mathcal C} $ and
on the right examples of double solitons corresponding to $ (c_1, c_2 ) $
indicated on the left (with $ a_1 = a_2 = 0 $ in the 
first figure and $ a_2 = -a_1 = 1 $, in the other two).
At the coordinate axes the double soliton 
degenerates into a single soliton. As one approaches the 
lines $ c_1 = \pm c_2 $ the solitons escape to infinities
in opposite direction.}
\label{f:col}
\end{figure}

\subsection{Double solitons for mKdV}
\label{dsm}
The \emph{single soliton} solutions to mKdV, \eqref{E:pmKdV} with 
$ b \equiv 0 $, are described in terms of 
the profile $\eta(x,a,c)$ as follows.  Let $\eta(x) = \sech x$ 
so that $-\eta + \eta'' + 2 \eta^3=0$, and 
let $\eta(x,c,a) = c\eta(c(x-a))$ for $a\in \mathbb{R}$, $c \in 
\RR \setminus 0$.  Then a single soliton defined by 
$$u(x,t) = \eta(x,a+c^2t,c) \,$$
is easily verified to be an exact solution to mKdV.  
Such solitary wave solutions are available for many 
nonlinear evolution equations.  
However, mKdV has richer structure -- it is completely 
integrable and can be studied using the inverse scattering method 
(Miura \cite{Miura}, Wadati \cite{W}).  One of the consequences is
the availability of larger families of 
explicit solutions. In the case of mKdV, 
we have \emph{$N$-solitons} and \emph{breathers}.  
In this paper we confine our attention to the $2$-soliton 
(or double soliton), which is described by the profile $ q_2(x,a,c)$ 
defined in \eqref{E:q-def} below.  The four real parameters,
$ a \in \RR^2 $, and $ c \in \RR^2 \setminus {\mathcal C} $,
\[ {\mathcal C} \defeq  \{ (c_1, c_2)  \; : \; c_1 = \pm c_2 
\} \cup \RR \times \{ 0 \} \cup \{ 0 \} \times \RR \,, \]
describe the position ($a$) and scale ($c$) of the double soliton.
At the diagonal lines the parametrization 
degenerates: for $ c_1 = \pm c_2 $, $ q_2 \equiv 0 $. At 
the coordinate axes in the $c$ space, we recover single solitons:
\[  q_2 ( x, a, ( c_1 , 0 ) ) = - c_1 \eta ( x , a_1 , c_1 ) \,, \ \ 
q_2 ( x, a , ( 0 , c_2 )) = c_2 \eta ( x, a_2 , c_2 ) \,. \]
Fig.\ref{f:col} shows a few examples.

Solving mKdV with $ u ( x, 0 ) = q_2 ( x, a , c ) $ gives the 
solution
\[  u ( x ,t ) = q_2 ( x , a_1 + t c_1^2, a_2 + t c_2^2 , c) \,,\]
that is, the double soliton solution. 

If, say,  $0 < c_1 < c_2$, then for $|a_1-a_2|$ large,
$$q(x,a,c) \approx \eta(x,a_1+\alpha_1,c_1) + \eta(x,a_2+\alpha_2,c_2)$$
where $\alpha_j$ are shifts defined in terms of $c$,
see Lemma \ref{L:q-asymp} for the precise statement.
This means that for large positive and negative times the evolving
double soliton is effectively a sum of single solitons. The decomposition
can be made exact preserving the particle-like nature of single 
solitons even during the interaction -- see \eqref{eq:Qdec} and
Fig.\ref{F:2sol-interact}.

We consider the set of $2$-solitons as a submanifold of $ H^2 ( \RR ; \RR ) $
with $ 8 $ open components corresponding to the components of 
$ \RR^2 \setminus {\mathcal C} $:
\begin{equation}
\label{eq:Msol}
M = \{ \, q(\cdot, a,c ) \, | \, a=(a_1,a_2)\in \mathbb{R}^2 \,, c=(c_1,c_2)\in \mathbb{R}^2 \setminus {\mathcal C}  \, \} \,. 
\end{equation}
As in the case of single solitons this submanifold is symplectic
with respect to the natural structure recalled in the
next subsection.

\subsection{Dynamical structure and effective equations of motion}
\label{dyass}

The equation \eqref{E:pmKdV} is a
Hamiltonian equation of evolution for 
\begin{equation}
\label{eq:Ham}
H_b ( u ) = \frac 12 \int ( u_x^2 - u^4 + bu^2 ) dx \,, 
\end{equation}
on the Schwartz space, $ {\mathcal S} ( \RR ; \RR ) $ equipped with 
the symplectic form
\begin{equation}
\label{eq:omega}
 \omega ( u , v ) = \frac12 \int_{-\infty}^{+\infty} \int_{-\infty}
^x ( u ( x ) v ( y ) - u ( y ) v ( x ) ) dy dx \,. 
\end{equation}
In other words, \eqref{E:pmKdV} is equivalent to 
\begin{equation}
\label{eq:Hamp}
u_t = \partial_x H'_b ( u ) \,,    \ \ \langle H_b'( u ) , \varphi\rangle 
\defeq \frac d {ds} H_b ( u + s \varphi ) |_{s=0 } \,,
\end{equation}
and $ \partial_x H_b' ( u ) $ is the Hamilton vector field of $ H_b $, 
$ \Xi_{H_b} $,  with respect to $ \omega $: 
$$ \omega ( \varphi,  \Xi_{ H_b} ( u ) ) = 
\langle H_b'( u ) , \varphi \rangle \,. $$

For $ b= 0 $, $ \Xi_{H_0}  $ is tangent to the 
manifold of solitons \eqref{eq:Msol}. Also, $ M $ is
symplectic with respect to $ \omega $, that is,
$ \omega $ is nondegenerate on $ T_u M $, $ u \in M $. Using the stability 
theory for $2$-solitons based on the work of Maddocks-Sachs \cite{MS},
and energy methods (enhanced and simplified using 
algebraic identities coming from complete integrability of
mKdV) we will show that the solution to \eqref{E:pmKdV} with 
initial data on $ M $ stays close to 
$ M $ for $ t \leq \log ( 1/h ) / h $. 

A basic intuition coming from symplectic geometry then indicates
that $ u ( t ) $ stays close to an integral curve on $ M $ of 
the Hamilton vector field (defined using 
$ \omega | _M $) of $ H_b $ {\em restricted} to $ M $:
\begin{gather}
\label{eq:HbM} 
\begin{gathered}
H_{\rm{eff}} ( a , c ) \defeq H_b |_M ( a, c ) = H_0 |_M ( a , c ) + \frac 12 \int b ( x ) 
q_2 ( x, a , c ) ^2dx \,, \\ H_0 |_M ( a, c ) = - \frac 1 3 ( |c_1|^3 + |c_2|^3 ) \,, \\ 
\omega |_M = d a_1 \wedge d|c_1| + d a_2 \wedge d |c_2| \,, \\
\Xi_{ H_{\rm{eff} } } = \sum_{ j =1}^2  \sgn (c_j) ( \partial_{a_j } H_{\rm{eff}} \,
\partial_{c_j} 
 - \partial_{c_j} H_{\rm{eff}} \, \partial_{a_j}  ) \,. 
\end{gathered}
\end{gather}
The effective equations of motion \eqref{E:eom} follow.
This simple but crucial observation was made in \cite{HZ1},\cite{HZ2}
and it did not seem to be present in earlier mathematical work on
solitons in external fields \cite{FrSi}. 

The condition made in the theorem,
that $ | c_1 ( t ) \pm c_2 ( t ) | $ and $ |c_j ( t ) | $ 
are bounded away from zero for $ t < T_0 ( h ) / h $ (where 
$ T_0 ( h ) $ could be $ \infty $),  follows from a 
condition involving a simpler system of decoupled
$ h$-independent ODEs -- see Appendix \ref{A:ode}. Here 
we state a condition which gives an $ h$-independent $ T_0 $ 
appearing in \eqref{eq:T0h}.

Suppose we are given $ b ( x, t) = b_0(hx,ht)$ in 
\eqref{E:pmKdV} and the initial condition is 
given by $ q_2 ( x , \bar a , \bar c )$, $ \bar a =
(\bar a_{1},\bar a_{2})$, $\bar c =(\bar c_{1}, \bar c_{2})$, 
$ | \bar c_1 \pm \bar c_2 | > \delta_0 $, $ | \bar c_j | > \delta_0 $,
We  consider an $h$-independent system of two 
decoupled differential equations for 
$$A(T)=(A_1(T),A_2(T)) \,, \ \ \ C(T)=(C_1(T),C_2(T)) \,, $$
given by 
\begin{equation}
\label{eq:ODEsh} \left\{
\begin{aligned}
& \partial_T A_j = C_j^2 - b_0(A_j,T) \\
& \partial_T C_j = C_j \partial_x b_0(A_j, T)
\end{aligned}
\right. \,, 
\qquad A(0) = \bar a h \,, \quad C(0) = \bar c \,, \ \ \ j=1,2\,.
\end{equation}
Then, for a given $ \delta_1 < \delta_0 $, 
$ T_0 ( h ) $ in \eqref{eq:T0h} can be replaced by 
\begin{equation}
\label{eq:ODE} T_0 \defeq \sup \{ T \; : \; 
| C_1 ( T ) \pm C_2 ( T) | > \delta_1 \,, \ |C_j ( T ) | > \delta_1\,, \ \ 
j=1,2 \} \,. 
\end{equation}

\begin{figure}
\includegraphics[width=6in]{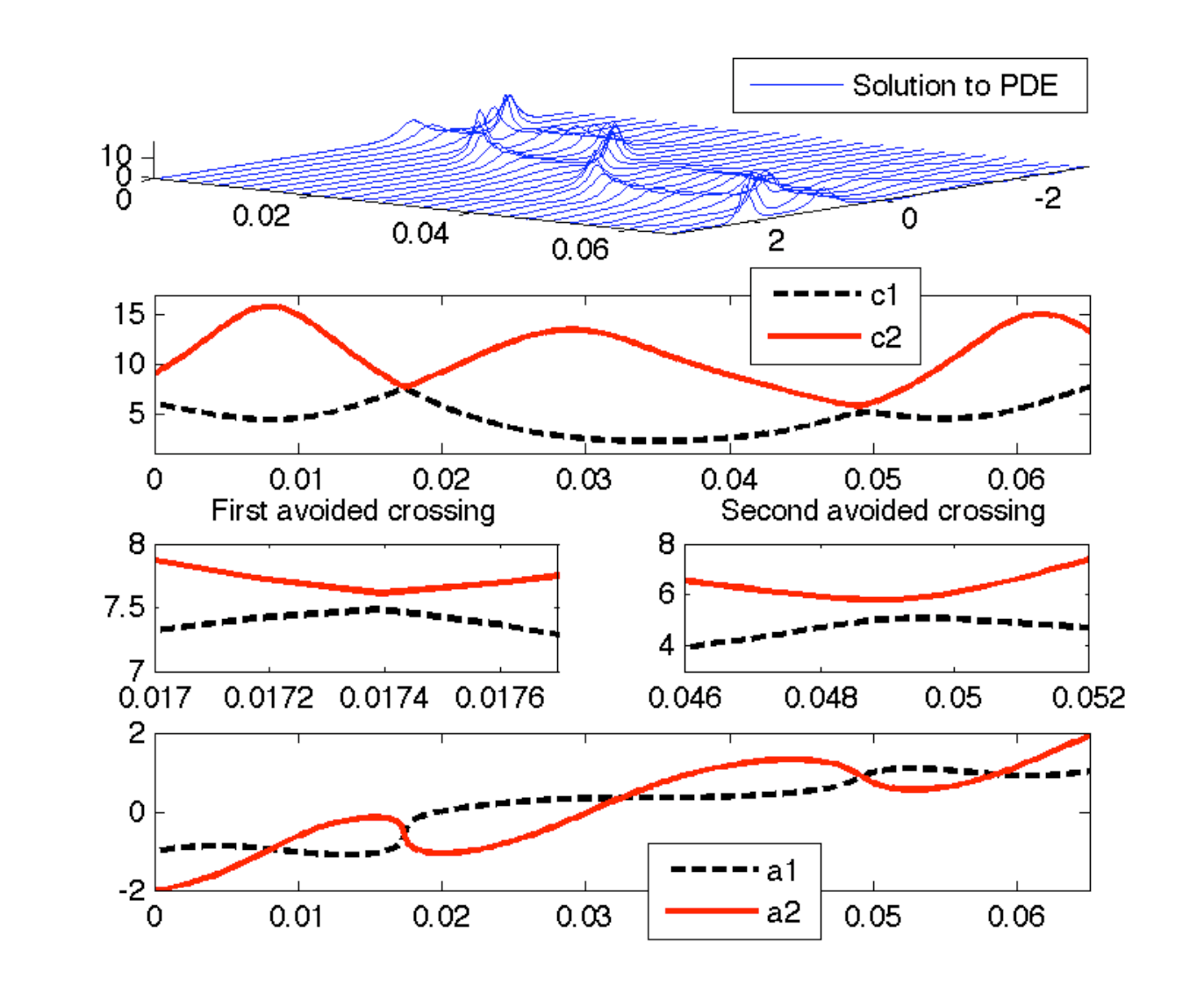}
\caption{The plots of $ c $ and $ a $ for the external 
potential given by the last $ b ( x, t ) $ in \eqref{eq:listex},
and $ \bar c = (6,10) $, $ \bar a = ( -1, -2 )$.
We see the avoided crossings near times at which the
decoupled dynamics \eqref{eq:ODEsh} would give 
a crossing of $ c_j$'s (see also Fig.\ref{f:A1}). 
The crossings are avoided with $ \exp ( - 1/Ch ) $ width
and $ a_1 = a_2 $ at the crossings.
These cases are not yet covered by our theory.
Of the five crossings of $ a_j $'s in the bottom figure,
three do not involve crossings of $c_j$'s are hence the 
description by effective dynamics there is
covered by our theorem. However, in the absence of 
avoided crossing of $ c_j$'s the solitons can interact only 
once.}
\label{f:cros} 
\end{figure}

\subsection{Outline of the proof}
\label{oop}
To obtain the effective dynamics we follow a long tradition
(see \cite{FrSi} and references given there) and 
define the {\em modulation parameters}
$$a(t)=(a_1(t),a_2(t)) \,, \ \ c(t)=(c_1(t),c_2(t))\,, $$ 
be demanding that 
$$v(x,t)  = u(x,t) - q(x,a(t),c(t))\,,  \ \ q= q_2 \,, $$
satisfies { symplectic orthogonality} conditions:  
\begin{align*}
&\omega(v,\partial_{a_1}q) = 0 && \omega(v,\partial_{a_2}q)=0 \\
&\omega(v,\partial_{c_1}q) = 0 && \omega(v,\partial_{c_2}q) =0 
\end{align*}
These can be arranged by the implicit function theorem thanks to 
the nondegeneracy of $\omega|_M$.
This makes $q$ the symplectic orthogonal projection of $u$ onto the manifold of solitons $M$.

Since $u=q+v$ and $u$ solves mKdV, we have
\begin{equation}
\label{eq:vv}\partial_t v = \partial_x( \mathcal{L}_{c,a}v - 6qv^2-2v^3 
+  bv) - F_0 \,, 
\end{equation}
where 
\[ {\mathcal L}_{c,a} = - \partial_x^2 - 6 q ( x, a , c)^2 v \,,  \]
and $F_0$ results from the perturbation and $\partial_t$ landing on the parameters:
$$
F_0 = \sum_{j=1}^2 (\dot a_j - c_j^2)\partial_{a_j}q + \sum_{j=1}^2 \dot c_j\partial_{c_j}q - \partial_x(bq)  \,. 
$$
We decompose $ F_0=F_\| +F_\perp $, where  
$F_\|$ is symplectic projection of $F_0$ onto $T_q M$,  
and $F_\perp$ is the symplectic projection onto its symplectic
orthogonal $(T_qM) ^\perp$.
As seen in \eqref{E:F-par}, $ F_\| \equiv 0 $ is equivalent to 
the equations of motion \eqref{E:eom} (we assume in the proof 
that $ c_2 > c_1 > 0 $).

Using the properties of $ q$, we show that $F_\perp$ is $ {\mathcal O}(h^2)$. 
In fact it is important to 
obtain a specific form for the ${\mathcal O} (h^2)$ term 
so that it is amenable to finding a certain correction term later
-- see \S \ref{S:F-perp}.

The estimates for $ F_\| $ are obtained using 
the symplectic orthogonality properties of $v$.  For example,
$0=\la v, \partial_x^{-1}\partial_{a_j}q\ra$
implies  
\[ 0= \partial_t \la v, \partial_x^{-1}\partial_{a_j}q\ra = \la 
\underbrace{\partial_t v }_{\substack{\uparrow\\\text{substitute equation 
\eqref{eq:vv}} }}
, \partial_x^{-1}\partial_{a_j}q\ra + \la v, \partial_t \partial_x^{-1}\partial_{a_j}q\ra \,,\]
which can be used to show that 
\begin{equation} 
\label{eq:Fp} |F_\| | \leq Ch^2\|v\|_{H^2} + \|v\|_{H^2}^2 \,,\ 
\end{equation}
see \S \ref{S:parameters}.  

The next step is to estimate $v$ satisfying \eqref{eq:vv} with $ v(0 ) = 
{\mathcal O} ( h^2) $ (in the theorem $ v( 0 ) = 0 $, but we need this
relaxed assumption for the bootstrap argument).
We want to show that on a time interval of length $h^{-1}$, 
that $v$ at most doubles.  
The Lyapunov functional $\mathcal{E}(t)$ that we use to achieve this 
comes from the variational characterization of the double soliton
(see \cite[\S 2]{Lax} and Lemma \ref{l:E7} below): if
$$H_c(u) = I_5(u) + (c_1^2+c_2^2)I_3(u) + c_1^2c_2^2I_1(u) \,, $$
then
$$H_c'(q(\cdot, a,c)) = 0 \,, \qquad \forall \; a\in \mathbb{R}^2\,, $$
and 
$$H_c''(q(\cdot,a,c)) = \mathcal{K}_{c,a} \,, $$ 
where $ {\mathcal K}_{c,a} $ is a fourth order operator 
given in \eqref{E:8} below. 
Hence 
$$\mathcal{E}(t) \defeq  H_{c(t)}(q( \bullet , a (t) , c ( t ) )+v( t ) ) - 
H_{c(t)}(q( \bullet , a (t) , c ( t ) ) ) \,, $$
satisfies 
$$\mathcal{E}(t) \approx \la \mathcal{K}_{c,a}v,v\ra \,,$$ 
and, as in Maddocks-Sachs \cite{MS} for KdV, 
$\mathcal{K}_{c,a}$ has a two dimensional kernel and one negative eigenvalue.  
However, the symplectic orthogonality conditions on $v$ imply that we project far enough away from these eigenspaces and hence we have the coercivity
$$\delta \|v\|_{H^2}^2 \leq \mathcal{E}(t) \,. $$

To get the upper bound on $\mathcal{E}(t)$, we compute 
 \[ \frac{d} {dt} {\mathcal E} ( t ) = 
{\mathcal O} (h) \| v ( t) \|_{H^2}^2 + 
\la \mathcal{K}_{c,a}v, \; F_\| \ra
+ \la \mathcal{K}_{c,a}v, \; F_\perp \ra \,,
\]
see \S \ref{S:energy}.
Using \eqref{eq:Fp} we can estimate the second term on the right-hand side
but $|F_\perp| = {\mathcal O} ( h^2) $ only. 
We improve this to $h^3$ using a { correction term to $v$} -- 
see \S \ref{S:correction}, and the comment at the end of this section.
 
All of this combined gives, on $[0,T]$,
\begin{gather*}
\|v\|_{H^2}^2 \lesssim \|v(0)\|_{H^2}^2 + T (|F_\| |\|v\|_{H^2}+ h^2\|v\|_{H^2}+\|v\|_{H^2}^2)\,, \\ 
|F_\| |  \leq Ch^2\|v\|_{H^2} + \|v\|_{H^2}^2 \,,
\end{gather*}
which implies
$$\|v\|_{H^2} \lesssim h^2\,,\qquad |F_ \| |\lesssim h^4\,,  \qquad \text{on } \ 
[0,h^{-1}] \,. $$
Iterating the argument $\delta\log(1/h) $ times
gives a slightly weaker bound for longer times.
The $O(h^4)$ errors in the ODEs can be removed without affecting the bound 
on $v$, proving the theorem. 
 
In the proofs various facts due to complete integrability 
(such as the miraculous Lemma \ref{L:near-conserved}) simplify
the arguments, in particular in the above energy estimate.

We conclude with the remark about the correction term 
added to $v$ in order 
to improve the bound on $\|F_\perp \|$ from $h^2$ to $h^3$.
A similar correction term was used in \cite{HZ2} for NLS $1$-solitons. 
Together with the symplectic projection interpretation, it was the 
key to sharpening the results in earlier works.
Implementing the same idea in the setting of $2$-solitons  is more subtle.  
The $2$-soliton is treated as if it were the sum of two decoupled $1$-solitons, 
the corrections are introduced for each piece, and the result is that 
$F_\perp$ is corrected so that
$$\|F_\perp\|_{H^2} \lesssim h^3 + h^2e^{-\gamma|a_1-a_2|}$$
That is, when $|a_1-a_2| = O(1)$, there is no improvement. 
However, this happens only on an $O(1)$ time scale and hence does
not spoil the long time estimate.

\subsection{Numerical experiments}
\label{nume}

Unlike NLS, KdV is a very friendly equation from the
numerical point of view and {\tt MATLAB} is sufficient for producing
good results. 

We first describe the simple codes on which our experiments
are based. Instead of considering \eqref{E:pmKdV} on the line,
we consider it on the circle identified with $ 
[ - \pi , \pi ) $. To solve it numerically we adapt
the code given in \cite[Chapter 10]{Tr}
which is based on the Fast Fourier Transform in $ x $, the method
of integrating factor for the $ - u_{xxx} \mapsto - i k^3 \hat u ( k ) $
term, and the fourth-order Runge-Kutta formula for the resulting
ODE in time. Unless the amplitude of the solution gets large
(which results in large terms in the equation due to the $ u^3 $
term) it suffices to take $ 2^N $, $ N = 8 $, discretization points
in $ x $.

For $ X \in [ -\pi, \pi ) $ we consider $ B ( X , T) $
periodic in $ X $, and compute $ U ( X , T ) $ satisfying
\[
\partial_T U= - \partial_X ( \partial_X^2 U + 2U^3 - B(X,T) U ) \,, \ \ 
U ( \pi , T ) = U ( -\pi, T ) \,.
\]
A simple rescaling, 
\[  u ( x , t ) = \alpha U ( \alpha x , \alpha^3 t ) \,,  \ \ 
b ( x , t ) = \alpha^2 B ( \alpha x , \alpha^3 t)  \,, \]
gives a solution of  \eqref{E:pmKdV} 
on $ [ - \pi/\alpha , \pi/ \alpha  ] $ with 
periodic boundary conditions. When $ \alpha $ is small this is a
good approximation of the equation on the line. If we use $ U ( X , T ) $
in our numerical calculations with the initial data $ q_2 ( X , A , C ) $,
$ A \in \RR^2 $, $ C \in \RR^2 \setminus {\mathcal C} $, the initial 
condition on for $ u ( x , t ) $ is given by 
\[ u ( x, 0 ) = q_2 ( x , A/\alpha , \alpha C ) \,. \]
If we want $ \bar c = \alpha C $ to  satisfy
the assumptions \eqref{eq:Tin}, the effective small constant 
$ h $ becomes $ h = \alpha $ and $ b_0 $ in \eqref{E:pmKdV} becomes
\[  b_0 ( x, t ) = h^2 B ( x, h^2 t ) \,. \]

In principle we have three scales: size of $ B $, 
size of $ \partial_x B $, and size of $ \partial_t B $,
which should correspond to three small parameters $ h$.
For simplicity we just use one scale $ h $ in the Theorem.

Figure \ref{f:sol1} shows four examples of evolution and
comparison with effective dynamics computed using 
the {\tt MATLAB} codes available at \cite{Codes}. 
The external potentials used are given by 
\begin{equation}
\label{eq:listex}
\begin{split}
& B ( x, t) = 100 \cos^2(x-10^3t)-50\sin(2x+10^3t)\,,
\\
& B ( x , t) = 100 \cos^2(x-10^3t)+50\sin(2x+10^3t)\,, \\
& B ( x, t ) = 
60 \cos^2(x+1-10^2t)+ 40 \sin (2x+2+10^2t)\,, 
\\
& 
B ( x, t ) = 40 \cos(2x+3-10^2t)+30 \sin(x+1+10^2t)\,.
\end{split}
\end{equation}
The rescaling the fixed size potential used in the
theorem, $ b_0 ( x, t ) = h^2 B ( x , h^2 t ) $, means
that our $ h$ satisfies $ h \simeq 1/5  $ in the last two 
examples. In the first two examples the scales in $ x $
are different than the ones in $ t$: the potential is 
not slowly varying in $t$ if $ h \simeq 1/10 $. The agreement
with the main theorem is very good in all cases. However,
the theorem in the current version does not apply to the two bottom
figures since the condition in \eqref{eq:ODE} is not satisfied
for the full time of the experiment. See also Fig. \ref{f:cros}  
and Appendix \ref{A:ode}.

We have not exploited numerical experiments in a fully
systematic way but the following conclusions can be deduced:

\begin{itemize}

\item For the case covered by our theorem the agreement with 
the numerical solution is remarkably close; the same thing
is true for times longer than $ T_0/h $, with $ T_0 $ defined
by \eqref{eq:ODE} despite the crossings of $ C_j$'s (resulting
in the avoided crossing of $ c_j$'s)
The agreement is weaker but the experiments involve only 
relatively large value of $ h$.

\item The soliton profile persists for long times but
we see a deviation from the effective dynamics. This suggest
the optimality of the bound $ \log(1/h)/ h $ in
\eqref{eq:Th}. 

\item The slow variation in $t$ required in the theorem 
can probably be relaxed. For instance, in the top plots
in Fig.\ref{f:sol1} $ \max|\partial_t b_0|/\max|\partial_x b_0| 
\sim 10 $, while the agreement with the effective dynamics
is excellent. For longer times it does break down as 
can be seen using the {\tt Bmovie.m} code 
presented in \cite[\S 3]{Codes}. An indication that
slow variation in time might be removable also comes
from \cite{AW}.

\item When the decoupled equations \eqref{eq:ODEsh}
predict crossing of $ C_j$'s, we observe an avoided 
crossing of $ c_j$'s -- see Fig.\ref{f:cros} and Fig.\ref{f:A1} --
with exponentially small width, $ \exp ( -1/Ch) $. 
At such times we also see the crossing of $ a_j$'s, 
though it really corresponds to solitons changing their
scale constants -- see Fig.\ref{f:A2}.
To have multiple interactions of a pair of solitons, this
type of crossing has to occur, and it needs to be investigated
further. 

  
\end{itemize}

\subsection{Acknowledgments} 
The authors gratefully acknowledge the following sources of 
funding: J.H. was supported in part by a Sloan fellowship
and the NSF grant DMS-0901582, G.P's visit to Berkeley in November 
of 2008 was supported in part by the France-Berkeley Fund, and 
M.Z. was supported in part by the NSF grant DMS-0654436.

\section{Hamiltonian structure and conserved quantities}
\label{S:Hamiltonian}

The symplectic form, at first defined on $ {\mathcal S} ( \RR; \RR) $
is given by 
\begin{equation}
\label{eq:sympl}  \omega(u,v) \defeq 
\la u, \partial_x^{-1}v\ra \,,  \ \ 
\la f,g\ra = \int fg \,, 
\end{equation}
where 
$$\partial^{-1}f(x) \defeq 
\frac12 \left(\int_{-\infty}^x - \int_x^{+\infty}\right) f(y) \, dy$$
  Then the mKdV (equation \eqref{E:pmKdV} with $ b \equiv 0 $) 
is the Hamiltonian flow 
$\partial_t u = \partial_x H_0'(u)$ and \eqref{E:pmKdV} is the Hamiltonian flow $\partial_t u = \partial_x H_b'(u)$, where
$$H_0 = \frac12 \int (u_x^2 - u^4) \, \qquad H_b = \frac12 \int (u_x^2 - u^4 + bu^2)$$

Solutions to mKdV have infinitely many conserved
integrals and the first four are given by 
\[
\begin{split}
I_0(u) & = \int u \, dx \,,  \\
I_1(u)  & = \int u^2 \, dx \,, \\
I_3(u) &  = \int (u_x^2 - u^4) \, dx \,,\\
I_5(u) & =  \int (u_{xx}^2 - 10 u_x^2u^2 + 2u^6) \, dx \,,
\end{split}
\]
which are the mass, momentum, energy, and second energy, respectively. 
In this paper we will only use these particular conserved quantities.

We write $I_j(u) = \int A_j(u)$,  which means that 
$A_j(u)$ denotes the $j$-th Hamiltonian \emph{density}.  

For future reference, we record the expressions appearing in the
Taylor expansions of these densities,
\begin{equation}
\label{E:A-def}
A_j(q+v) = A_j(q) + A_j'(q)(v) + \frac12 A''(q)(v,v) + {\mathcal O}(v^3)\,, 
\end{equation}
\begin{align*}
&A_1'(q)(v) = 2qv \,, \\
&A_3'(q)(v) = 2q_xv_x - 4q^3v \,,  \\
&A_5'(q)(v) = 2q_{xx}v_{xx} - 20 q_xq^2v_x - 20 q_x^2qv + 12q^5v \,, 
\end{align*}
and
\begin{align*}
&A_1''(q)(v,v) = 2v^2 \,, \\
&A_3''(q)(v,v) = 2v_x^2 - 12q^2v^2 \,,  \\
&A_5''(q)(v,v) = 2v_{xx}^2 - 20 q^2v_x^2 -20 q_x^2v^2 - 80qq_xvv_x + 60q^4v^2 \,. 
\end{align*}
The differentials, $I_j'(q)$, are identified with functions
by writing:
$$\la I_j'(q),v\ra = \int A_j'(q)(v) \,.$$
It is useful to record a formal expression for $ I_j'(q)$'s valid
when $ A_j ( q) $'s are polynomials in $ \partial_x^\ell q $:
\begin{equation}
\label{eq:formal}
I'_j ( q ) = \sum_{ \ell \geq 0 } (- \partial_x )^\ell \frac{ \partial
A_j ( q ) } {\partial q_x^{(\ell)} } \,, \ \ 
q_x ^{(\ell)} = \partial_x^\ell q \,. 
\end{equation}
The Hessians, $I_j''(q)$, are the (self-adjoint) operators given 
by 
$$\la I_j''(q)v,v\ra = \int A_j''(q)(v,v) \,. $$

One way to generate the mKdV energies is as follows (see Olver \cite{Olver}).
Let us put
$$\Lambda(u) = -\partial_x^2 - 4u^2 - 4u_x \partial_x^{-1} u \,,$$
and recall that 
$\Lambda(u) \partial_x$ is skew-adjoint:
\[ \begin{split} \Lambda(u) \partial_x & = 
-\partial_x^3 - 4u^2 \partial_x - 4u_x \partial_x^{-1} u 
\partial_x \\
& = -\partial_x^3 - 4u^2 \partial_x - 4 u_x u 
+ 4u_x \partial_x^{-1} u_x \,, 
\end{split} \]
where we used the formal integration by parts $ \partial^{-1}_x 
( u f_x) = - \partial^{-1}_x ( u_x f ) + u f $.

With this notation we have the fundamental recursive identity:
\begin{equation}
\label{E:hierarchy}
\partial_x I_{2k+1}'(u) = \Lambda(u) \partial_x I_{2k-1}'(u) \,, 
\end{equation}
which together with skew-adjointness of $ \Lambda ( u ) \partial_x $
shows that 
\[ \la I_j'(u), \partial_xI_k'(u)\ra = 
\la I_{j-2}'(u) , \partial_x I_{k+2}' ( u ) \ra \,, \]
for $ j $ and  $ k $ odd (if we use 
\eqref{E:hierarchy} with $ m $ even the choice $ I_{2m} ( u ) = 0 $, for $ m > 0 $ is consistent). 
By iteration this shows that 
\begin{equation}
\label{E:conserved}
\la I_j'(u), \partial_xI_k'(u)\ra =0 \,, \ \  \forall \, j \,, k \,. 
\end{equation}
In fact, since $ j $ and $ k $ are odd 
we can iterate all the way down to $ j = 1$ and apply \eqref{eq:formal}:
\[ \begin{split} \la  I_1' ( u ) , \partial_x I_{k + j -1 }' ( u ) \ra & = 
- \la \partial_x u_x^{(\ell)}  ,  \sum_{ \ell \geq 0 } 
\partial A_{j+k-1} ( u )  / {\partial u_x^{(\ell)} }  \ra \\
& = - \int \partial_x ( A_{j+k-1} ( u ) ) dx = 0 \
\,. 
\end{split} \]

If $u$ solves mKdV, then $\partial_t u = \frac12 \partial_x I_3'(u)$ and hence by \eqref{E:conserved} we obtain
$$\partial_t I_j(u) = \la I_j'(u), \partial_t u \ra = \frac12 
\la I_j'(u), \partial_x I_3'(u)\ra = 0 \,.$$

The following identities related to the conservation laws will be needed in \S \ref{S:energy}.  Recalling the definition \eqref{E:A-def} of $A_j$, we have:
\begin{lemma}
\label{L:near-conserved}
For any function $u \in {\mathcal S} $, and for $ b \in C^\infty \cap 
{\mathcal S}' $, we have
\begin{align*}
&\la I_1'(u), (bu)_x \ra = \la b_x, A_1(u) \ra \\
&\la I_3'(u), (bu)_x \ra = 3\la b_x, A_3(u) \ra - \la b_{xxx}, A_1(u) \ra \\
&\la I_5'(u), (bu)_x \ra = 5 \la b_x, A_5(u) \ra - 5\la b_{xxx}, A_3(u) \ra 
+\la b_{xxxxx}, A_1(u) \ra 
\end{align*}
\end{lemma}
\begin{proof}
By taking arbitrary  $b \in {\mathcal S} $, we see
 that the claimed formulae are equivalent to
\begin{align*}
& u \partial_x I_1'(u) = \partial_x A_1(u) \,, \\
& u \partial_x I_3'(u) = 3\partial_x A_3(u) - \partial_x^3 A_1(u) \,, \\
& u \partial_x I_5'(u) = 5\partial_x A_5(u) - 5\partial_x^3 A_3(u) + 
\partial_x^5 A_1(u)\,,
\end{align*}
and these can be checked by direct computation.
\end{proof}
\begin{lemma}  
\label{L:near-conserved2}
For any function $u , q \in {\mathcal S} $, and for $ b \in C^\infty \cap 
{\mathcal S}' $, we have
\begin{align*}
 \la I_1''(q)v, (bq)_x\ra - \la \partial_x I_1'(q), bv \ra &= \la b_x, A_1'(q)(v) \ra \\
 \la I_3''(q)v, (bq)_x\ra - \la \partial_x I_3'(q), bv \ra &= 3\la b_x, A_3'(q)(v) \ra - \la b_{xxx}, A_1'(q)(v) \ra \\
 \la I_5''(q)v, (bq)_x\ra - \la \partial_xI_5'(q), bv \ra &= 5 \la b_x, A_5'(q)(v) \ra - 5\la b_{xxx}, A_3'(q)(v) \ra  \\
& \qquad + \la b_{xxxxx}, A_1'(q)(v) \ra
\end{align*}
\end{lemma}
\begin{proof}
Differentiate the formul{\ae} in Lemma \ref{L:near-conserved} with 
respect to $ u $ at $ q $ in the direction of $ v$.
\end{proof}

\section{Double soliton profile and properties}
\label{S:q-properties}

Here we record some properties of mKdV and its double soliton solutions. 
The parametrization of the family of double solitons follows 
the presentation for NLS in Faddeev--Takhtajan \cite{FT}.

The double-soliton is defined in terms of the profile $q(x,a,c)$,
where 
\begin{gather}
\label{eq:ac}
\begin{gathered}
 a=(a_1,a_2)\in \mathbb{R}^2 \,, \ \ 
c=(c_1,c_2) \in \mathbb{R}^2 \setminus {\mathcal C} \,, \\ 
{\mathcal C} \defeq \{ (c_1, c_2)  \; : \; c_1 = \pm c_2 \} 
 \cup \RR \times \{ 0 \} \cup \{ 0 \} \times \RR 
\,. 
\end{gathered}
\end{gather}

The profile $q = q_2 $ (from now on we drop the subscript $ 2 $) 
is defined by
\begin{equation}
\label{E:q-def}
q(x,a,c) = \frac{\det M_1}{\det M}
\end{equation}
where
$$M = [ M_{ij} ]_{1\leq i,j \leq 2} \,,  \qquad
M_{ij} = \frac{1+\gamma_i\gamma_j}{c_i+c_j} 
\,, \qquad 
M_1=
\left[
\begin{array}{c|c}
M & \begin{array}{c} \gamma_1 \\ \gamma_2 \end{array} \\
\hline
\begin{array}{cc} 1 & 1 \end{array} & 0
\end{array}
\right]
$$
and 
$$\gamma_j = (-1)^{j-1}  \exp({-c_j(x-a_j)}), \quad j=1,2\,.$$
For conveninece 
we will consider the 
\[ 0 < c_1 < c_2 \]  
connected component of 
$ \RR^2 \setminus {\mathcal C} $ throughout the paper.
Since 
\[
\begin{split} 
& q ( x, a_1, a_2 , c_1, c_2 ) = - q ( x , a_2 , a_1, c_2, c_1 ) \,, \\ 
& q ( x, a_1, a_2 , - c_1 , -c_2 ) = - q ( -x , - a_1 , -a_2,  c_1, c_2) \,,
\end{split}
\]
the only other component to consider would be, say, $ 0 < - c_1 < c_2 $
(see Fig.\ref{f:col}), and the analysis is similar.

We should however mention that in numerical experiments 
it is more useful to introduce a phase parameter $ \epsilon = 
( \epsilon_1 , \epsilon_2 ) $, $ \epsilon_j = \pm 1 $, and define
$ \tilde q ( x , a , c , \epsilon ) $ by \eqref{E:q-def} but with 
$ \gamma_j$'s replaced by 
$$\tilde \gamma_j = (-1)^{j-1} \epsilon_j \exp({-c_j(x-a_j)}) , \quad j=1,2\,. $$
We can then check that 
\[ \tilde q ( x, a, c , \epsilon ) = q ( x , a , ( \epsilon_1 c_1 , 
\epsilon_2 c_2)  ) \,, \]
but $ \tilde q $ seems more stable in numerical calculations.

\begin{figure}
\includegraphics[width=6in]{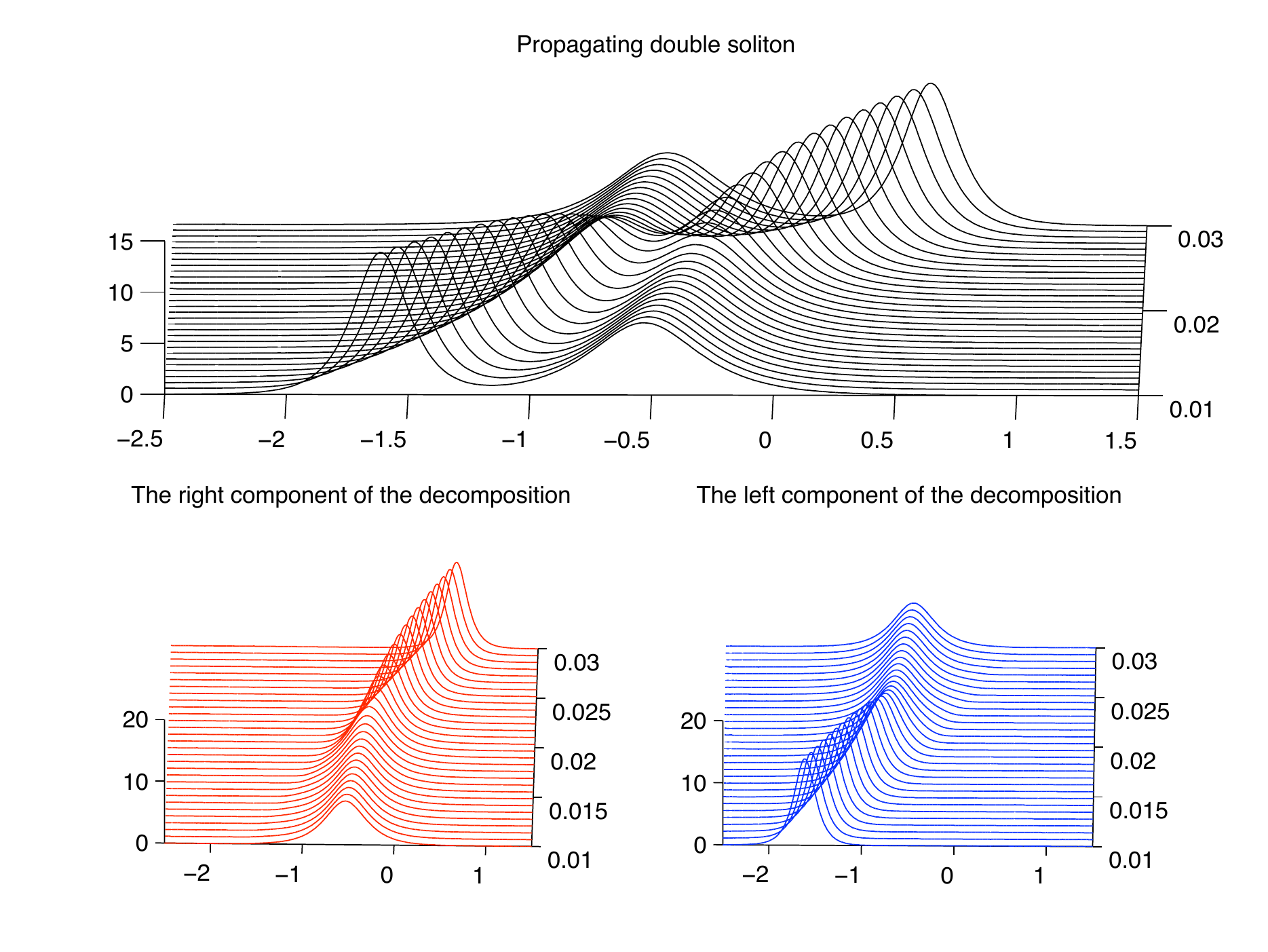}
\caption{A depiction of the double soliton solution given by \eqref{E:ds}.
The top figure shows the evolution of a double soliton.
The bottom two figures show the evolution of its two components 
defined using \eqref{eq:Qdec}. One possible
``particle-like'' interpretation of the two soliton interaction \cite{Kas}
is that the slower soliton, shown in the left bottom plot is hit
by the fast soliton shown in the right bottom plot. Just like 
billiard balls, the slower one picks up speed, and the fast one
slows down. But unlike billiard balls, 
the solitons simply switch velocities.}
\label{F:2sol-interact}
\end{figure}

The corresponding \emph{double-soliton}
\begin{equation}
\label{E:ds}
u(x,t) = q(x, a_1+c_1^2t,a_2+c_2^2t,c_1,c_2)
\end{equation}
is an exact solution to {mKdV}. For the double soliton
this  can be checked by 
an explicit calculation but it is a consequence of the inverse
scattering method.  This is the only place in this paper where we appeal 
directly to the 
inverse scattering method.
Fig.~\ref{F:2sol-interact} illustrates some aspects of
this evolution.

The scaling properties of {mKdV} imply that
\begin{gather}
\label{eq:scaleq}
\begin{gathered} 
 q ( x + t, a + ( t, t ) , c ) = q ( x , a, c) \,, \\
 q ( t x, t a, c/t ) =  q ( x , a , c) /t \,. 
\end{gathered}
\end{gather}
Both properties also follow from the formula for $ q $,
with the second one being slightly less obvious:
\[  \begin{split} q ( t x , t a, c /t ) & = \frac 1 { \det t M } \det
\left[
\begin{array}{c|c}
 t M & \begin{array}{c} \gamma_1 \\ \gamma_2 \end{array} \\
\hline
\begin{array}{cc} 1 & 1 \end{array} & 0 
\end{array}
\right]  \\
& = \frac 1 { \det t M } \det \left( \left[ \begin{array}{c|c} \begin{array}{ll} t & 0 \\
0 & t \end{array}  & \begin{array}{c} 0 \\ 0 \end{array} \\
\hline
\begin{array}{cc} 0 & 0 \end{array} & 1 
\end{array} \right] M_1
\left[ \begin{array}{c|c}   \begin{array}{ll} 1 & 0 \\
0 & 1 \end{array}  & \begin{array}{c} 0 \\ 0 \end{array} \\
\hline
\begin{array}{cc} 0 & 0 \end{array} & 1/t  
\end{array} \right] \right) \\
& =  q ( x , a, c ) /t \,. 
\end{split} 
\]

Now we discuss in more detail the properties of the profile $q$.  
Recalling that we suppose that $c_2 >c_1 > 0 $, let
\begin{equation}
\label{E:6}
\alpha_1
\defeq 
\frac{1}{c_1} \log \left( \frac{c_1+c_2}{c_2-c_1} \right) \,, 
\qquad \alpha_2
\defeq 
\frac{1}{c_2} \log \left( \frac{c_2-c_1}{c_1+c_2} \right) \,,
\end{equation}
noting that for $ c_2 > c_1 > 0 $, $ \alpha_1 > 0 $ and $ \alpha_2 < 0 $.
Fix a smooth function, $\theta \in C^\infty (\RR , [ 0 , 1 ] ) $, such that 
\begin{equation}
\label{eq:the} 
\theta (s) = \left\{ \begin{array}{ll}  \  \ 1  & \text{for } \ s\leq -1 \,, \\
- 1 & \text{for} \ \ s\geq 1 \,. \end{array}
\right. 
\end{equation}
   Define the shifted positions as
\begin{equation}
\label{eq:hata}
\hat a_j \defeq a_j + \alpha_j  \theta(a_2-a_1) 
\end{equation}
that is,
\[ \hat a_j = \left\{ \begin{array}{ll} a_j + \alpha_j  \,, & 
a_2 \ll a_1  \,,\\
a_j - \alpha_j \,, & a_2 \gg a_1 \,. \end{array} \right. 
\]
see Fig.~\ref{F:shift}.
We note that $\hat a_j= \hat a_j(a_j,c_1,c_2)$.

\begin{figure}
\includegraphics[width=6in]{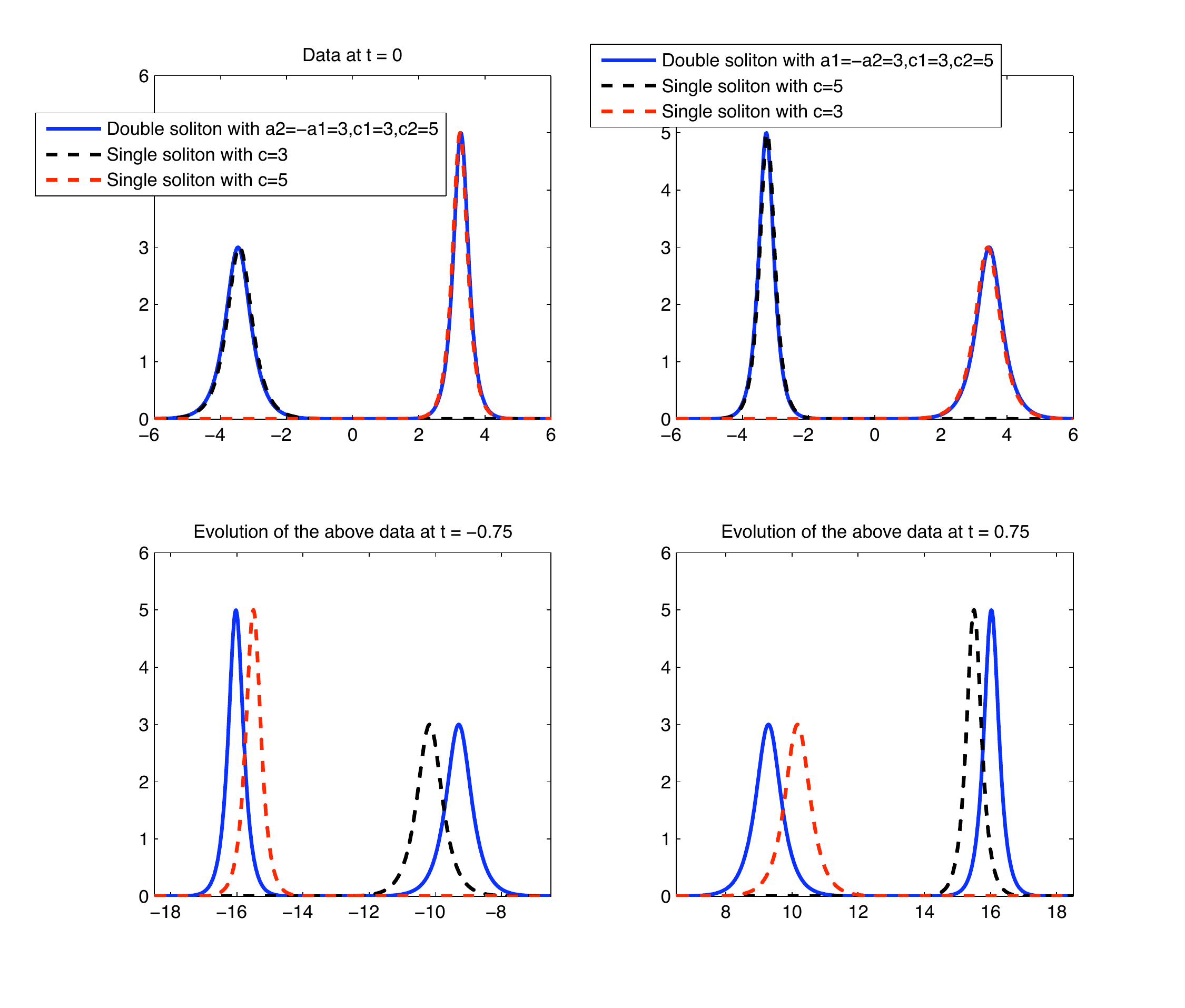}
\caption{The top plots show show $ q ( x , 3,5 , \mp 3 , \pm 3 ) $,
the corresponding $ \eta ( x , \hat a_j , c_j ) $ given by 
Lemma \ref{L:q-asymp}. The bottom plots show the post-interaction 
pictures at times $ \mp 0.75$. Since the sign of $ a_2 - a_1 $
changes after the interaction we see the shift compared to 
the evotion of $ \eta ( x , \hat a_j , c_j )$'s.}
\label{F:shift}
\end{figure}

\newcommand{\Ssol}{\mathcal{S}_\textnormal{sol}}
\newcommand{\Serr}{\mathcal{S}_\textnormal{err}}

Let $\mathcal{S}$ denote the Schwartz space.  We will next introduce function classes $\Ssol$ and $\Serr$, and then show that $q\in \Ssol$ and give an approximate expression for $q$ with error in $\Serr$.  

\begin{definition}
Let $\Serr$ denote the class of functions,  $ \varphi = \varphi(x,a,c)$, 
$ x \in \RR $, $ a \in \RR^2 $,  $ 0 < \delta < c_1 < c_2 - \delta < 
1/ \delta $ (for any fixed $ \delta $) satisfying
\[  \left|  \partial_x^\ell \partial_c^k \partial_a^{p} \varphi 
 \right|  \leq 
C_2 \exp ( - (|x-a_1|+|x-a_2|)/C_1 )  \,, \]
where $ C_j $ depend on $ \delta $, $ \ell$, $ k $, and $ p $ only.

Let $\Ssol$ denote the  class of functions of $(x,a,c)$ of the form
$$p_1(c_1,c_2) \varphi_1(c_1(x-\hat a_1)) + p_2(c_1,c_2) \varphi_2(c_2(x-\hat a_2))  + \varphi ( x, a , c ) $$
where
\begin{enumerate}
\item $ |\partial_k^\ell \varphi_j(k) | \leq C_\ell \exp ( - |k|/ C ) $, for some $ C$, 
\item $p_j\in C^\infty ( \RR^2 \setminus {\mathcal C} ) $.
\item $ \varphi \in \Serr $.
\end{enumerate}
\end{definition}

Some elementary properties of $\Ssol$ and $\Serr$ are given in 
the following.

\begin{lemma}[properties of $\Serr$]
\label{L:S-prop}
\quad
\begin{enumerate}
\item $\partial_x \Serr \subset \Serr$, $\partial_{a_j} \Serr \subset \Serr$, $\partial_{c_j} \Serr \subset \Serr$.  
\item $(x-a_j)\Serr\subset \Serr$ and $(x-\hat a_j)\Serr\subset \Serr$ \,.
\item 
\label{I:inverse-S}
If $f\in \Serr$ and $\int_{-\infty}^{+\infty}f =0$, 
then $\partial_x^{-1}f \in \Serr$. 
\end{enumerate}
\end{lemma}

The class $ \Serr $ allows to formulate the following


\begin{lemma}[asymptotics for $q$]
\label{L:q-asymp}
Suppose that $  0 < c_1 < c_2 < c_1/\epsilon < 1/\epsilon^2 $, 
for $ \epsilon > 0 $. 
Then for $ |a_2 - a_1 | \geq C_0 / (c_1 + c_2 ) $, 
\begin{equation}
\label{eq:decoq}
 \left|  \partial_x^\ell \partial_c^k \partial_a^{p} \left( 
q(x,a,c) - \sum_{j=1}^2 \eta(x , \hat a_j, c_j) \right) \right|  \leq 
C_2 \exp ( - (|x-a_1|+|x-a_2|)/C_1 )  \,, 
\end{equation}
where $ C_2 $ depends on $ k, \ell , p $ and $ \epsilon $, and $ C_0 $,
$ C_1 $ on $ \epsilon $ only.
In other words,
$$ q(x,a,c) - \sum_{j=1}^2 \eta(x , \hat a_j, c_j) \in \Serr \,. $$

\end{lemma}


\begin{corollary}
\label{C:q-inverse}
$\partial_x^{-1} \partial_{a_j}q$, $\partial_x^{-1} \partial_{c_j}q \in \Ssol$.
\end{corollary}
\begin{proof}
By Lemma \ref{L:q-asymp}, we have
$$\partial_{c_j}q = \partial_{c_j} \sum_{j=1}^2 \eta(\cdot, \hat a_j, c_j) + f$$
where $f\in \Serr$.  By direct computation with the $\eta$ terms, we find that 
$$\int_{-\infty}^{+\infty} \partial_{c_j} \sum_{j=1}^2 \eta(\cdot, \hat a_j, c_j) = 0 \,.$$
By the remark in Lemma \ref{L:Iq}, we have $\int_{-\infty}^{+\infty} \partial_{c_j}q =0$.  Hence $\int_{-\infty}^{+\infty} f =0$. By Lemma \ref{L:S-prop}\eqref{I:inverse-S}, we have $\partial_x^{-1}f \in \Serr$.  Hence
$$\partial_x^{-1} \partial_{c_j}q = \partial_x^{-1}\partial_{c_j} \sum_{j=1}^2 \eta(\cdot, \hat a_j, c_j) + \Serr$$
and the right side is clearly in $\Ssol$.
\end{proof}



\begin{proof}[Proof of Lemma \ref{L:q-asymp}]
We define 
\begin{equation}
\label{eq:defQ} 
Q ( x, \alpha , \delta ) \defeq q ( x, -\alpha , \alpha
 , 1 - \delta , 1 + \delta ) \,,
\end{equation}
so that, using \eqref{eq:scaleq},
\begin{gather}
\begin{gathered}
\label{eq:q2Q}
q ( x , a_1 , a_2 , c_1 , c_2 ) = 
\frac{ c_1 + c_2 } 2 Q  \left( \left(
\frac {c_1 + c
_2} 2  \right)  \left( x  - \frac{ a_1 + a_2} 2 \right)
, \alpha, \delta \right) \,, \\ 
\alpha =  \left(\frac {c_1 + c_2 } 2 \right) 
\left( \frac {a_2 - a_1 } 2 \right) \,, \ \ 
\delta =  \frac {c_2 - c_1 } { c_2 + c_1 }  \,. 
\end{gathered}
\end{gather}
Hence it
is enough to study the more symmetric expression \eqref{eq:defQ}.
We decompose it in the same spirit as the decomposition 
of double solitons for KdV  was performed in \cite{Kas}:
\begin{equation}
\label{eq:Qdec} Q ( x , \alpha, \delta ) =  \tau ( x , \alpha, \delta )  + 
 \tau ( - x , - \alpha, \delta ) \,, 
\end{equation}
where 
\begin{equation}
\label{eq:deftau}
 \tau ( x, \alpha , \delta ) = \frac12 
\frac{ ( 1 + \delta ) \exp({ ( 1 - \delta ) ( x + \alpha ) } )
+ ( 1 - \delta ) \exp({ ( 1 + \delta ) ( x - \alpha ) } )} 
{ \delta \sech^2({ x - \delta \alpha }) 
+ \delta^{-1} \cosh^2 ({ \delta x - \alpha  })  }
\,. 
\end{equation}
This follows from a straightforward but tedious
calculation which we omit.

Thus, to show \eqref{eq:decoq} we have to show that
\begin{equation}
\label{eq:tauas} 
\begin{split} 
& | \partial_x^\ell \partial_\alpha^p \partial_\delta^k ( 
\tau ( x , \alpha, \delta ) - \eta ( x - | \alpha | - \log(1/\delta) 
/ ( 1 \pm \delta ) , 1 \pm \delta ) ) | \\
& \ \ \ \leq C_2 \exp ( - ( |x| + |
\alpha | )/ C_1 ) \,, \ \ \pm \alpha \gg 1 \,,
\end{split} 
\end{equation}
uniformly for $ 0 <  \delta \leq 1 - \epsilon $.

To see this 
put 
$ \gamma = ( 1 - \delta )/( { 1 + \delta } ) $, and 
multiply 
the numerator and denominator of \eqref{eq:deftau} by 
$ e^{- ( 1 + \delta ) ( x - \alpha ) } $:
\begin{equation}
\label{eq:tau11}  \tau ( x , \alpha, \delta ) = 
 \frac{ 2 (1 - \delta) \left( 1 + \gamma^{-1} 
 e^{ 2 \alpha - 2 \delta x }\right) }
 { \delta e^{ ( 1 - \delta ) ( x + \alpha )} 
( 1 - e^{-2x + 2 \delta \alpha } )^2 + \delta^{-1} e^{-(1- \delta) ( x + \alpha) 
} ( 1 + e^{-2\delta x + 2 \alpha } )^2 }  \,. \end{equation}
Similarly, the multiplication by 
$ e^{- ( 1 + \delta ) ( x - \alpha ) } $ gives
\begin{equation}
\label{eq:tau12} \begin{split} 
\tau ( x , \alpha, \delta )  = & \frac{ 2 (1 + \delta) \left( 1 + 
\gamma e^{ - 2 \alpha + 2 \delta x }  \right) }
 { \delta e^{ ( 1 + \delta ) ( x - \alpha )} 
( 1 - e^{-2x + 2 \delta \alpha } )^2 + \delta^{-1} e^{-(1 +  \delta) ( x - \alpha) }
 ( 1 + e^{-2\delta x - 2 \alpha } )^2 } 
 \\
& = 
\frac{ 2 (1 + \delta) \left( 1 +  \gamma 
 e^{ - 2 \alpha + 2 \delta x } \right) ( 1 + e^{-2\delta x - 2 \alpha } )^{-2 }
} 
 { \delta e^{ ( 1 + \delta ) ( x - \alpha )} 
\left(({ 1 - e^{-2x + 2 \delta \alpha } })/
({  1 + e^{-2\delta x - 2 \alpha }  } ) \right)^2 
+ \delta^{-1} e^{-(1 +  \delta) ( x - \alpha) } } 
\,.
\end{split} 
\end{equation}

This shows that 
for negative values of $ x $, $ \tau $ is negligible:
multiplying the numerator and denominator by $ \delta $ and 
using \eqref{eq:tau11} for $ \alpha \leq 0 $ and \eqref{eq:tau12}
for $ \alpha \geq 0 $, gives
      %
\begin{equation}
\label{eq:tau1} 
\tau ( x , \alpha , \delta ) \leq 
\left\{ \begin{array}{ll} \delta ( 1 + \delta ) ( 1 + 
e^{-2 (|\alpha| +  \delta |x|) } ) e^{ - ( 1 + \delta ) ( |x | + | \alpha|) } \,, 
& \alpha \geq 0 \,, \\
\delta ( 1 + \delta ) 
( 1 + e^{2 \delta |x| - 2 | \alpha|  } )^{-1}  e^{- ( 1 - \delta ) ( |x | + | \alpha|)}\,,  & \alpha \leq 0 \,, \end{array} \right. 
\end{equation}
and in fact this is valid uniformly for $ 0 \leq \delta \leq 1 $. 
Similar estimates hold also for derivatives.

For $ x \geq 0 $,  $ 0 \leq \delta \leq 1 - \epsilon $, and 
for 
$ \alpha  \ll - 1 $, we use \eqref{eq:tau11} to obtain,
\[ \tau( x, \alpha, \delta ) 
= ( 1 - \delta ) \sech \left( ( 1 - \delta ) \left( x - |\alpha| - 
\frac 1 { 1 - \delta } \log \frac 1 \delta \right) \right) + 
\epsilon_- ( x , \alpha, \delta ) \,, \]
and for $ \alpha  \gg  1 $, \eqref{eq:tau12}:
\[ \tau( x, \alpha, \delta ) 
= ( 1 + \delta ) \sech \left( ( 1 + \delta ) \left( x - |\alpha | - 
\frac 1 { 1 + \delta } \log \frac 1 \delta \right) \right) + 
\epsilon_+ ( x , \alpha, \delta ) \,,\]
where 
\[ |\partial_x^k \epsilon_\pm | \leq C_k \exp ( - ( | x| + |\alpha|)/c) 
\,,  \ \ c> 0 \,, \]
uniformly in $ \delta $, $ 0 < \delta < 1 - \epsilon $. 
Inserting the resulting decomposition into \eqref{eq:q2Q} completes the proof.
\end{proof}

\begin{lemma}[fundamental identities for $q$]
\label{L:magic}
With  $q=q(\cdot,a,c)$, we have
\begin{equation}
\label{E:magic2}
 \partial_x I_3'(q) = 2 
\partial_x(- \partial_x^2q-2q^3) =  2 \sum_{j=1}^2 c_j^2 \partial_{a_j}q \,,
\end{equation}
\begin{equation}
\label{E:magic1}
\partial_x I_1'(q) = 2 \partial_x q = -  2 \sum_{j=1}^2\partial_{a_j}q  \,,
\end{equation}
\begin{equation}
\label{E:magic3}
q = \sum_{j=1}^2 (x-a_j)\partial_{a_j}q + \sum_{j=1}^2 c_j\partial_{c_j}q \,.
\end{equation}
\end{lemma}

These three identities are analogues of the following three identities for the single-soliton $\eta=\eta(\cdot,a,c)$, which are fairly easily verified by direct inspection.
$$\partial_x I_1'(\eta) = \partial_x\eta =- \partial_a \eta$$
$$\partial_x I_3'(\eta) = \partial_x(-\partial_x^2\eta-2\eta^3) = c^2\partial_a\eta$$
$$\eta = (x-a)\partial_a \eta + c\partial_c\eta$$

\begin{proof}
The first identity is just the statement that \eqref{E:ds} solves 
mKdV and we take it on faith from the inverse scattering
method (or verify it by a computation). 
To see \eqref{E:magic1} and \eqref{E:magic3} we differentiate
\eqref{eq:scaleq} with respect to $ t $.
\end{proof}

The value of $I_j(q)$ for all $j$ is recorded in the next lemma.

\begin{lemma}[values of $I_j(q)$]
\label{L:Iq}
\begin{equation}
\label{E:2}
I_0(q)= 2\pi
\end{equation}
For $j=1,3,5$, we have
\begin{equation}
\label{E:Iq}
I_j(q) = 2(-1)^{\frac{j-1}{2}} \frac{c_1^j+c_2^j}{j} \,.
\end{equation}
Also,
\begin{equation}
\label{E:com}
\int xq(x,a,c)^2 \, dx = 2a_1c_1+2a_2c_2 \,.
\end{equation}
Note that by \eqref{E:2},
$$\int_{-\infty}^{+\infty} \partial_{a_j}q =0, \quad \int_{-\infty}^{+\infty} \partial_{c_j}q =0, \quad j=1,2\,.$$
from which it follows that $\partial_x^{-1}( \partial_{a_j}q )$ and $\partial_x^{-1}( \partial_{c_j}q )$ are Schwartz class functions.
\end{lemma}
\begin{proof}
We prove \eqref{E:Iq}, \eqref{E:2} by reduction to the $1$-soliton case.  Let 
$u(t) = q(\cdot, a_1+tc_1^2, a_2+tc_2^2, c_1,c_2)$.   Then by the asymptotics in Lemma \ref{L:q-asymp},
$$I_j(q) = I_j(u(0)) = I_j(u(t)) = \sum_{k=1}^2 I_j(\eta(\cdot, (a_k+c_k^2t)\hat{\;} ,c_k)) + \omega(t)$$
where
$$|\omega(t)| \lesssim \la c_2 (( a_1+tc_1^2) - (a_2+tc_2^2)) \ra^{-2}$$
But note that by scaling,
$$I_j ( \eta(\cdot, (a_k+c_k^2t)\hat{\;} ,c_k)) = c_k^j I_j(\eta)$$
By sending $t\to +\infty$, we find that
$$I_j(q) = (c_1^j+c_2^j)I_j(\eta)$$
To compute $I_j(\eta)$, we let $\eta_c(x)=c\eta(cx)$.  By scaling $I_j(\eta_c) = c^j I_j(\eta)$.  Hence 
\begin{align*}
j I_j(\eta) &= \partial_{c}\big|_{c=1} I_j(\eta_c) = \la I_j'(\eta), \partial_{c}\big|_{c=1} \eta_c \ra \\
&= \la I_j'(\eta), (x\eta)_x \ra = 2(-1)^{\frac{j-1}{2}} \la \eta, (x\eta)_x \ra = 2(-1)^{\frac{j-1}{2}} \,,
\end{align*}
where we have used the identity
\begin{equation}
\label{eq:Ijeta} 
I_j'(\eta) = 2(-1)^{\frac{j-1}{2}} \eta\,,
\end{equation}
which follows from the energy hierarchy. In fact, 
$I_1'(\eta) = 2\eta$ is just the definition of $ I'_1 $.
Assuming that $I_j'(\eta) =2(-1)^{\frac{j-1}{2}} \eta$, we compute
\begin{align*}
\partial_x I_{j+2}'(\eta) &= \Lambda(\eta) \partial_x I_j'(\eta) \\
&= 2(-1)(-1)^{\frac{j-1}{2}}(\partial_x^2 +4\eta^2 + 4\eta_x\partial_x^{-1}\eta) \eta_x \\
&= 2(-1)^{\frac{j+1}{2}}\partial_x(\eta_{xx}+2\eta^3) \\
&= 2(-1)^{\frac{j+1}{2}}\partial_x\eta
\end{align*}

We now prove \eqref{E:com}.  By direct computation, if $u(t)$ solves mKdV, then $\partial_t \int xu^2 = - 3I_3(u)$.  Again let $u(t) = q(\cdot, a_1+tc_1^2, a_2+tc_2^2, c_1,c_2)$.  By \eqref{E:Iq} with $j=3$, we have 
$$\int xq(x,a,c)^2 \, dx = \int xu(0,x)^2 \, dx = \int x u(t,x)^2 \,dx  - 2(c_1^3+c_2^3)t$$
By the asymptotics in Lemma \ref{L:q-asymp},
$$\int xu(t,x)^2 = \sum_{j=1}^2 \int x \eta(x,(a_j + tc_j^2)\hat{\;},c_j)^2 + \omega(t)$$
where
$$|\omega(t)| \leq (a_1+tc_1^2) \la c_2( (a_1+c_1^2t) - (a_2+tc_2^2)) \ra^{-2}$$
But
$$ \int x \eta(x,\hat a_j,c_j)^2 = 2c_j\hat a_j$$
Combining, and using that $c_1\hat a_1+c_2\hat a_2 = c_1a_1+c_2a_2$, we obtain
$$\int xq(x,a,c)^2 \, dx =  2(c_1a_1+c_2a_2) + \omega(t)$$
Send $t\to +\infty$ to obtain the result.
\end{proof}

We define the four-dimensional manifold of $2$-solitons $M$ as
$$M = \{ \, q(\cdot, a,c ) \, | \, a=(a_1,a_2)\in \mathbb{R}^2 \,, c=(c_1,c_2)\in (\mathbb{R})^2 \setminus {\mathcal C}  \, \}$$

\begin{lemma}
\label{E:sf}
The symplectic form \eqref{eq:sympl} restricted to the manifold of 
$2$-olitons is given by 
$$\omega|_M = \sum_{j=1}^2 da_j\wedge dc_j \,. $$
In particular, it is nondegenerate and $ M $ is a symplectic manifold.
\end{lemma}
\begin{proof}
By \eqref{E:Iq}  with $j=1$ and \eqref{E:magic1},
\begin{align*}
0 &= \frac 12 \partial_{a_1} I_1(q) = \frac 12 \la I_1'(q), \partial_{a_1}q\ra = \la \partial_{a_1}q, \partial_x^{-1} \partial_{a_1} q \ra + \la \partial_{a_2}q, \partial_x^{-1} \partial_{a_1}q \ra \\
&= \la \partial_{a_2}q, \partial_x^{-1} \partial_{a_1}q\ra 
\end{align*}
Again by \eqref{E:Iq} with $j=1$ and \eqref{E:magic1},
\begin{equation}
\label{E:s1}
1 = \frac 12 \partial_{c_1} I_1(q) = \frac 12 
\la I_1'(q), \partial_{c_1}q\ra = \la \partial_{a_1}q, \partial_x^{-1} \partial_{c_1} q \ra + \la \partial_{a_2}q, \partial_x^{-1} \partial_{c_1}q \ra
\end{equation}
 By \eqref{E:Iq} with $j=3$ and \eqref{E:magic2},
\begin{equation}
\label{E:s2}
-c_1^2 = \frac 12 \partial_{c_1} I_3(q) = \frac 12 \la I_3'(q), \partial_{c_1}q\ra = -c_1^2\la \partial_{a_1}q, \partial_x^{-1} \partial_{c_1} q \ra - c_2^2\la \partial_{a_2}q, \partial_x^{-1} \partial_{c_1}q \ra 
\end{equation}
Solving \eqref{E:s1} and \eqref{E:s2}, we obtain  that $\la \partial_{a_1}q, \partial_x^{-1} \partial_{c_1} q \ra =1$ and $\la  \partial_{a_2}q, \partial_x^{-1} \partial_{c_1} q \ra =0$.  We similarly obtain that  $\la \partial_{a_2}q, \partial_x^{-1} \partial_{c_2} q \ra =1$ and $\la  \partial_{a_1}q, \partial_x^{-1} \partial_{c_2} q \ra =0$.  It remains to show that $\la \partial_{c_1}q, \partial_x^{-1}\partial_{c_2}q \ra =0$:
\begin{align*}
\la \partial_{c_1}q, \partial_x^{-1}\partial_{c_2}q \ra 
&= \frac{1}{c_1} \la \sum_{j=1}^2 c_j \partial_{c_j}q, \partial_x^{-1}\partial_{c_2}q \ra \\
&= \frac{1}{c_1} \la q-\sum_{j=1}^2 (x-a_j)\partial_{a_j}q, \partial_x^{-1}\partial_{c_2}q \ra && \text{by }\eqref{E:magic3} \\
&= \frac{1}{c_1} \la q+xq_x, \partial_x^{-1}\partial_{c_2}q \ra + \frac1{c_1} \sum_{j=1}^2 a_j \la \partial_{a_j}q, \partial_x^{-1}\partial_{c_2}q \ra&& \text{by }\eqref{E:magic1}\\
&= -\frac{1}{2c_1} \partial_{c_2} \int xq^2 + \frac{a_2}{c_1} \\
&= 0 && \text{by }\eqref{E:com}
\end{align*}

\end{proof}

\noindent
{\bf Remark.} If $ | a_1 - a_2 | \gg 2 $, and $ c_1 < c_2 $ then, in the 
notation of \eqref{eq:hata},
\[   \sum_{j=1,2} d a_j \wedge dc_j =   \sum_{j=1,2} d \hat a_j \wedge dc_j \,,\]
that is the map $ ( a , c ) \mapsto ( \hat a , c ) $ is symplectic.


The nondegeneracy of the symplectic form \eqref{eq:sympl} restricted
to the manifold of $2$-olitons, $ M $ 
shows that $H^2$ functions close to $ M $ can be uniquely decomposed
into an element $ q $, of $ M $ and a function symplectically orthogonal 
$ T_q M $. We recall this standard fact in the following

\begin{lemma}[Symplectic orthogonal decomposition]
\label{L:decomp}
Given $\tilde c$, there exist constants $\delta>0$, $C>0$ such that the following holds.  If $u = q(\cdot,\tilde a, \tilde c)+\tilde v$ with $\|\tilde v\|_{H^2} \leq \delta$, then there exist unique $a$, $c$ such that 
$$|a-\tilde a| \leq C\|\tilde v \|_{H^2} \,, \qquad |c-\tilde c| \leq C\| \tilde v\|_{H^2}$$ 
and $v \defeq u-q(\cdot,a,c)$ satisfies
\begin{equation}
\label{E:so}
\la v,\partial_x^{-1}\partial_{a_j}q \ra =0 \text{ and } \la v, \partial_x^{-1}\partial_{c_j}q\ra =0 \,, \; j=1,2 \,.
\end{equation}
\end{lemma}
\begin{proof}
Let $\varphi: H^2\times \mathbb{R}^2 \times (\mathbb{R}_+)^2 \to \mathbb{R}^4$ be defined by 
$$\varphi(u,a,c) = 
\begin{bmatrix}
\la u-q(\cdot,a,c),\partial_x^{-1}\partial_{a_1}q \ra \\
\la u-q(\cdot,a,c),\partial_x^{-1}\partial_{a_2}q \ra \\
\la u-q(\cdot,a,c),\partial_x^{-1}\partial_{c_1}q \ra \\
\la u-q(\cdot,a,c),\partial_x^{-1}\partial_{c_2}q \ra \\
\end{bmatrix}
$$
Using that $\omega\big|_M = da_1\wedge dc_1 + da_2\wedge dc_2$, we compute the Jacobian matrix of $\varphi$ with respect to $(a,c)$ at $(q(\cdot,\tilde a,\tilde c),\tilde a,\tilde c)$ to be
$$D_{a,c}\varphi(q(\cdot,\tilde a,\tilde c),\tilde a,\tilde c) = 
\begin{bmatrix}
0 & 0 & 1 & 0 \\
0 & 0 & 0 & 1 \\
1 & 0 & 0 & 0 \\
0 & 1 & 0 & 0 
\end{bmatrix} \,.
$$
By the implicit function theorem, the equation $\varphi(u,a,c)=0$ can be solved for $(a,c)$ in terms of $u$ in a neighbourhood of $q(\cdot,\tilde a,\tilde c)$.
\end{proof}

We also record the following lemma which will be useful in the 
next section:
\begin{lemma}
\label{l:linorth}
Suppose $ v $ solves a linearized equation
\[ \partial_t v = \frac 12 \partial_x I_3''(q(t)) v = \partial_x ( 
- \partial_x^2 - 6q(t)^2) v \,, \ \ \ q ( x, t ) = q ( x , a_j + t c_j^2, c_j ) \,. \]
Then 
\[ \partial_t \langle v (t), \partial_x^{-1} ( \partial_{c_j} q ) ( t ) \ra = 
\partial_t \la v ( t) , \partial_x^{-1} ( \partial_{a_j} q ) ( t ) \ra = 0 \,, \]
where $ ( \partial_{c_j } q )( t ) = (\partial_{c_j} q ) ( x , a_j + tc_j^2 , c_j ) $
(and not $ \partial_{c_j} ( q ( x , a_j + t c_j^2 , c_j ) ) $).
In addition, for $  v(0) = \partial_{a_j} q $, 
$ v ( t) = ( \partial_{a_j} q) ( t)  $,  and for $ v (0 ) 
= \partial_{c_j} q  $,  
$$ v ( t) = (\partial_{c_j}  q )( t )  +
2 c_j t ( \partial_{a_j} q ) ( t ) \,. $$
\end{lemma}

\section{Lyapunov functional and coercivity}
\label{S:Lyapunov}

In this section we introduce the function 
$H_c$ adapted from the KdV theory of Maddocks-Sachs \cite{MS}. We will build our Lyapunov functional $\mathcal{E}$ from $ H_c $.

Thus let
$$H_c(u) \defeq I_5(u) + (c_1^2+c_2^2)I_3(u) + c_1^2c_2^2I_1(u) \,.$$
We give a direct proof that $q(\cdot,a,c)$ is a critical point of $H_c$:
\begin{lemma}[$q$ is a critical point of $H$]
\label{l:E7}
We have
\begin{equation}
\label{E:7}
H_c'(q(\cdot, a,c)) =0\,,
\end{equation}
that is 
\[  I_5'(q) + (c_1^2+c_2^2)I_3'(q) + c_1^2c_2^2I_1'(q ) = 0 \,. \]
\end{lemma}
\begin{proof}
We follow Lax \cite[\S 2]{Lax}: we want to find $ A = A (q) $ and $ B 
= B ( q ) $ such 
that 
\[  H'(q) \defeq I_5'(q) + A I_3'(q) +  B I_1'(q ) = 0 \,, \]
for all $ q = q ( x , a, c ) \in M $. 
If we consider the mKdV evolution of $ q $ given by 
\eqref{E:ds}, then  Lemma \ref{L:q-asymp} shows that 
as $ t \rightarrow \pm \infty $ we can express $ H'(q) $ 
asymptotically using 
$ H'( \eta_{c_1}  ) $ and $  H'( \eta_{c_2}  ) $. 
From  \eqref{eq:Ijeta} we see that
\[ H'(\eta_c) =  I_5'(\eta_c ) + A I_3' ( \eta_c ) +  B I_1' ( \eta_c ) = 
2 ( c^4 - A c^2 + B )\eta_c \,. \]
Two parameters $ c_1 $ and $ c_2 $ are roots of this equation
if $ A = c_1^2 + c_2^2 $ and $  B = c_1^2 c^2_2 $ and 
this choice gives 
\begin{gather}
\label{eq:rexp}
\begin{gathered}
 H'( q( t ) ) = r ( t ) \,, \ \| r ( t ) \|_{L^2} \leq C \exp ( -|t|/C ) \,, \\
 q ( t ) \defeq q ( x, a_1 + c_1^2 t , a _2 + c_2^2 t , c_1, c_2 ) \,,
\end{gathered}
\end{gather}
where the exponential decay of $ r ( t ) $ comes from Lemma \ref{L:q-asymp} 
and the fact that $ c_1 \neq c_2 $.

To prove \eqref{E:7} we need to show that $ r ( 0 ) \equiv 0 $.
For the reader's convenience we provide a direct proof of this 
widely accepted fact. Since it suffices to prove that 
$ \langle r ( 0 ) , w \rangle=0 $, for all $ w \in {\mathcal S} $,
we consider the mKdV linearized equation
at $ q ( t ) $,
\begin{equation}
\label{eq:Lax-v}
  v_t = \frac 12 \partial_x I_3'' ( q ( t) ) v \,, \ \ v (0) =w \in 
{\mathcal S} \,,
\end{equation}
and will show that 
\begin{equation}
\label{eq:folf}  \partial_t 
\la r ( t) , v ( t) \ra = 
 \partial_t  \la H' ( q ( t ) ), v ( t ) \ra = 0 \,. 
\end{equation}
The conclusion $ \langle r(0), w \rangle = 0 $ will the follow
from showing that  
\begin{equation}
\label{eq:limr}
 \langle r ( t) , v ( t ) \rangle \rightarrow 0 \,, \ \ 
 t \rightarrow \infty \,. 
\end{equation}

We first claim that
\[  \partial_t  \la I_k' ( q ) , v \ra = 0 \,, \ \ \forall \, k \,.\]
In fact, from \eqref{E:conserved} we have $ \la I_k' ( \varphi) , \partial _x 
I_3' ( \varphi) \ra = 0 $ for all $ \varphi \in {\mathcal S} $. Differentiating
with respect to $ \varphi $ in the direction of $ v $, we obtain
\[ \la I_k'' ( \varphi) v , \partial _x I_3' ( \varphi) \ra = -
\la I_k' ( \varphi) , \partial _x I_3'' ( \varphi) v \ra \,. \]
Applying this with $ v = v ( t) $ and $ \varphi = q ( t ) $ we conclde that 
\[ \begin{split} \partial_t \langle I_k' ( q ) , v \rangle & = \langle
I''_k ( q ) \partial_t q , v \ra + \frac12\langle I'_k ( q ) , \partial_x I_3''(q) v \ra \\
& = \frac 12 \langle I''_k ( q ) \partial_x I_3'(q) , v \rangle + 
\frac 12 \langle I_k' ( q ) , \partial_x I_3''(q) v \ra \,, \\
&=0 \,. \end{split} \]
Since $ H $ is a linear combination of $ I_k$'s, $ k=1,3,5$, 
this gives \eqref{eq:folf}.

We now want to use the exponential decay of $ \| r( t ) \|_{L^2} $ in \eqref{eq:rexp},  
and 
\eqref{eq:folf} to show \eqref{eq:limr}. 
Clearly, all we need is a 
subexponential estimate on $  v ( t ) $, that is 
\begin{equation}
\label{eq:vtse}
\forall \, \epsilon > 0 \ \exists \, t_0\,,  \ \ 
\| v ( t ) \|_{L^2} \leq e^{ \epsilon t } \, , \ \ t > t_0 \,.
\end{equation}

Let $\psi$ be a smooth function such that $\psi(x) =1 $ for all $|x|\leq 1$ and $\psi(x) \sim e^{-2|x|}$ for $|x|\geq 1$.  With the notation of Lemma \ref{L:q-asymp}  define
$$\psi_j( x, t)  = \psi(\delta(x-(a_j+c_j^2t)\widehat{\;}\, )). $$
for $0 < \delta \ll 1$ to be selected below and $j=1,2$.
We now establish that
\begin{equation}
\label{eq:enmet}
\left| \partial_t \left( \|v\|_{L^2}^2 + \|v_x\|_{L^2}^2 + 6\int q^2v^2\right) \right| \lesssim \sum_{j=1}^2 \|\psi_j v\|_{L^2}^2 \,.
\end{equation}
To prove \eqref{eq:enmet}, apply $\partial_x^{-1}$ to \eqref{eq:Lax-v} and pair with $v_t$ to obtain
$$0 = \la \partial_x^{-1}v_t,v_t\ra + \la v_{xx},v_t\ra + \la 6q^2v,v_t\ra$$
which implies
\begin{equation}
\label{eq:stepA}
\partial_t \left( \frac12 \|v_x\|_{L^2}^2 + 3\int q^2 v^2 \right)
= 6 \int q q_t v^2 
\end{equation}
Next, pair \eqref{eq:Lax-v} with $v$ to obtain
$$ 0 = \la v_t,v \ra + \la v_{xxx},v \ra + 6\la \partial_x (q^2v),v \ra$$
which implies
\begin{equation}
\label{eq:stepB}
\partial_t \|v\|_{L^2}^2 = -12 \int q q_x v^2
\end{equation}
Summing \eqref{eq:stepA} and \eqref{eq:stepB} gives \eqref{eq:enmet}.

The inequality \eqref{eq:enmet} shows that we need to control
is $ \| \psi_j  v ( t) \|$, $ j =1, 2$. For $ t $ large $ \psi_j $ 
provides a localization to the region where $ q $ decomposes into 
an approximate sum of decoupled solitons (see Lemma \ref{L:q-asymp}).
Hence we define
\[ {\mathcal L}_j = c_j^2 - \partial_x^2 - 6 \eta^2 ( x , (a_j + t c_j^2)\widehat{\;}\, , c_j) \]
(see also \S \ref{S:correction} below for a use of similar operators). 
A calculation shows that
\begin{equation}
\label{eq:dostar}
t \geq T ( \delta ) \ \Longrightarrow \ \partial_t \langle {\mathcal L}_j \psi_j v , \psi_j v \rangle = {\mathcal O} ( \delta ) 
\| v \|_{H^1}^2 \,, \end{equation}
where $ T ( \delta ) $ is large enough 
to ensure that the supports of $ \psi_j$'s are separated.   
It suffices to assume that $v(0)=w$ satisfies
  $\la w, \partial_x^{-1}\partial_{a_j}q\ra =0$ and $\la w, \partial_x^{-1}\partial_{c_j}q\ra =0$, since Lemma \ref{l:linorth} already showed 
that the evolutions of $\partial_{a_j}q$ and $\partial_{c_j}q$ are linearly bounded in $t$.  Under this assumption, we have by Lemma \ref{l:linorth} that $\la v(t), \partial_x^{-1}\partial_{a_j}q(t)\ra =0$ and $\la v(t), \partial_x^{-1}\partial_{c_j}q(t)\ra =0$.

We now want to invoke the well known coercivity estimates for operators $ {\mathcal
L}_j $ -- see for instance \cite[\S 4]{HZ1} for a self contained presentation. 
For that we need to check that 
\[ |\la \psi_j v , \partial_x^{-1} \partial_a\eta ( \hat 
a_j + tc_j^2 , c_j ) \ra| \ll 1\,, \ \
 |\la \psi_j v , \partial_x^{-1} ( \partial_c \eta ( \hat a_j +t c_j^2 , c_j ) | \ra|
\ll 1 \,. \]
This follows from the fact that $ v $ is symplectically orthogonal to 
$ (\partial_{c_j} q ) ( t ) $ and $ \partial_{a_j} q ( t ) $ (Lemma \ref{l:linorth}
again), the fact that $ q $ decouples into two solitons for $ t $ large,
and from the remark after the proof of Lemma \ref{E:sf}.

Hence, 
\[ \la {\mathcal L}_j \psi_j v , \psi_j v \ra \gtrsim \| \psi_j v \|_{H^1}^2 \,.\]
We now sum \eqref{eq:enmet} and \eqref{eq:dostar} multiplied by $ \delta^{-\frac12} $
to obtain, for $ t $ suffieciently large (depending on $ \delta$),
\begin{gather*}  F' ( t ) \leq C \delta^{\frac 12} F ( t )  \,, \\
F ( t ) \defeq   \|v ( t) \|_{H^1}^2 + 6\int q^2 ( t ) v ( t ) ^2  + \delta^{-\frac12} 
\la {\mathcal L}_j( t )  \psi_j ( t )  v ( t)  , \psi_j ( t) v ( t ) \ra 
\end{gather*}
(where we added the additional $ \int q^2 v^2 $ term to the right hand side
at no cost). Consequently, $ F ( t ) \leq  \exp ( C' \delta^{\frac12} t ) $, for 
$ t > T_1 ( \delta ) $. 

We recall that this implies \eqref{eq:vtse} and going back to \eqref{eq:folf}
show that $ r ( 0 ) = 0 $, and hence $ H'( q ) = 0 $.
\end{proof}

We denote the Hessian of $ H_c $ at $ q ( \bullet, a, c ) $ by $ 
{\mathcal K}_{c,a} $:
$$\mathcal{K}_{c,a} = I_5''(q) + (c_1^2+c_2^2)I_3''(q) + c_1^2c_2^2I_1''(q)$$
It is a fourth order self-adjoint operator on $ L^2 ( \RR ) $ and 
a calculation shows that
\begin{equation}
\label{E:8}
\frac12  \mathcal{K}_{c,a} =
\begin{aligned}[t]
& (-\partial_x^2+c_1^2)(-\partial_x^2+c_2^2) \\
&  + 10 \partial_x \, q^2 \partial_x 
+10(-q_x^2 + (q^2)_{xx} +3q^4) -6(c_1^2+c_2^2) q^2 \\
\end{aligned}
\end{equation}

\begin{lemma}[mapping properties of $\mathcal{K}$]
\label{L:mpk}
The kernel of $ {\mathcal K}_{c,a} $ in $ L^2 ( \RR ) $ 
is spanned by $ \partial_{a_j} q $:
\begin{equation}
\label{E:Kqa}
\mathcal{K}_{c,a}\partial_{a_j}q = 0,
\end{equation}
and
\begin{equation}
\label{E:Kqc}
\begin{aligned}[t]
\mathcal{K}_{c,a}\partial_{c_j}q & = 
 4 ( -1)^jc_j(c_1^2-c_2^2)\partial_x^{-1}\partial_{a_j}q
\end{aligned}
\end{equation}
\end{lemma}
\begin{proof}
Equations \eqref{E:Kqa} follow from 
differentiation of \eqref{E:7} with respect to $a_j$.   As $ x \rightarrow 
\infty $, the leading part of $ {\mathcal K}_{c,a} $ is given by 
$ ( - \partial_x^2 + c_1^2 ) ( -\partial_x^2 + c_2^2 ) $ and hence 
the kernel in  $ L^2 $ is at most two dimensional.

To see \eqref{E:Kqc} 
recall that 
$$I_1'(q) = 2q = -2\partial_x^{-1}(\partial_{a_1}q+\partial_{a_2}q)$$
$$I_3'(q) = -2q''-4q^3 = 2\partial_x^{-1}
(c_1^2\partial_{a_1}q+c_2^2\partial_{a_2}q) \,, $$
where we used Lemma \ref{L:magic}. 
By differentiating 
$H'(q)=I_5'(q) + (c_1^2+c_2^2)I_3'(q) + c_1^2c_2^2I_1'(q)=0$ 
with respect to $c_j$, we obtain
\begin{equation}
\label{eq:dcK}
\mathcal{K}(\partial_{c_1}q) = -2c_1(I_3'(q)+c_2^2I_1'(q)) \,, \ \ 
\mathcal{K}(\partial_{c_2}q) = -2c_2(I_3'(q)+c_1^2I_1'(q))\,. 
\end{equation}
Inserting the above formul{\ae} for $I_1' ( q ) $ and $ I_2 ' ( q )$ 
gives \eqref{E:Kqc}.
\end{proof}

The main result of this section is the following coercivity result:
\begin{proposition}[coercivity of $\mathcal{K}$]
\label{p:ck}
There exists $\delta=\delta(c)>0$ such that for all $v\in H^2$ satisfying the symplectic orthogonality conditions 
$$\la v,\partial_x^{-1}\partial_{a_j}q \ra =0 \text{ and } \la v, \partial_x^{-1}\partial_{c_j}q\ra =0 \,, \; j=1,2 \,,$$ 
we have
\label{L:coercivity}
\begin{equation}
\label{E:10}
 \delta \|v\|_{H^2}^2 \leq \la \mathcal{K}_{c,a}v,v \ra \,.
\end{equation}
\end{proposition}

The proposition is proved in a few steps. In Lemma \ref{L:mpk} 
we already described the kernel $ {\mathcal K}_{c,a} $ and now
we investigate the negative eigenvalues:

\begin{proposition}[Spectrum of $ {\mathcal K}$]
\label{L:nsk}
The operator $ {\mathcal K}_{c,a} $ has a single negative 
eigenvalue, $ h \in L^2 ( \RR )  $:
\begin{equation}  {\mathcal K}_{c,a} h= - \mu h \,, \ \ \mu > 0 \,. 
\label{eq:Kh}
\end{equation}
In addition, for 
$$ 0 < \delta < c_1 < c_2 - \delta < 1/\delta \,, $$ 
 there
exists a constant, $ \rho $, depending only on $ \delta $, 
such that
\begin{equation}
\label{eq:infs}
\min \{ \lambda > 0 \; : \; \lambda \in \sigma ( \mathcal{K}_{c,a}) \}
> \rho   \,, \ \ a \in \RR^2 \,, 
\end{equation}
\end{proposition} 
\begin{proof}
As always we assume $0<c_1<c_2$.
We know the continuous spectrum of $\mathcal{K}_{c,a}$,
\[ \sigma_{\rm{ac}} ( \mathcal{K}_{c,a} ) =  [2c_1^2c_2^2,+\infty) \]
and that for all $a,c$, 
there is a two-dimensional kernel 
given by 
 $ \spn\{\partial_{a_1}q, \partial_{a_2}q\}$.  
The eigenvalues depend continuously on $a$, $c$, and hence the constant
dimension of the kernel shows that the number of negative eigenvalues
is constant (since the creation or annihilation of a negative 
eigenvalue would increase the dimension of $ \ker {\mathcal K}_{c,a} $.)

Hence it suffices to determine the number of negative 
eigenvalues of $\mathcal{K}$ for any convenient values of $a$, $c$.
To do that we use the following fact:
\begin{lemma}[Maddocks-Sachs {\cite[Lemma 2.2]{MS}}]
Suppose that $\mathcal{K}$ is a self-adjoint, $4$th order operator of the 
form 
$$\mathcal{K}=  2 (-\partial_x^2 + c_1^2)(-\partial_x^2+c_2^2) 
+  p_0(x) - \partial_x p_1(x)\partial_x  \,,$$
where the coefficients $p_j(x)$ are smooth, real,
and rapidly decaying as $x\to \pm\infty$.  Let $r_1(x)$, $r_2(x)$ be 
two linearly independent solutions of $\mathcal{K}r_j=0$ such that $r_j \to 0 $ as $x\to -\infty$.  

Then the number of negative eigenvalues of $ {\mathcal K} $ is 
equal to 
\begin{equation}
\label{eq:roots}
\sum_{x\in \RR } \dim \ker \left[ \begin{array}{ll} r_1 ( x ) & r_1'( x) \\
r_2 ( x ) & r_2' ( x ) \end{array} \right] \,. \end{equation}
\end{lemma}
We apply this lemma with $ {\mathcal K} = {\mathcal K}_{c,a} $, 
in which case 
\begin{align*}
&p_1 = 20q^2 \,, 
\ \ \ 
p_0 = 40q_{xx}q + 20 q_x^2 + 60q^4 - 12(c_1^2+c_2^2)q^2\,, \ \ \
q = q ( \bullet, a , c ) \,.
\end{align*}
Convenient values of $ a $ and $ c $ are provided by 
$ a_1 = a_2 = 0 $ and $ c_1 = 0.5 $, $ c_2 = 1.5 $. 
In the notation of \eqref{eq:defQ} we then have 
$ q ( x, a , c  ) = Q ( x, 0 , 0.5) $, and since
\[ \partial_x Q = -\partial_{a_1} q - \partial_{a_2} q \,, \ \ 
\partial_\alpha Q = -\partial_{a_1} q + \partial_{a_2 } q \,, \]
we can take $ r_1 = \partial_x Q $ and $ r_2 = \partial_\alpha Q $.
A computation based on \eqref{eq:Qdec} and \eqref{eq:deftau} shows 
that
\begin{equation}
\label{eq:12}
\begin{split}
 Q ( x, 0.5, 0 ) &  =  \sech(x/2)  \,,  \ \ \
  \partial_x Q ( x , 0.5 , 0 )  = -  \frac{ \sinh ( x/2 ) }
{2 \cosh^2 ( x /2 ) } \,, \\
\partial_\alpha Q ( x , 0.5 , 0 ) & =   \frac{ \sinh ( x/2 ) }
{4 \cosh^4 ( x/ 2 ) } ( 9 - 2 \cosh^2( x/2 ) ) \\
& = \frac{ 9 \sinh ( x/2 ) } {4 \cosh^4 ( x / 2 ) } + \partial_x Q (x, 0.5, 0 ) 
\,. 
\end{split}
\end{equation}
Since $ x \mapsto y = \sinh ( x/ 2 ) $ is invertible, 
we only need to check the dimension of the kernel the Wronskian matrix of 
\[ \tilde r_1 ( y ) = \frac y { 1 + y^2 } \,, \ \ \tilde r_2 ( y ) 
 = \frac y { ( 1 + y^2)^2 } \,, \] 
and that is equal to $ 1 $ at $ y = 0 $ and $ 0 $ on $ \RR \setminus \{0\} $.
In view of \eqref{eq:roots}
this completes the proof of \eqref{eq:Kh}

To prove \eqref{eq:infs} we first note that by rescaling \eqref{eq:q2Q} 
we only need to 
prove the estimate for 
\[ K ( c , \alpha) \defeq  {\mathcal K}_{((c,1),(-\alpha,\alpha))} \,, 
\ \ c \in [ \delta , 1 - \delta ] \,, \ \ 0 < \delta < 1/2 \,.  \]
For that we 
introduce another operator 
\begin{equation} 
\label{eq:defP}
P( c) \defeq ( - \partial_x^2 + 1 ) ( - \partial_x^2 + c^2) 
+ 10 \partial_x \eta^2 \partial_x + 10 ( 3 \eta^2 - 2 \eta^4 ) 
- 6 ( 1 + c^2 ) \eta^2 \,, 
\end{equation}
where 
\[ \eta = \sech x \,, \ \ c \in \RR_+ \setminus \{ 1 \} \,. \]
The operator $ P ( c ) $ is the Hessian of $ H_{(c,1)} $ 
at $ \eta $, which is also a critical point for $ H_{(c,1)} $.
In particular,
\[  P ( c) \partial_x \eta = 0 \,.\]
Putting, 
$$ U_\alpha f ( x ) \defeq f ( x + \alpha + \log ( (1 + c)/ ( 1 - c))) )
\,, $$ 
and 
$$ P_+( c, \alpha ) \defeq U_\alpha^* P ( c ) U_\alpha \,, $$
we see that 
\[  K ( c , \alpha ) = 2 P_+ ( c, \alpha )  +
{\mathcal O} ( e^{ - ( \alpha +  |x| )/C } )  \partial_x^2  + 
{\mathcal O} ( e^{ - ( \alpha +  |x| )/C }) \,, \ \
x \geq 0  \,. \]
Similarly, if 
$$  T_c f ( x ) \defeq \sqrt c f ( cx ) \,, $$ 
and 
$$ P_- ( c , \alpha ) \defeq c^2 U_\alpha T_c  P ( 1/c ) T_c^* U_\alpha^* \,,$$
then 
\[  K ( c , \alpha ) = 2 P_- ( c, \alpha) + 
{\mathcal O} ( e^{ - ( \alpha +  |x| ) /C } )  \partial_x^2  + 
{\mathcal O} ( e^{ - ( \alpha +  |x| ) /C }) \,, \ \
x \leq 0  \,. \]

We reduce the estimate \eqref{eq:infs} to a spectral 
fact about the operators $ P ( c ) $ and $ P ( 1/c) $:
\begin{lemma}
\label{l:KP}
Suppose that there exists 
\[  \alpha \longmapsto \lambda( c, \alpha ) \in \RR \setminus \{ 0 \} \,
\]
such that 
\[  \lambda( c, \alpha  ) \in \sigma ( K ( c , \alpha)) \,, \ \
\lambda ( c , \alpha ) \longrightarrow 0 \,, \ \alpha \longrightarrow \infty
\,. \]
Then we have 
\begin{equation}
\label{eq:lKP}    \dim \ker_{L^2} P ( c ) + \dim \ker_{L^2} 
 P( 1/c ) > 2 \,, 
\end{equation}
where $ \ker_{L^2} $ means the kernel in $ L^2 $.
\end{lemma}
\begin{proof} 
The assumption that $ 0 \neq \lambda ( c , \alpha) 
\rightarrow 0 $ as $ \alpha \rightarrow \infty $ implies that
there exists a family of quasimodes $ f_\alpha $, $ \| f_\alpha \|_{L^2} = 1$,
\begin{equation}
\label{eq:Kca} \|  K ( c , \alpha ) f_{\alpha } \|_{ L^2} = o(1)\,, \ \  \alpha  \longrightarrow \infty \,, \ \ \ \ f_\alpha \perp \ker_{L^2} K(c , \alpha ) \,. 
\end{equation}
Since we know that the kernel of $ K ( c, \alpha ) $ is spanned 
by $ U_\alpha ^* \partial_x \eta + {\mathcal O} ( e^{ - ( | x | + \alpha ) /C } ) $
and $ U_\alpha T_c \partial_x \eta + {\mathcal O} ( e^{ - ( | x | + \alpha ) /C } ) 
$, we can modify $ f_\alpha $ and replace the orthogonality condition by 
\[ f_\alpha \perp {\rm{span}}\, ( U_\alpha^* \partial_x \eta , 
U_\alpha T_c \partial_x \eta ) \,. \]
The estimate in \eqref{eq:Kca}, and $ \|f_\alpha \|_{L^2} = {\mathcal O}(1) 
$,  imply that
\begin{equation}
\label{eq:H2}  \| f_\alpha \|_{ H^2 } = {\mathcal O} ( 1 ) \,, \ \
\alpha \longrightarrow \infty \,. 
\end{equation}

We first claim that 
\begin{equation}
\label{eq:m11}   \int_{-1}^1  |f_\alpha ( x )  |^2 dx = o ( 1 ) \,, \ \ 
\alpha  \longrightarrow \infty \,.  
\end{equation}
In fact, on $ [ - \alpha/2, \alpha/2 ] $, 
\[ K( c, \alpha ) = ( - \partial_x^2 + c^2 ) ( - \partial_x^2 + 1 ) 
+ {\mathcal O} ( e^{-\alpha/C } )  \partial_x^2 + 
{\mathcal O} ( e^{-\alpha/ c } ) \,, \]
and hence, using \eqref{eq:H2},
\[ ( - \partial_x^2 + c^2 ) ( - \partial_x^2 + 1) f_\alpha = r_\alpha \,,
\ \  \| r_\alpha \|_{ L^2 ( [ - \alpha/2 , \alpha/2 ] )} = o ( 1 ) \,. \]
Putting 
$$ e_\alpha \defeq
 [ ( - \partial_x^2 + c^2 ) ( - \partial_x^2 + 1) ]^{-1} \left( r_\alpha 
\bbbone_{ [-\alpha/2 , \alpha/2] } \right) \,, \ \ \ 
\| e_{\alpha } \|_{ H^2 } = o( 1) \,, $$
we see that $ f_\alpha = g_\alpha + e_\alpha $ where
\begin{equation}
\label{eq:eqg}   ( - \partial_x^2 + c^2 ) ( - \partial_x^2 + 1) g_\alpha ( x ) = 0 \,, \ \ 
|x| <  \alpha/2 \,. \end{equation}
Suppose now that \eqref{eq:m11} were not valid. Then the same would
be true for $ g_\alpha $, and there would exist a constant $ c_0 > 0 $, 
and a sequence $ \alpha_j 
\rightarrow \infty $,  for which $  \| g_{\alpha_j} \|_{ L^2 ([-1,1])} > c_0 $.
In view of \eqref{eq:eqg} this implies that
\[  g_{\alpha_j} ( x ) = \sum_{\pm} 
\left( a^\pm_{j} e^{\pm x} + b_{j}^\pm e^{\pm c x} \right) 
\,, \ \ |x| < \alpha/2 \,, \ \ |a_j^\pm |, |b_j^\pm | = {\mathcal O}(1) \,, \] 
and for at least one choice of sign, 
$$ |a_j^\pm |^2 + |b_j^\pm |^2 > c_1 > 0 \,.$$
We can choose a subsequence so that this is true for a fixed sign, 
say, $ + $, for all $ j$. In that case, a simple calculation shows that
for $ M_j \rightarrow \infty $, $ M_j \leq \alpha_j / 2 $, 
\[ 
\begin{split} \int_0^{M_j}  | g_{\alpha_j}  ( x ) |^2 dx & \geq 
\frac 12 |a_j^+|^2 e^{2M_j} + \frac 1 {2c} |b_j^+|^2
e^{ 2 c M_j }  - \frac 2 { c+1} 
|a_j^+ | | b_j^+| e^{ (c+1) M_j } \\
& \ \ \ \ \   - \frac 2 {1-c} 
|a_j^+ | | b_j^-| e^{ (1-c) M_j } - {\mathcal O} ( 1 ) \\
& \geq \frac12 \left( \frac{ 1 - c} { 1 + c } \right)^2 
\left(  |a_j^+|^2 e^{2M_j} + \frac 1 {c} |b_j^+|^2 e^{ 2 M_j c } \right) \\
& \ \ \ \ \ - \frac 4 { ( 1 - c )^2 } |a_j^+|^2 e^{ 2 ( 1 - c) M_j} 
- {\mathcal O} ( 1 ) \,, 
\end{split} 
\]
where we used the fact that $ 0 < \delta < c < 1 - \delta $.
Hence
\[ \begin{split} 
\| f_{\alpha_j } \|_{L^2} & \geq \int_0^{M_j}  | f_{\alpha_j}  ( x ) |^2 dx 
\geq \int_0^{M_j}  | g_{\alpha_j}  ( x ) |^2 dx - o ( 1 ) \\
& \geq  \frac12 \left( \frac{ 1 - c} { 1 + c } \right)^2 c_1 e^{ 2 M_j c }
- {\mathcal O} ( 1 ) \longrightarrow \infty \,, \ \ j \rightarrow \infty \,.
\end{split} \]
Since $ \| f_{\alpha } \|_{L^2} = 1 $ we obtain a contradiction 
proving \eqref{eq:m11}. 

Now let $ \chi_\pm C^\infty  ( \RR ) $ be supported in $ \pm [ -1 , \infty ) $,
and satisfy $ \chi_+^2 + \chi_-^2 = 1$. Then \eqref{eq:m11} (and the 
corresponding estimates for derivatives obtained from \eqref{eq:Kca})
shows that
\[  \|  P_\pm ( c, \alpha) 
(\chi_\pm  f_{\alpha } ) \|_{ L^2 }  = o(1)\,, \ \  \alpha  \longrightarrow \infty \,. \]
For at least one of the signs we must have $ \| \chi_\pm f_\alpha \|_{L^2}
> 1/3 $ (if $ \alpha $ is large enough), and hence we obtain 
a quasimode for $ P_\pm ( c , \alpha ) $, orthogonal to the 
known element of the kernel of $ P_\pm ( c, \alpha ) $. This 
means that $ P_\pm ( c, \alpha ) $, for at least one of the signs
has an additional eigenvalue approaching $ 0 $ as $ \alpha \rightarrow \infty $.
Since the spectrum of $ P_\pm ( c , \alpha) $ is independent of 
$ \alpha $ it follows that for at least one sign the kernel is two 
dimensional. This proves \eqref{eq:lKP}.
\end{proof}

The next lemma shows that \eqref{eq:lKP} is impossible:
\begin{lemma}
\label{l:sP}
For $ c \in \RR_+ \setminus \{ 1 \} $
\begin{equation}
\label{eq:lsP}
\ker_{L^2}  P ( c ) = \CC \cdot \partial_x \eta \,.
\end{equation}
\end{lemma}
\begin{proof} 
Let $ {\mathcal L} \defeq ( I_3'' ( \eta )  + I''_1 ( \eta )) /2  $:
\[ {\mathcal L} v = - v_{xx} - 6 \eta^2 v + v 
\,, \ \ \eta ( x ) = \sech(x ) \,.\]
We recall (see the comment after \eqref{eq:defP}) that
\[ P ( c ) = \frac 12 H''_{(c,1)} ( \eta ) = 
\frac 12 \left( I_5''(\eta) + ( 1 + c^2) I_3'' ( \eta ) + c^2 I_1''( \eta ) 
\right) \,.\]
We already noted that 
\[ {\mathcal L} ( \partial_x \eta ) = P ( c ) \partial_x \eta = 0 \,, \]
and proceeding as in \eqref{eq:dcK} we also have
\begin{equation}
\label{eq:xdx} {\mathcal L} ( \partial_x( x  \eta ) ) = - 2 \eta \,, \ \
P ( c )  ( \partial_x( x  \eta)  ) =  2 ( 1 -c^2 ) \eta \,. 
\end{equation}

We claim that 
\begin{equation}
\label{eq:comm}
P ( c ) \partial_x {\mathcal L } = {\mathcal L} \partial_x P ( c ) 
\end{equation}
Since $ I_j' ( \eta + t v ) = t I_j'' ( \eta ) v +{\mathcal O} (t^2) $, 
$ v \in {\mathcal S} $, 
the equation \eqref{E:conserved} implies that
\[ \langle I''_j ( \eta ) v , \partial_x I_k'' ( \eta ) v \rangle = 0 
\,, \ \ \forall \, j , k \,, \ \ v \in {\mathcal S} \,. \]
From this we see that 
\[  \langle P ( c ) v , \partial_x {\mathcal L} v \rangle =0 \,, \ \
\forall v \in {\mathcal S} \,, \]
and hence by polarization,
\[  \langle P ( c ) v , \partial_x {\mathcal L} w \rangle 
= - \langle P ( c ) w , \partial_x {\mathcal L} v  \rangle = 
\langle \partial_x P ( c ) w , {\mathcal L } v \rangle \,. \]
which implies \eqref{eq:comm}. 

Suppose now that $ \dim \ker_{L^2} P ( c ) = 2 $ for some $ c \neq 1 $,
and let $ \eta_x $ and $ \psi $ be the basis of this kernel. Since
$ P ( c ) $ is symmetric with respect to the reflection $ x \mapsto - x $,
$ \psi $ can be chosen to be either even or odd. 
Applying \eqref{eq:comm} to $ \psi$ we get $ P ( c) \partial_x 
{\mathcal L } \psi = 0 $ and hence
\[ \partial_x {\mathcal L} \psi = \alpha \eta_x + \beta \psi \,, \]
for some $ \alpha, \beta \in \RR $. 

If $ \psi $ is odd then $ \partial_x {\mathcal L} \psi $ is even,
and therefore $ \alpha = \beta = 0 $. But then $ \psi \in 
\ker_{L^2} {\mathcal L} = \CC \cdot \eta_x $,
giving a contradiction.

If $ \psi $ is even then $ \partial_x {\mathcal L} \psi $ is odd,
$ \beta =0 $ and $ {\mathcal L} \psi = \alpha \eta $. 
We have  $ \alpha \neq 0 $
since $ \psi $ is orthogonal to the kernel of $ {\mathcal L}$, spanned
by $ \partial_x \eta $. From \eqref{eq:xdx} we obtain
\[ \psi = - \frac \alpha 2 \partial_x ( x \eta ) \,. \]
Applying the second equation in \eqref{eq:xdx} we then obtain
\[ P ( c) \psi = - \alpha ( 1 - c^2 ) \eta  \,, \]
contradicting $ \psi \in \ker_{L^2 } P ( c ) $.
\end{proof}

With this lemma we complete the proof of Proposition \ref{L:nsk}.
\end{proof}

To obtain the coercivity statement in Proposition \ref{p:ck} we first
obtain coercivity under a different orthogonality condition:
\begin{lemma}
\label{l:ck}
There exists a constant $\rho>0$ depending only on $c_1$, $c_2$, such that the following holds:  If $\la u, \partial_x^{-1}\partial_{a_1}q\ra =0$, $\la u, \partial_x^{-1}\partial_{a_2}q\ra =0$, $\la u, \partial_{a_1}q\ra =0$, $\la u, \partial_{a_2}q\ra =0$, then $\la \mathcal{K}_{c,a} u,u \ra \geq \rho\|u\|_{L^2}^2$. 
\end{lemma}
\begin{proof}
To simplify notation we put $ \mathcal{K}=\mathcal{K}_{c,a} $ in the proof. 
Using \eqref{E:Kqc} and the expression for the symplectic form, 
$\omega\big|_M = da_1\wedge dc_1 + da_2\wedge dc_2$, we have
$$\la \mathcal{K}\partial_{c_1}q, \partial_{c_1}q \ra = -4c_1(c_1^2-c_2^2)\la \partial_x^{-1}\partial_{a_1}q,\partial_{c_1}q\ra = 4c_1(c_1^2-c_2^2)$$
and similarly
\begin{equation}
\label{E:5}
\la \mathcal{K}\partial_{c_2}q, \partial_{c_2}q \ra = -4c_2(c_1^2-c_2^2) \,.
\end{equation}
Since we assumed that 
 $c_1<c_2$,  $\la \mathcal{K}\partial_{c_1}q, \partial_{c_1}q \ra<0$.  


Let $\widetilde{\partial_{c_1}q}$ be the orthogonal 
projection of $\partial_{c_1}q$ on $\left(\ker \mathcal{K} \right)^\perp$.  
We first claim that there exists a constant $\alpha$ such that $u=\tilde u + \alpha \widetilde{\partial_{c_1}q}$ with $\la \tilde u, h \ra =0$,
where $ \mu $ and $ h $ are defined in Proposition \ref{L:nsk}.

  To prove this, decompose $\partial_{c_1}q$ as $\partial_{c_1}q=\xi+\beta h$ with $\la \xi, h \ra=0$.  Then by \eqref{E:5}
\begin{align*}
0 &> \la \mathcal{K}\partial_{c_1}q,\partial_{c_1}q \ra \\
&= \la \mathcal{K}\xi,\xi\ra + 2\beta \la \mathcal{K}h, \xi \ra + \beta^2 \la \mathcal{K}h,h\ra \\
&= \la \mathcal{K}\xi,\xi\ra - \mu \beta^2
\end{align*}
Since $\la \mathcal{K}\xi,\xi\ra \geq 0$, we must have that $\beta \neq 0$.  Hence there exists $u'$ and $\alpha$ such that $u=u'+\alpha\partial_{c_1}q$ with $\la u', h\ra =0$.  Now take $\tilde u$ to be the projection of $u'$ away from the kernel of $\mathcal{K}$.  This completes the proof of the claim.

We have that
$$\la u, \mathcal{K}\partial_{c_1} q \ra = -4c_1(c_2^2-c_1^2) \la u, \partial_x^{-1}\partial_{a_1} q \ra =0$$
by \eqref{E:Kqc} and hypothesis.  Substituting $u=\tilde u +\alpha \widetilde{\partial_{c_1}q}$, we obtain
\begin{equation}
\label{E:7'}
\la \tilde u, \mathcal{K}\partial_{c_1}q \ra = -\alpha \la \widetilde{\partial_{c_1}q} , \mathcal{K}\partial_{c_1}q \ra = -\alpha \la \partial_{c_1} q , \mathcal{K}\partial_{c_1}q\ra
\end{equation}

Now let $\tilde \rho$ denote the bottom of the positive spectrum of $\mathcal{K}$.  We have
\begin{align*}
\la \mathcal{K}u, u \ra &= \la \mathcal{K}(\tilde u +\alpha \widetilde{\partial_{c_1} q}) , (\tilde u +\alpha \widetilde{\partial_{c_1} q} )\ra\\
&= \la \mathcal{K}\tilde u, \tilde u \ra + 2\alpha \la \mathcal{K}\tilde u, \partial_{c_1} q\ra + \alpha^2 \la \mathcal{K}\partial_{c_1}q, \partial_{c_1} q\ra \\
&= \la \mathcal{K}\tilde u, \tilde u \ra - \alpha^2 \la \mathcal{K}\partial_{c_1}q, \partial_{c_1}q\ra && \text{by }\eqref{E:7'}\\
&\geq \tilde \rho \|\tilde u \|_{L^2}^2 + 4c_1(c_2^2-c_1^2)\alpha^2\\
&\geq \tilde C(\|\tilde u\|_{L^2}^2 + \alpha^2)
\end{align*}
where $\tilde C$ depends on $c_1$, $c_2$ and $\tilde \rho$.
However, since $u=\tilde u + \alpha \widetilde{\partial_{c_1}q}$, we have
$$\|u\|_{L^2}^2 \leq C(\|\tilde u \|_{L^2}^2+\alpha^2)$$
where $C$ depends on $c_1$, $c_2$ which completes the proof.
\end{proof}

We now put 
\begin{equation}
\label{eq:ale}
\begin{split} E & = E_{a,c} =\ker \mathcal{K}=\spn\{\partial_{a_1}q,\partial_{a_2}q\}\,, \\
F & = F_{a,c} = \spn \{ \partial_x^{-1}\partial_{c_1}q, \partial_x^{-1}\partial_{c_2}q \}\,, \\
G & = G_{a,c} = \spn \{\partial_x^{-1}\partial_{a_1}q, \partial_x^{-1}\partial_{a_2}q \}
\,. 
\end{split} \end{equation}

In this notation Lemma \ref{l:ck} states
that 
\[ u \perp (E+G)   \ \Longrightarrow \ 
\la \mathcal{K}u,u\ra \geq \theta \|u\|_{L^2}^2 \,, \]
while to establish Proposition \ref{p:ck} we need
\[ u \perp (F+G) \Longrightarrow \la \mathcal{K}u, 
u\ra \geq \tilde \theta \|u\|_{L^2}^2 \,. \]
That is, 
we would like to replace orthogonality with the kernel $E$ by orthogonality with a ``nearby'' subspace $F$.  For this, we apply the following analysis with $D=F^\perp$.

\begin{definition}
Suppose that $D$ and $E$ are two closed subspaces in a Hilbert space.  Then $\alpha(D,E)$, the \emph{angle between $D$ and $E$}, is
$$\alpha(D,E) \defeq \cos^{-1} \sup_{\substack{\|d\|=1, \; d\in D \\ \|e\|=1, \; e\in E}} \la d,e\ra $$
\end{definition}

It is clear that 
 $0\leq \alpha(D,E)\leq {\pi}/{2}$, $\alpha(D,E)=\alpha(E,D)$, and
that 
$\alpha(E,D)={\pi}/{2}$ if and only if $E\perp D$. 
We will need slightly more subtle properties stated in the following
\begin{lemma}
\label{l:subtle}
Suppose that $D$ and $E$ are two closed subspaces in a Hilbert space. Then
\begin{equation}
\label{eq:lsu}  \alpha(D,E) = \cos^{-1} \sup_{\|d\|=1, d\in D} \|P_E d\| \,, \ \ 
\alpha(D,E) = \sin^{-1} \inf_{\|d\|=1, d\in D} \|P_{E^\perp} d\| \,.
\end{equation}
In addition if $ E $ is finite dimensional then 
\begin{equation}
\label{eq:lfi}
\alpha(D,E) =0 \  \Longleftrightarrow \ D\cap E \neq \{0\} \,. 
\end{equation}
\end{lemma}
\begin{proof}
To see \eqref{eq:lsu} 
let $d\in D$, with $\|d\|=1$.  By the definition of the projection operator,
\begin{align*}
1- \|P_Ed\|^2 &=\| d-P_E d\|^2 
= \inf_{e\in E} \| d-e\|^2 
= \inf_{\substack{e\in E \\ \|e\|=1}} \inf_{\alpha \in \mathbb{R}} \|d-\alpha e\|^2 
\\ & = \inf_{\substack{e\in E \\ \|e\|=1}} \inf_{\alpha \in \mathbb{R}} (1-2\alpha \la d, e\ra + \alpha^2) 
= \inf_{\substack{e\in E \\ \|e\|=1}} (1-\la  d,e\ra^2)
\\ & = 1-\sup_{\substack{e\in E \\ \|e\|=1}} \la  d,e\ra^2
\end{align*}
and consequently,
$$\|P_Ed\| = \sup_{\substack{e\in E \\ \|e\|=1}} \la  d,e\ra \,, $$
from which the first formula in \eqref{eq:lsu} follows.
The second one is a consequence of the first one as 
$1=\|P_Ed\|^2+\|P_{E^\perp}d\|^2$. 

The $ \Leftarrow $ implication 
in \eqref{eq:lfi} is clear. To see the other
implication, we observe that if $ D \cap E = \{ 0 \} $ and 
$ E $ is finite dimensional then 
\[  \inf_{\substack{{ y \in E } \\ { \| y \|= 1 }} } d ( y  , D ) > 0 
\,, \]
where $ d ( y , D ) = \inf_{ z \in D } \| y - z \| $ is the distance from 
$ y $ to $ D $.
This implies that 
\[  \begin{split} 0 & <  
\inf_{\substack{{ y \in E } \\ { \| y \|= 1 }} } 
\inf_{{ z \in D } } \| y - z \|^2 = 
\inf_{\substack{{ y \in E } \\ { \| y \|= 1 }} } 
\inf_{{ z \in D } } ( 1 - 2 \langle y ,z \rangle + \|z \|^2 ) \\
& \leq \inf_{\substack{{ y \in E } \\ { \| y \|= 1 }} } 
\inf_{\substack{{ z \in D } \\ { \| z \|= 1 }} } 
( 2 - 2 \langle y ,z \rangle )  = 2 ( 1 - 
 \sup_{\substack{{ y \in E } \\ { \| y \|= 1 }} } 
\sup_{\substack{{ z \in D } \\ { \| z \|= 1 }} } 
\langle y ,z \rangle ) \\
& =  2 ( 1 - \cos \alpha ( D , E ) ) \,. 
\end{split} \]
Thus if $ D \cap E = \{ 0 \} $ then $ \alpha ( D, E ) >  0 $.
But that is the $ \Rightarrow $ implication in \eqref{eq:lfi}.
\end{proof}

In the notation of \eqref{eq:ale}, the translation symmetry gives
$$\alpha(E_{a,c},F^\perp_{a,c}) = F ( c_1, c_2,  a_1-a_2 ) \,, $$
where $ F $ is a continuous fuction in $ {\mathcal C} \times \RR $.
We claim that 
\begin{equation}
\label{eq:Fka}  F ( c_1 , c_2 , \alpha ) \geq \kappa_\delta > 0  \ \ \text{for} \ \ 
\delta \leq c_1\leq c_1+\delta \leq c_2 \leq \delta^{-1}\,. 
\end{equation}
Consider now the case $|a_1-a_2| \leq A$ (where $A$ is chosen large below), and hence $c_1$, $c_2$, and $a_1-a_2$ vary within a compact set.  Thus it suffices to check that  $\alpha(E_{a,c},F^\perp_{a,c})$ is nowhere zero and 
this amounts to checking $E\cap F^\perp = \{0\}$. 

 Suppose the contrary, that is that there exists
$$u = z_1 \partial_{a_1}q + z_2 \partial_{a_2}q \in  F^\perp \,. $$
Since
$\omega\big|_M = da_1\wedge dc_1 + da_2\wedge dc_2$,
\[ z_j = \la u, \partial_x^{-1}\partial_{c_j}q\ra = 0 \,. \]
This proves \eqref{eq:Fka}. To complete the argument in the case $|a_1-a_2| \leq A$, we need:

\begin{lemma}  Let $E=\ker \mathcal{K}$, and suppose that $G$ is  a subspace such that $E\perp G$ and the following holds:  
\[u\perp (E+G) \ \Longrightarrow \ \la 
\mathcal{K}u,u \ra \geq \theta \|u\|_{L^2}^2 \,. \]
Then, for any other subspace $F$ we have
\[ u\perp (F+ G) \ \Longrightarrow \ 
\la Ku,u \ra \geq \theta \sin^2\alpha (E,F^\perp) \; \|u\|_{L^2}^2 \,. \]
\end{lemma}
\begin{proof}
Suppose $u\perp (F+G)$ and consider its orthogonal decomposition, 
$u=P_E u + \tilde u$.  Since $E\perp G$ and $u\perp G$, 
we have $\tilde u \perp (E+G)$. Hence, by the hypothesis we have 
$$
\la \mathcal{K}u,u \ra = \la \mathcal{K}\tilde u, \tilde u\ra 
\geq \theta \|\tilde u\|_{L^2}^2 
= \theta \|P_{E^\perp} u\|_{L^2}^2 \,.
$$
An application of \eqref{eq:lsu},
$$\sin \alpha(E,F^\perp) = \inf_{\substack{\|d\|=1 \\ d\in F^\perp}} \|P_{E^\perp} d\|_{L^2} \leq \frac{ \|P_{E^\perp} u\|_{L^2}}{\|u\|_{L^2}}\,, $$
concludes the proof.
\end{proof}

\section{Set-up of the proof}
\label{S:setup}

Recall the definition of  $T_0$ (for given $\delta_0>0$ and ${\bar a}$, ${\bar c}$) stated in the introduction.
Recall
$$B(a,c,t) \defeq \int b(x,t)q^2(x,a,c) \, dx \,.$$
In the next several sections, we establish the key estimates required for the proof of the main theorem. 
 Let us assume that on some time interval $[0,T]$, there are $C^1$ parameters $a(t)\in \mathbb{R}^2$, $c(t)\in \mathbb{R}^2$ such that, if we set
\begin{equation}
\label{E:v-def}
v(\cdot,t) \defeq u(\cdot,t) - q(\cdot,a(t),c(t))
\end{equation}
then the symplectic orthogonality conditions \eqref{E:so} hold.  
Since $u$ solves \eqref{E:pmKdV}, $v(t)$ satisfies
\begin{equation}
\label{E:v}
\partial_t v = \partial_x (-\partial_x^2 v - 6q^2v -6qv^2 - 2v^3 + bv) - F_0
\end{equation}
where $F_0$ results from the perturbation and $\partial_t$ landing
on the parameters:
\begin{equation}
\label{E:19}
F_0 \defeq \sum_{j=1}^2 (\dot a_j - c_j^2)\partial_{a_j}q + \sum_{j=1}^2 \dot c_j \partial_{c_j}q - \partial_x(bq)
\end{equation}
Now decompose
$$F_0 = F_\| + F_\perp$$
where $F_\|$ is symplectically parallel to $M$ and $F_\perp$ is symplectically orthogonal to $M$.  Explicitly,
\begin{equation}
\label{E:F-par}
F_\| = \sum_{j=1}^2 (\dot a_j - c_j^2 + \tfrac12 \partial_{c_j}B)\partial_{a_j}q + \sum_{j=1}^2 (\dot c_j - \tfrac12 \partial_{a_j}B)\partial_{c_j}q
\end{equation}
\begin{equation}
\label{E:F-perp}
F_\perp = - \partial_x(bq) + \frac12 \sum_{j=1}^2 [-(\partial_{c_j}B)\partial_{a_j}q + (\partial_{a_j}B) \partial_{c_j}q]
\end{equation}
All implicit constants will depend upon $\delta_0>0$ and $L^\infty$ norms of $b_0(x,t)$ and its derivatives.  We further assume that 
\begin{equation}
\label{E:c-sep}
\delta_0\leq c_1(t)\leq c_2(t)-\delta_0\leq \delta_0^{-1}
\end{equation}
holds on all of $[0,T]$.

In \S \ref{S:F-perp} we will estimate $ F_\perp $ using the properties
of $ q $ recalled in \S \ref{S:q-properties}. We note that $ F_\| \equiv 0 $
would mean that the parameters solve the effective equations of 
motion \eqref{E:eom}. Hence the estimates on $ F_\| $ are related
to the quality of our effective dynamics and they are provided in 
\S \ref{S:parameters}. In \S \ref{S:correction}
we then construct a correction term which
removes the leading non-homogeneous terms from the equation for $ v$.
Finally energy estimates in \S\ref{S:energy} 
based on the coercivity of $ {\mathcal K} $
lead to the final bootstrap argument in \S\ref{s:bootstrap}.

\section{Estimates on $F_\perp$}
\label{S:F-perp}

Using the identities in Lemma \ref{L:magic}, we will prove that $F_\perp$ is ${\mathcal O}(h^2)$; in fact, we obtain more precise information. 
 For notational convenience, we will drop the $t$ dependence in $b(x,t)$, and will write $b'$, $b''$, $b'''$, to represent $x$-derivatives.

We will use the following consequences of Lemma \ref{L:q-asymp}:
\begin{equation}
\label{E:qa}
\partial_{a_j}q = -\partial_x \eta(\cdot, \hat a_j,c_j) + \Serr
\end{equation}
and
\begin{equation}
\label{E:qc}
c_j\partial_{c_j}q =
\begin{aligned}[t]
&\partial_x[ (x- a_j) \eta(x,\hat a_j,c_j)] 
+ \frac{ 2 c_{3-j} \theta (a_2-a_1) }{(c_1+c_2)(c_1-c_2)} \partial_x \eta(x,\hat a_j,c_j) \\
&- \frac{2 c_j \theta (a_2 - a_1) }{(c_1+c_2)(c_1-c_2)} \partial_x\eta(x,\hat a_{3-j},c_{3-j}) + \Serr \,, 
\end{aligned}
\end{equation}
where $ \theta $ is given by \eqref{eq:the}.

Importantly, as the last formula shows, $\partial_{c_j}q$ is not localized around $\hat a_j$ due to the $c_j$-dependence of $\hat a_{3-j}$.  Also note that it is $(x-a_j)$ and not $(x-\hat a_j)$ in the first term inside the brackets.

\begin{definition}
Let $\mathcal A$ denote the class of functions of $a,c$ that are of the form
$$h^2 \varphi( a_1-a_2, a , c ) + q(a, c )h^3\,, $$
$ a = (a_1, a_2 ) \in \RR^2 $, $ 0 < \delta < c_1 < c_2 - \delta < 1/ \delta 
$, where
\[  \left|  \partial_\alpha^\ell \partial_c^k \partial_a^{p} \varphi ( 
\alpha, a , c ) 
 \right|  \leq 
C \langle \alpha \rangle^{-N}   \,,  \ \ 
  \left| \partial_c^k \partial_a^{p} q ( a, c ) \right| \leq C \,, \]
where $ C $ depends on $ \delta $, $ N$, $ \ell$, $ k $, and $ p $ only.
\end{definition}
We note that if $f\in \Serr$, then $\int f ( x ) dx $ 
has the form $\varphi( a_1-a_2,a , c)$, $ \varphi \in {\mathcal A}$.  The most important feature of the class $\mathcal{A}$ is that for $f\in \mathcal{A}$,  
$$|\partial_{a_j}^k \partial_{c_j}^\ell f | \lesssim h^2 \la a_1-a_2 \ra^{-N} + h^3$$
with implicit constant depending on $c_1$, $c_2$.

\begin{lemma} 
We have
\begin{align}
\label{E:Ba} &\partial_{a_j}B(a,c,\cdot) = 2c_j b'(\hat a_j) + \mathcal{A} \\
\label{E:Bc} &\partial_{c_j}B(a,c,\cdot) = 
\begin{aligned}[t]
&2b(\hat a_j) + 2b'(\hat a_j)(a_j-\hat a_j) - \frac{\pi^2}{12}b''(\hat a_j) c_j^{-2}\\
&- \frac{ 2(-1)^j c_{3-j}(b'(\hat a_2)-b'(\hat a_1))\theta}{(c_1+c_2)(c_1-c_2)} + \mathcal{A}
\end{aligned}
\end{align}
\end{lemma}

\begin{proof}
First we compute $\partial_{a_j}B(a,c,t)$.  We have that $\partial_{a_j} q$ is exponentially localized around $\hat a_j$.  Substituting the Taylor expansion of $b$ around $\hat a_j$, we obtain
\begin{align*}
\partial_{a_j}B(a,c,t) &= 
\begin{aligned}[t]
&b(\hat a_j) \int \partial_{a_j} q^2 + b'(\hat a_j) \int (x-\hat a_j) \partial_{a_j} q^2 \\
&+ \frac12 b''(\hat a_j) \int (x-\hat a_j)^2\partial_{a_j} q^2 + {\mathcal O}(h^3)
\end{aligned}\\
&= \text{I} + \text{II} + \text{III} + {\mathcal O}(h^3)
\end{align*}
Terms I and II are straightforward.  Using \eqref{E:Iq} and \eqref{E:com},
\begin{align*}
&\text{I} = b(\hat a_j) \partial_{a_j} \int q^2 = 0 \\
&\text{II} = b'(\hat a_j)\left( \partial_{a_j} \int x q^2 - \hat a_j\partial_{a_j} \int q^2 \right)= 2c_j b'(\hat a_j)
\end{align*}
For III, we will substitute \eqref{E:qa} and hence pick up ${\mathcal O}(h^2)\la a_1-a_2\ra^{-N}$ errors.
$$\text{III} = -\frac12 b''(\hat a_j) \int (x-\hat a_j)^2 \partial_x \eta^2(x,\hat a_j,c_j) \, dx + \mathcal{A} = \mathcal{A}$$
Thus, we obtain \eqref{E:Ba}.  Next, we compute $\partial_{c_j}B(a,c,t)$.  Note that $\partial_{c_j} q$ is \emph{not} localized around $\hat a_j$.  Begin by rewriting $\partial_{c_j} B$ as
$$\partial_{c_j}B = \int b(\hat a_j) \partial_{c_j} q^2 + \int b'(\hat a_j)(x-\hat a_j) \partial_{c_j} q^2 + \int \tilde b_j \, \partial_{c_j} q^2
$$
where
$$\tilde b_j(x) \defeq b(x)-b(\hat a_j)-b'(\hat a_j)(x-\hat a_j) \,.$$
Now substitute \eqref{E:qc} into the last term and note that the $\Serr$ term in \eqref{E:qc} produces an $\mathcal{A}$ term here.
\begin{align*}
\partial_{c_j}B &= 
\begin{aligned}[t]
&\int b(\hat a_j) \partial_{c_j} q^2 + \int b'(\hat a_j)(x-\hat a_j) \partial_{c_j} q^2 \\
&+ \frac{2}{c_j}\int \tilde b_j(x) \, \partial_x [ (x-a_j)\eta(x,\hat a_j,c_j)] \eta(x,\hat a_j,c_j)\\
&+ \frac{c_{3-j}\theta}{c_j(c_1+c_2)(c_1-c_2)}\int \tilde b_j(x) \partial_x \eta^2 (x,\hat a_j, c_j) \\
&- \frac{\theta}{(c_1+c_2)(c_1-c_2)} \int \tilde b_j(x) \partial_x \eta^2(x,\hat a_{3-j}, c_{3-j})  +\mathcal{A}
\end{aligned} \\
&= \text{I}+\text{II}+\text{III}+\text{IV}+\text{V}  +\mathcal{A}
\end{align*}
where terms I-V are studied separately below.
\begin{align*}
\text{I} &= b(\hat a_j) \partial_{c_j} \int q^2 = 2b(\hat a_j) \\
\text{II} &= b'(\hat a_j) \left( \partial_{c_j} \int xq^2 - \hat a_j \partial_{c_j} \int q^2 \right) \\
&= 2b'(\hat a_j)(a_j-\hat a_j)
\end{align*}
Term III is localized around $\hat a_j$, and thus we integrate by parts in $x$ and Taylor expand $\tilde b_j$ around $\hat a_j$ to obtain
\begin{align*}
\text{III} &=  \frac{1}{c_j}\int \left( -\tilde b_j'(x)(x-a_j) + \tilde b_j(x)\right) \eta^2(x,\hat a_j,c_j) \\
&= -\frac12 \frac{b''(\hat a_j)}{c_j} \int (x-\hat a_j)^2 \eta^2(x,\hat a_j,c_j) \\
&\qquad - b''(\hat a_j)(\hat a_j-a_j) \int (x-\hat a_j) \eta^2(x,\hat a_j,c_j) +{\mathcal O}(h^3) \\
&= -\frac{\pi^2}{12}b''(\hat a_j) c_j^{-2} +{\mathcal O}(h^3)
\end{align*}
Term IV is localized around $\hat a_j$, and thus we integrate by parts in $x$ and Taylor expand $\tilde b_j$ around $\hat a_j$ to obtain
\begin{align*}
 \int \tilde b_j(x)\,\partial_x \eta^2(x,\hat a_j,c_j)
&=  -\int (b'(x)-b(\hat a_j))\eta^2(x,\hat a_j,c_j) \\
&=   -\frac12 b''(\hat a_j) \int (x-\hat a_j)\eta^2(x,\hat a_j,c_j) +{\mathcal O}(h^3)\\
&= {\mathcal O}(h^3)
\end{align*}
Term V is localized around $\hat a_{3-j}$, and thus we integrate by parts in $x$ and Taylor expand $\tilde b_j$ around $\hat a_{3-j}$.
\begin{align*}
\int \tilde b_j(x) \partial_x \eta^2(x,\hat a_{3-j},c_{3-j})
&= - \int (b'(x)-b(\hat a_j))\eta^2(x,\hat a_{3-j},c_{3-j})\\
&= -(b'(\hat a_{3-j}) - b'(\hat a_j)) \int \eta^2(x,\hat a_{3-j},c_{3-j}) \\
&\qquad - b''(\hat a_{3-j}) \int (x-\hat a_{3-j}) \eta^2(x,\hat a_{3-j},c_{3-j})  +{\mathcal O}(h^3)\\
&= -2c_{3-j} (b'(\hat a_{3-j}) - b'(\hat a_j))+{\mathcal O}(h^3)
\end{align*}
\end{proof}

\begin{lemma}[estimates on $F_\perp$]
\begin{equation}
\label{E:F_perp_est}
\partial_x^{-1}\partial_{a_j} F_\perp  = {\mathcal O}(h^2)\cdot \Ssol \,, \qquad \partial_x^{-1}\partial_{c_j} F_\perp  = {\mathcal O}(h^2) \cdot \Ssol  \,, \quad j=1,2 
\end{equation}
\begin{equation}
\label{E:F_perp_expand}
F_\perp = -\frac12\sum_{j=1}^2 \frac{b''(\hat a_j)}{c_j^2} \partial_x \tau(\cdot, \hat a_j, c_j)  + \mathcal{A}\cdot \Ssol 
\end{equation}
where
\begin{equation}
\label{E:tau}
\tau \defeq \left(\frac{\pi^2}{12}+x^2\right)\eta(x) \,, \qquad \tau(x,\hat a_j,c_j) \defeq c_j \tau(c_j(x-\hat a_j)) \,.
\end{equation}
\end{lemma}
In light of the above lemma, we introduce the notation $F_\perp = (F_\perp)_0 + \tilde F_\perp$, where
\begin{equation}
\label{E:F_perp_0}
(F_\perp)_0 =-\frac12 \sum_{j=1}^2 \frac{b''(\hat a_j)}{c_j^2} \partial_x \tau(\cdot, \hat a_j, c_j)
\end{equation}
and $\tilde F_\perp \in \mathcal{A}\cdot \Ssol$. 
We make use of \eqref{E:F_perp_est} in \S\ref{S:parameters} and \eqref{E:F_perp_expand} in \S\ref{S:correction}--\ref{S:energy}.

\begin{proof}
We begin by proving \eqref{E:F_perp_expand}.  By \eqref{E:magic1}, \eqref{E:magic3},
\begin{equation}
\label{E:bq}
\begin{aligned}
\partial_x(bq) &= (\partial_xb) q + b (\partial_xq) \\
& = (\partial_xb) \sum_{j=1}^2 ((x-a_j)\partial_{a_j}q + c_j\partial_{c_j}q) -b\sum_{j=1}^2\partial_{a_j}q \\
& = \sum_{j=1}^2 (-b+(\partial_xb)(x-a_j))\partial_{a_j}q + \sum_{j=1}^2(\partial_x b) c_j\partial_{c_j}q + {\mathcal O}(h^3)\cdot \Ssol
\end{aligned}
\end{equation}
The $\partial_{a_j}q$ term is well localized around $\hat a_j$, and thus we can Taylor expand the coefficients around $\hat a_j$. The $\partial_{c_j}q$ term we leave alone for the moment. 

We have 
$ \partial_x(bq) = \, $
\begin{gather*}
 \sum_{j=1}^2 \Big( -b(\hat a_j) + b'(\hat a_j)(\hat a_j-a_j) 
 + b''(\hat a_j)(\hat a_j-a_j)(x-\hat a_j) + \frac12 b''(\hat a_j) (x-\hat a_j)^2 \Big) \partial_{a_j}q\\
+ \, \sum_{j=1}^2 b'(x) c_j\partial_{c_j}q + \mathcal{A}\cdot \Ssol 
\end{gather*}
Substituting the above together with \eqref{E:Ba} and \eqref{E:Bc} into \eqref{E:F-perp}, we obtain
\begin{gather*} F_\perp = 
\frac12 \sum_{j=1}^2 b''(\hat a_j) \Big( \frac{\pi^2}{12}c_j^{-2}-2(\hat a_j-a_j)(x-\hat a_j) - (x-\hat a_j)^2 \Big) \partial_{a_j}q \\
+ \, \frac{(b'(\hat a_2)-b'(\hat a_1))\theta}{(c_1+c_2)(c_1-c_2)} \sum_{j=1}^2 (-1)^jc_{3-j}\partial_{a_j}q  
 - \sum_{j=1}^2 (b'(x) - b'(\hat a_j))c_j \partial_{c_j}q + \mathcal{A}\cdot \Ssol
\end{gather*}
We now substitute \eqref{E:qa} and  \eqref{E:qc}  recognizing that this will only generate errors of type $\mathcal{A}$ times a Schwartz class function.  We also Taylor expand around $\hat a_j$ or $\hat a_{3-j}$ depending upon the localization. 
$$
F_\perp 
= 
\begin{aligned}[t]
&\frac12 \sum_{j=1}^2 b''(\hat a_j) \Big( -\frac{\pi^2}{12}c_j^{-2}+2(\hat a_j-a_j)(x-\hat a_j) + (x-\hat a_j)^2 \Big) \partial_x\eta(x,\hat a_j,c_j) 
& \leftarrow \text{I} \\
&- \frac{(b'(\hat a_2)-b'(\hat a_1))\theta}{(c_1+c_2)(c_1-c_2)} \sum_{j=1}^2 (-1)^jc_{3-j}\partial_x \eta(x,\hat a_j,c_j) 
& \leftarrow \text{II} \\
& - \sum_{j=1}^2 b''(\hat a_j) (x-\hat a_j) \partial_x[ (x-a_j) \eta(x,\hat a_j, c_j)] 
& \leftarrow \text{III} \\
&-\sum_{j=1}^2 \frac{c_{3-j}\theta}{(c_1+c_2)(c_1-c_2)} b''(\hat a_j)(x-\hat a_j) \partial_x \eta(x,\hat a_j, c_j) 
& \leftarrow \text{IV} \\
&+ \sum_{j=1}^2  \frac{c_j\theta}{(c_1+c_2)(c_1-c_2)}b''(\hat a_{3-j})(x-\hat a_{3-j}) \partial_x \eta(x,\hat a_{3-j},c_{3-j})
& \leftarrow \text{V} \\
&+ \sum_{j=1}^2  \frac{c_j \theta}{(c_1+c_2)(c_1-c_2)}(b'(\hat a_{3-j})-b'(a_j))\partial_x \eta(x,\hat a_{3-j},c_{3-j})
& \leftarrow \text{VI} \\
&+ \mathcal{A}\cdot \Ssol
\end{aligned}
$$
We have that $\text{IV}+\text{V}=0$ and $\text{II}+\text{VI}=0$.  Hence
\begin{align*}
F_\perp &= \text{I}+\text{III} + \mathcal{A}\cdot \Ssol\\
&= -\frac12 \sum_{j=1}^2 b''(\hat a_j) \partial_x \left( \big(\frac{\pi^2}{12}c_j^{-2} + (x-\hat a_j)^2 \big) \eta(x,\hat a_j,c_j) \right)
\end{align*}
This completes the proof of \eqref{E:F_perp_expand}.  To obtain \eqref{E:F_perp_est}, we note that a consequence of \eqref{E:F_perp_expand} is $F_\perp = {\mathcal O}(h^2)f$, where $f\in \Ssol$.  By the definition \eqref{E:F-perp} of $F_\perp$ and Corollary \ref{C:q-inverse}, we have $\partial_x^{-1}F_\perp \in \Ssol$, and hence $f\in \Ssol$.
\end{proof}

\section{Estimates on the parameters}
\label{S:parameters}

\newcommand{\coef}{\operatorname{coef}}

The equations of motion are recovered (in approximate form) using the symplectic orthogonality properties \eqref{E:so} of $v$ and the equation \eqref{E:v} for $v$.  For a function $G$ of the form
$$G = g_1 \partial_{a_1}q + g_2 \partial_{a_2}q + g_3 \partial_{c_1}q + g_4 \partial_{c_2}q$$
with $g_j= g_j(a,c)$, define
$$\coef(G) = (g_1,g_2,g_3,g_4) \,.$$

\begin{lemma}
\label{L:par-est}
Suppose we are given $\delta_0>0$ and $b_0(x,t)$, and parameters $a(t)$, $c(t)$ such that $v$ defined by \eqref{E:v-def} satisfies the symplectic orthogonality conditions \eqref{E:so}.  Suppose, moreover, that the amplitude separation condition \eqref{E:c-sep} holds.  Then (with implicit constants depending upon $\delta_0>0$ and $L^\infty$ norms of $b_0$ and its derivatives), if $\|v\|_{H^2} \lesssim 1$, then we have
\begin{equation}
\label{E:F-par-est}
 | \coef(F_\|) | \lesssim  h^2 \|v\|_{H^1} + \|v\|_{H^1}^2 \,.
\end{equation}
\end{lemma}

\begin{proof}
Since $\la v, \partial_x^{-1}\partial_{a_j}q\ra =0$, we have upon substituting \eqref{E:v}
\begin{align*}
0 &= \partial_t \la v, \partial_x^{-1}\partial_{a_j}q \ra \\
&= \la \partial_t v, \partial_x^{-1}\partial_{a_j}q\ra + \la v, \partial_t \partial_x^{-1}\partial_{a_j}q \ra \\
&= 
\begin{aligned}[t]
&\la (\partial_x^2 v + 6q^2)v, \partial_{a_j}q\ra + \la (6qv^2+2v^3), \partial_{a_j}q \ra 
&& \leftarrow \text{I}+\text{II}\\
&- \la bv, \partial_{a_j}q \ra - \la F_\|, \partial_x^{-1}\partial_{a_j}q \ra -\la F_\perp, \partial_x^{-1} \partial_{a_j}q \ra
&& \leftarrow \text{III}+\text{IV}+\text{V}\\
& + \la v, \partial_x^{-1} \partial_{a_j}\left( \sum_{k=1}^2 \partial_{a_k} q \, \dot a_k +  \sum_{k=1}^2 \partial_{c_k}q\, \dot c_k \right) \ra 
&& \leftarrow \text{VI}\\
\end{aligned}
\end{align*}
We have, by \eqref{E:magic2},
\begin{align*}
\text{I} &= \la v, \partial_{a_j} (\partial_x^2q + 2q^3) \ra \\
&= -\frac12 \la v, \partial_x^{-1} \partial_{a_j}\partial_x I_3'(q) \ra \\
&=- \la v, \partial_x^{-1}\partial_{a_j} \sum_{k=1}^2 c_k^2 \partial_{a_k}q \ra
\end{align*}
Also, by \eqref{E:F-perp}
\begin{align*}
\text{III} &= -\la bv, \partial_{a_j}q \ra\\
&= -\la v, \partial_{a_j} (bq) \ra\\
&= -\la v, \partial_x^{-1} \partial_{a_j} \partial_x (bq) \ra\\
&= -\la v, \partial_x^{-1}\partial_{a_j} \big( -F_\perp - \tfrac12 \sum_{k=1}^2 (\partial_{c_k}B) \partial_{a_k}q + \tfrac12 \sum_{k=1}^2 (\partial_{a_k}B)\partial_{c_k}q \big) \ra 
\end{align*}
Thus
\begin{align*}
|\text{I}+\text{III}+\text{VI}| &=| \la v, \partial_x^{-1}\partial_{a_j} F_\perp \ra +\la v, \partial_x^{-1}\partial_{a_j} F_\| \ra| \\
&\leq \|v\|_{L^2}( \|\partial_x^{-1}\partial_{a_j}F_\perp\|_{L^2} + \|\partial_x^{-1}\partial_{a_j} F_\| \|) \\
&\leq \|v\|_{L^2}(h^2 + |\coef(F_\|)|)
\end{align*}
Next, we note that by Cauchy-Schwarz,
$$|\text{II}| \lesssim \|v\|_{H^1}^2 \,.$$
Next, observe from \eqref{E:F-par} and Lemma \ref{E:sf} that
$$\text{IV} = \la F_\|, \partial_x^{-1}\partial_{a_j}q \ra = -(\dot c_j-\frac12\partial_{a_j}B) \,.$$
Of course, we have $\text{V} =\la F_\perp, \partial_x^{-1} \partial_{a_j}q \ra= 0$.  Combining, we obtain
\begin{equation}
\label{E:F-par-1st}\left|\dot c_j - \frac12 \partial_{a_j}B \right| \lesssim \|v\|_{H^1}(h^2+|\coef(F_\|)|)+ \|v\|_{H^1}^2 \,.
\end{equation}
A similar calculation, applying $\partial_t$ to the identity $0=\la v, \partial_x^{-1}\partial_{c_j}q\ra$, yields
\begin{equation}
\label{E:F-par-2nd}
\left| \dot a_j - c_j^2 + \frac12\partial_{c_j}B \right| \lesssim \|v\|_{H^1}(h^2+|\coef(F_\|)|)+ \|v\|_{H^1}^2 \,.
\end{equation}
Combining \eqref{E:F-par-1st} and \eqref{E:F-par-2nd} gives \eqref{E:F-par-est}.
\end{proof}

\section{Correction term}
\label{S:correction}

Recall the definition \eqref{E:tau} of $\tau$.  Let $\rho$ be the unique function solving 
$$(1-\partial_x^2 - 6\eta^2) \rho = \tau \,,$$
see \cite[Proposition 4.2]{HZ2} for the properties of this equation.
The function $\rho$ is smooth, exponentially decaying at $\infty$, and satisfies the symplectic orthogonality conditions
\begin{equation}
\label{E:rho-so}
\la \rho, \eta\ra =0 \,, \qquad \la \rho, x\eta \ra =0
\end{equation}
Set
$$\rho(x,\hat a_j, c_j) \defeq c_j^{-1}\rho(c_j(x-\hat a_j))$$
and note that
$$(c_j^2 - \partial_x^2 - 6\eta^2(\cdot ,\hat a_j, c_j)) \rho(\cdot, \hat a_j, c_j) = \tau(\cdot , \hat a_j, c_j)$$
Define the symplectic projection operator
$$Pf \defeq \sum_{j=1}^2 \la f, \partial_x^{-1}\partial_{c_j}q\ra \partial_{a_j}q + \sum_{j=1}^2 \la f, \partial_x^{-1}\partial_{a_j}q \ra \partial_{c_j}q \,.$$
Define
\begin{equation}
\label{E:w-def}
w \defeq -\frac12 (I-P)\sum_{j=1}^2 \frac{b''(\hat a_j)}{c_j^2} \rho(\cdot, \hat a_j, c_j)
\end{equation}
Note that $w = {\mathcal O}(h^2)$ and clearly now $w$ satisfies 
\begin{equation}
\label{E:so-w}
\la w, \partial_x^{-1}\partial_{a_j}q \ra =0 \,, \qquad 
\la w, \partial_x^{-1}\partial_{c_j}q \ra =0 \,.
\end{equation} 
Recall the definition \eqref{E:F_perp_0} of $(F_\perp)_0$.
\begin{lemma}
If  $\dot c_j = {\mathcal O}(h)$, and $\dot a_j = c_j - b(\hat a_j) + {\mathcal O}(h)$, then
\begin{equation}
\label{E:w-flow}
\partial_t w + \partial_x(\partial_x^2 w + 6q^2 w - bw) = -(F_\perp)_0 - G+ \mathcal{A}\cdot \Ssol \,.
\end{equation}
where $G$ is an ${\mathcal O}(h^2)$ term that is symplectically parallel to $M$, i.e. 
$$G\in \spn \{ \partial_x^{-1}\partial_{a_1}q, \partial_x^{-1}\partial_{a_2}q, \partial_x^{-1}\partial_{c_1}q, \partial_x^{-1}\partial_{c_2}q \} \,.$$
\end{lemma}
\begin{proof} 
Let 
$$w_j = \frac{b''(\hat a_j)}{c_j^2} \rho(\cdot, \hat a_j, c_j)$$
Then
\begin{align*}
\partial_t w_j 
&= 
\begin{aligned}[t]
& b'''(\hat a_j) \dot{\hat{a}}_j c_j^{-2}  \rho(\cdot, \hat a_j, c_j) -2 b''(\hat a_j) c_j^{-3}\dot c_j  \rho(\cdot, \hat a_j, c_j) \\
&+ b''(\hat a_j) c_j^{-2} \dot{\hat a}_j \partial_{a_j} \rho (\cdot, \hat a_j, c_j) + b''(\hat a_j) c_j^{-2} \dot{\hat c}_j \partial_{c_j} \rho (\cdot, \hat a_j, c_j) + \partial_t b''(\hat a_j)c_j^{-2} \rho(\cdot, \hat a_j, c_j)
\end{aligned} \\
&= -\dot a_j \partial_x w_j  + \mathcal{A}\cdot \Ssol
\end{align*}
Also, we have
\begin{align*}
(\partial_x^2 + 6q^2)w_j 
&= (\partial_x^2 + 6\eta^2(\cdot, \hat a_j, c_j))w_j + \mathcal{A}\cdot \Ssol \\
&= c_j^2 w_j - b''(\hat a_j) c_j^{-2}\tau(\cdot, \hat a_j, c_j) + \mathcal{A}\cdot \Ssol
\end{align*}
Also,
$$b w_j = b(\hat a_j) w_j + \mathcal{A}\cdot \Ssol$$
Combining, we obtain
\begin{align*}
\indentalign \partial_t w_j + \partial_x (\partial_x^2 w_j + 6q^2 w_j - bw_j)  \\
&=  - b''(\hat a_j) c_j^{-2}\partial_x \tau(\cdot, \hat a_j, c_j) + (-\dot a_j +c_j^2 - b(\hat a_j))\partial_x w_j + \mathcal{A}\cdot \Ssol \\
&=  - b''(\hat a_j) c_j^{-2}\partial_x \tau(\cdot, \hat a_j, c_j) + \mathcal{A}\cdot \Ssol
\end{align*}

Now we discuss $\partial_t P w_j$. 
$$\partial_t Pw_j = 
\begin{aligned}[t]
&\la \partial_t w_j, \partial_x^{-1}\partial_{a_1}q \ra \partial_{c_1}q + \la w_j, \partial_t \partial_x^{-1}\partial_{a_1} q \ra \partial_{c_j}q + \text{similar} \\
&+ \la w_j, \partial_x^{-1} \partial_{a_1} q \ra \partial_t \partial_{c_1}q + \text{similar}
\end{aligned}
$$
The first line of terms is symplectically parallel to $M$.  For the second line,  note that by \eqref{E:rho-so}, we have $\la w_j, \partial_x^{-1}\partial_{a_1}q\ra =\mathcal{A}$.  Consequently,
$$\partial_tP w_j = T_qM + \mathcal{A}\cdot\Ssol$$
\end{proof}

Define $\tilde u$ and $\tilde v$ by
\begin{equation}
\label{E:tildeu-def}
u = \tilde u + w \,, \qquad v = \tilde v + w \,.
\end{equation}
Of course, it follows that $\tilde u = q+\tilde v$.  Note that by \eqref{E:so} and \eqref{E:so-w}, we have
\begin{equation}
\label{E:so-tilde}
\la \tilde v,\partial_x^{-1}\partial_{a_j}q \ra =0 \text{ and } \la \tilde v, \partial_x^{-1}\partial_{c_j}q\ra =0 \,, \; j=1,2 \,.
\end{equation}
Note that $\tilde u$ solves
\begin{equation}
\label{E:u-tilde-1}
\partial_t \tilde u = - \partial_x (\partial_x^2 \tilde u + 2\tilde u^3 - b\tilde u) - \partial_t w - \partial_x(\partial_x^2 w +6\tilde u^2 w - bw) + {\mathcal O}(h^4)
\end{equation}
where the ${\mathcal O}(h^4)$ terms arise from $w^2$ and $w^3$.  Moreover, if we make the mild assumption that $\tilde v = {\mathcal O}(h)$, then $\tilde u^2 w = q^2 w + {\mathcal O}(h^3)$.
By \eqref{E:u-tilde-1} and \eqref{E:w-flow}, we have
\begin{equation}
\label{E:u-tilde}
\partial_t \tilde u = - \partial_x (\partial_x^2 \tilde u + 2\tilde u^3 - b\tilde u)+ (F_\perp)_0+G + \mathcal{A}\cdot \Ssol
\end{equation}
Since $\tilde u = q+\tilde v$, we have (in analogy with \eqref{E:v})
\begin{equation}
\label{E:v-tilde}
\partial_t \tilde v = \partial_x (-\partial_x^2 \tilde v - 6q^2\tilde v  + b\tilde v) - F_\| - \tilde F_\perp +G+ \mathcal{A}\cdot \Ssol + O(h^3)H^1
\end{equation}
where we have made the assumption that $\tilde v = {\mathcal O}(h^{3/2})$ in order to discard the $\tilde v^2$ and $\tilde v^3$ terms.   We thus see that, in comparison to $v$, the equation for $\tilde v$ has a lower-order inhomogeneity, but still satisfies the symplectic orthogonality conditions \eqref{E:so-tilde} and $v=\tilde v +{\mathcal O}(h^2)$.  

\section{Energy estimate}
\label{S:energy}

Since $w = {\mathcal O}(h^2)$, to obtain the desired bound on $v$ it will suffice to obtain a bound for $\tilde v$.  This will be achieved by the ``energy method.''

\begin{lemma}
\label{L:energy}
Suppose we are given $\delta_0>0$ and $b_0(x,t)$, and parameters $a(t)$, $c(t)$ such that $v$ defined by \eqref{E:v-def} satisfies the symplectic orthogonality conditions \eqref{E:so} on $[0,T]$.  Suppose, moreover, that the amplitude separation condition \eqref{E:c-sep} holds on $[0,T]$.  Then (with implicit constants depending upon $\delta_0>0$ and $L^\infty$ norms of $b_0$ and its derivatives), if $\|v\|_{H^2} \lesssim 1$ and $T\ll  h^{-1}$, then
$$\|v\|_{L_{[0,T]}^\infty H^2}^2 \lesssim \|v(0)\|_{H^2}^2 + h^4\left(1+\int_0^T \la a_1-a_2\ra^{-N} \,dt\right)^2 \,.$$
\end{lemma}

\begin{proof}
Recall that we have defined
$$H_c(u) = I_5(u) + (c_1^2+c_2^2)I_3(u) + c_1^2c_2^2I_1(u) \,.$$
With $w$ given by \eqref{E:w-def} and $\tilde u$ given by \eqref{E:tildeu-def},  let
$$\mathcal{E}(t) = H_c(\tilde u) - H_c(q) \,.$$
Then
\begin{align*}
\partial_t \mathcal{E} &= 
\begin{aligned}[t]
&\la H_c'(\tilde u), \partial_t \tilde u \ra - \la H_c'(q), \partial_t q \ra + 2(c_1\dot c_1+c_2\dot c_2)(I_3(\tilde u)-I_3(q)) \\
&+ 2c_1c_2(c_1\dot c_2+\dot c_1c_2)(I_1(\tilde u)-I_1(q))
\end{aligned} \\
&= \text{I}+\text{II}+\text{III}+\text{IV}
\end{align*}
Note that $\text{II}=0$ since Lemma \ref{l:E7} 
showed that $H_c'(q)=0$.  For III, we have by \eqref{E:magic2} and the orthogonality conditions \eqref{E:so-tilde},
\begin{align*}
\text{III} &= 2(c_1\dot c_1+c_2\dot c_2)( \la I_3'(q),\tilde v \ra + {\mathcal O}(\|\tilde v\|_{H^1}^2))\\
&=  4(c_1\dot c_1+c_2\dot c_2) \la \sum_{j=1}^2 c_j^2 \partial_x^{-1}\partial_{a_j}q, \tilde v \ra + {\mathcal O}((|\dot c_1|+|\dot c_2|) \|\tilde v\|_{H^1}^2) \\
&=  {\mathcal O}((|\dot c_1|+|\dot c_2|) \|\tilde v\|_{H^1}^2)
\end{align*}
Term IV is bounded similarly.  It remains to study Term I.  Writing \eqref{E:u-tilde} as $\partial_t \tilde u = \frac12 \partial_x I_3'(\tilde u) + \partial_x (b\tilde u)+(F_\perp)_0+G+\mathcal{A}\cdot \Ssol$ and appealing to \eqref{E:conserved}, we have by  Lemma \ref{L:near-conserved} (with $u$ replaced by $\tilde u$ in that lemma) that
\begin{align*}
\text{I} &=  \la H_c'(\tilde u), \partial_x(b\tilde u) \ra + \la  H_c'(\tilde u), (F_\perp)_0 +\mathcal{A}\cdot \Ssol\ra \\
&=  
\begin{aligned}[t]
&5 \la b_x, A_5(\tilde u) \ra - 5\la b_{xxx}, A_3(\tilde u) \ra 
+\la b_{xxxxx}, A_1(\tilde u) \ra \\
&+(c_1^2+c_2^2) ( 3\la b_x, A_3(\tilde u) \ra - \la b_{xxx}, A_1(\tilde u) \ra) + c_1^2c_2^2 \la b_x, A_1(\tilde u) \ra \\
&+ \la  H_c'(\tilde u), (F_\perp)_0 +\mathcal{A}\cdot \Ssol\ra
\end{aligned}
\end{align*}
Expand $A_j(\tilde u) =A_j(q+\tilde v)=A_j(q)+A_j'(q)(\tilde v) + {\mathcal O}(\tilde v^2)$
and $H_c'(\tilde u) = H_c'(q)+\mathcal{K}_{c,a}\tilde v + {\mathcal O}(\tilde v^2)=\mathcal{K}_{c,a}\tilde v + {\mathcal O}(\tilde v^2)$ to obtain $\text{I}=\text{IA}+\text{IB}+\text{IC}$, where
\begin{align*}
&\text{IA} = 
\begin{aligned}[t]
&5 \la b_x, A_5(q) \ra - 5\la b_{xxx}, A_3(q) \ra 
+\la b_{xxxxx}, A_1(q) \ra \\
&+(c_1^2+c_2^2) ( 3\la b_x, A_3(q) \ra - \la b_{xxx}, A_1(q) \ra) + c_1^2c_2^2 \la b_x, A_1(q) \ra 
\end{aligned} \\
&\text{IB} = 
\begin{aligned}[t]
&5 \la b_x, A_5'(q)(\tilde v) \ra - 5\la b_{xxx}, A_3'(q)(\tilde v) \ra 
+\la b_{xxxxx}, A_1'(q)(\tilde v) \ra \\
&+(c_1^2+c_2^2) ( 3\la b_x, A_3'(q)(\tilde v) \ra - \la b_{xxx}, A_1'(q)(\tilde v) \ra) + c_1^2c_2^2 \la b_x, A_1'(q)(\tilde v) \ra 
\end{aligned} \\
&\text{IC} = 
\begin{aligned}[t]
& \la  \mathcal{K}_{c,a}\tilde v, (F_\perp)_0 \ra + {\mathcal O}(h \|\tilde v\|_{H^2}^2) + {\mathcal O}(\mathcal{A}\cdot \|\tilde v\|_{H^2})
\end{aligned} 
\end{align*}
Then reapply Lemma \ref{L:near-conserved} (with $u$ replaced by $q$ in that lemma) to obtain that $\text{IA} = - \la H_c'(q), \partial_x(bq) \ra =0$.  Applying Lemma \ref{L:near-conserved2},
\begin{align*}
\text{IB} &= \la \mathcal{K}_{c,a}\tilde v, (bq)_x \ra - \la \partial_xH_c'(q),b\tilde v\ra\\
&=\la \mathcal{K}_{c,a}\tilde v, (bq)_x \ra
\end{align*}
In summary thus far, we have obtained that 
$$\partial_t \mathcal{E} =  \la \mathcal{K}_{c,a}\tilde v, (bq)_x +(F_\perp)_0\ra + {\mathcal O}( h \|\tilde v\|_{H^2}^2)+{\mathcal O}(\mathcal{A}\|\tilde v\|_{H^2})$$
By \eqref{E:Kqa}, \eqref{E:Kqc}, and \eqref{E:so-tilde} (recalling the definition \eqref{E:19} of $F_0$), we obtain
$$\la \mathcal{K}_{c,a} \tilde v, \partial_x(bq)\ra = - \la \mathcal{K}_{c,a}\tilde v, F_0\ra = -\la \mathcal{K}_{c,a} \tilde v, F_\| + F_\perp \ra $$
Hence
$$\partial_t \mathcal{E} =  -\la \mathcal{K}_{c,a}\tilde v, F_\|+\tilde F_\perp \ra + {\mathcal O}( h \|\tilde v\|_{H^2}^2)+{\mathcal O}(\mathcal{A}\|\tilde v\|_{H^2})$$
It follows from Lemma \ref{L:par-est} and $\tilde F_\perp\in \mathcal{A}\cdot \Ssol$ (see \eqref{E:F_perp_expand}, \eqref{E:F_perp_0}) that
$$|\partial_t \mathcal{E}| \lesssim (h^2\la a_1-a_2\ra^{-N}+h^3)\|\tilde v\|_{H^2}+ h\| \tilde v\|_{H^2}^2$$
If $T = \delta h^{-1}$,
$$\mathcal{E}(T)=\mathcal{E}(0) + h^2\left(1+\int_0^T \la a_1-a_2\ra^{-N}\right)\|\tilde v\|_{L_{[0,T]}^\infty H_x^2} + h \| \tilde v\|_{L_{[0,T]}^\infty H^2}^2 \,.$$
By Lemma \ref{l:E7}, the definition of $\mathcal{E}$ and $\mathcal{K}_{c,a}$, and the fact that $\tilde u = q+\tilde v$, we have
$$|\mathcal{E}-\la \mathcal{K}_{c,a}\tilde v, \tilde v \ra | \lesssim \|\tilde v\|_{H^2}^3 \,.$$
Applying this at time $0$ and $T$, together with the coercivity of $\mathcal{K}$ (Proposition \ref{p:ck}),
$$\|\tilde v(T)\|_{H^2}^2 \lesssim \|\tilde v(0)\|_{H^2}^2 + h^2\left(1+\int_0^T \la a_1-a_2\ra^{-N}\right)\|\tilde v\|_{L_{[0,T]}^\infty H_x^2} + h \| \tilde v\|_{L_{[0,T]}^\infty H^2}^2 \,.$$
Replacing $T$ by $T'$ such that $0\leq T'\leq T$, and taking the supremum in $T'$ over $0\leq T'\leq T$, we obtain
$$\|\tilde v\|_{L_{[0,T]}^\infty H^2}^2 \lesssim \|\tilde v(0)\|_{H^2}^2 + h^2\left(1+\int_0^T \la a_1-a_2\ra^{-N}\right)\|\tilde v\|_{L_{[0,T]}^\infty H_x^2} + h \| \tilde v\|_{L_{[0,T]}^\infty H^2}^2 \,.$$
By selecting $\delta$ small enough, we obtain
$$\|\tilde v\|_{L_{[0,T]}^\infty H^2}^2 \lesssim \|\tilde v(0)\|_{H^2}^2 + h^4\left(1+\int_0^T \la a_1-a_2\ra^{-N} \,dt\right)^2$$
Finally, using that $\|w\|_{H^2} \sim h^2$, and $v=\tilde v+w$, we obtained the claimed estimate.
\end{proof}

\section{Proof of the main theorem}
\label{s:bootstrap}

We start with the proposition which links the ODE analysis with 
the estimates on the error term $ v$:
\begin{proposition}
\label{P:onestep}
Suppose we are given $b_0\in C_b^\infty(\mathbb{R}^2)$ and $\delta_0>0$.  (Implicit constants below depend only on $b_0$ and $\delta_0$).   Suppose that we are further given ${\bar a}\in \mathbb{R}^2$, $\bar c\in \mathbb{R}^2\backslash \mathcal{C}$, $\kappa \geq 1$, $h>0$, and $v_0$ satisfying \eqref{E:so}, such that
$$0< h \lesssim \kappa^{-1} \,, \qquad \|v_0\|_{H_x^2} \leq \kappa h^2 \,.$$  Let $u(t)$ be the solution to \eqref{E:pmKdV} with $b(x,t)=b_0(hx,ht)$ and initial data $\eta(\cdot, {\bar a},{\bar c})+v_0$.   Then there exist a time $T'>0$ and trajectories $a(t)$ and $c(t)$ defined on $[0,T']$ such that $a(0)={\bar a}$, $c(0)={\bar c}$ and the following holds, with $v\defeq u-\eta(\cdot, a,c)$:
\begin{enumerate}
\item On $[0,T']$, the orthogonality conditions \eqref{E:so} hold.
\item \label{I:time-scale}
Either $c_1(T')=\delta_0$, $c_1(T')=c_2(T')-\delta_0$, $c_2(T')=\delta_0^{-1}$, or $T' = \omega h^{-1}$, where $\omega \ll 1$.
\item $|\dot a_j - c_j^2 +b(a_j,t)| \lesssim h$.
\item $|\dot c_j - c_jb'(a_j) | \lesssim h^2$.
\item \label{I:en-bd}
$\|v\|_{L_{[0,T']}^\infty H_x^2} \leq \alpha \kappa h^2 \,,$ where $\alpha \gg 1$.
\end{enumerate}
Here $\alpha$ and $\omega$ are constants depending only on $b_0$ and $\delta_0$ (independent of $\kappa$, etc)
\end{proposition}

\begin{proof}
Recall our convention that implicit constants depend only on $b_0$ and $\delta$.
By Lemma \ref{L:decomp} and the continuity of the flow $u(t)$ in $H^2$, there exists some $T''>0$ on which $a(t)$, $c(t)$ can be defined so that \eqref{E:so} hold.  Now take $T''$ to be the maximal time on which $a(t)$, $c(t)$ can be defined so that \eqref{E:so} holds.  Let $T'$ be first time $0\leq T' \leq T''$ such that 
$c_1(T')=\delta_0$,  
$c_1(T')=c_2(T')-\delta_0$,
$c_2(T')=\delta_0^{-1}$,
$T'=T''$, or $\omega h^{-1}$
(whichever comes first).  Here, $0 < \omega \ll 1$ is a constant that we will chosen suitably small at the end of the proof (depending only upon implicit constants in the estimates, and hence only on $b_0$ and $\delta$).

\begin{remark}
\label{R:bs1}
We will show that on $[0,T']$, we have $\|v(t)\|_{H_x^2}\lesssim \kappa h^2$, and hence by Lemma \ref{L:decomp} and the continuity of the $u(t)$ flow, it must be the case that either $c_1(T')=\delta_0$,  
$c_1(T')=c_2(T')-\delta_0$,
$c_2(T')=\delta_0^{-1}$, 
or $\omega h^{-1}$
(i.e. the case $T'=T''$ does not arise).
\end{remark}

Let $T$, $0<T\leq T'$, be the maximal time such that
\begin{equation}
\label{E:bs1}
\|v\|_{L_{[0,T]}^\infty H_x^2} \leq \alpha \kappa h^2 \,,
\end{equation}
where $\alpha$ is suitably large constant related to the implicit constants in the estimates (and thus dependent only upon $b_0$ and $\delta_0>0$).

\begin{remark}
\label{R:bs2}
We will show, assuming that \eqref{E:bs1}  holds, that $\|v\|_{L_{[0,T]}^\infty H_x^1} \leq \frac12 \alpha \kappa h^{1/2}$ and thus by continuity we must have $T = T'$.
\end{remark}

In the remainder of the proof, we work on the time interval $[0,T]$, and we are able to assume that the orthogonality conditions \eqref{E:so} hold,  $\delta_0 \leq c_1(t) \leq c_2(t)-\delta_0 \leq \delta_0^{-1}$, and that \eqref{E:bs1} holds.  
By Lemma \ref{L:par-est} and Taylor expansion, we have (since $\kappa^2h^4 \lesssim h^2$)
\begin{equation}
\label{eq:oded}
\left\{
\begin{aligned}
&\dot a_j = c_j^2 - b(a_j,t) + {\mathcal O}(h)  \\
&\dot c_j = c_j\partial_xb(a_j,t)+{\mathcal O}(h^2)\,,
\end{aligned}
\right.
\end{equation}
with initial data $a_j(0)=\bar a_j$, $c_j(0)=\bar c_j$.  Let
$$\xi(t) \defeq \frac{b(a_1(t),t)-b(a_2(t),t)}{a_1(t)-a_2(t)}$$
and let $\Xi(t)$ denote an antiderivative.
By the mean-value theorem $|\xi| \lesssim h$, and since $T\leq \omega h^{-1}$, we have $e^\Xi \sim 1$.  We then have
$$\frac{d}{dt}\left(e^{\Xi}( a_2 - a_1 )\right) = e^\Xi(c_2^2-c_1^2) + {\mathcal O}(h)\,.$$
Since $\delta_0^2 \leq c_2^2-c_1^2$, we see that $e^{\Xi}( a_2 - a_1 )$ is strictly increasing.  Let $0\leq t_1\leq T$ denote the unique time at which $e^{\Xi}( a_2 - a_1 )=0$ (if the quantity is always positive, take $t_1=0$, and if the quantity is always negative, take $t_1=T$, and make straightforward modifications to the argument below).  If $t<t_1$, integrating from $t$ to $t_1$ we obtain
$$ \delta_0^2(t_1-t) \lesssim -e^{\Xi(t)}(a_2(t)-a_1(t))= e^{\Xi(t)}|a_2(t)-a_1(t)|$$
If $t>t_1$, integrating from $t_1$ to $t$ we obtain
$$\delta_0^2(t-t_1) \lesssim e^{\Xi(t)}(a_2(t)-a_1(t))\,.$$
Hence,
$$\int_0^T \la a_2(t)-a_1(t) \ra^{-2} \lesssim 1 \,.$$
By Lemma \ref{L:energy}, we conclude that
$$\|v\|_{L_T^\infty H_x^2} \leq \frac{\alpha}{4}(\|v(0)\|_{H^2} + h^2) \leq \frac{\alpha}{4}(\kappa h^2+h^2) \leq \frac12\alpha \kappa h^2 \,.$$
\end{proof}

We can now complete
\begin{proof}[Proof of the main Theorem]
Suppose that $\|v_0\|_{H^2} \leq h^2$. 
Iterate Prop. \ref{P:onestep}, as long as the condition 
\begin{equation}
\label{E:c-constraint}
\delta_0 \leq  c_1\leq c_2-\delta_0 \leq \delta_0^{-1}
\end{equation} 
remains true, as follows:  for the $k$-th iterate, put $\kappa = \alpha^k$ in Prop. \ref{P:onestep} and advance from time $t_k=k\omega h^{-1}$ to time $t_{k+1}=(k+1)\omega h^{-1}$.   At time $t_k$, we have $\|v(t_k)\|_{H^2} \leq \alpha^k h^2$, and we find from Prop. \ref{P:onestep} that $\|v\|_{L_{[t_k,t_{k+1}]}^\infty H_x^2} \leq \alpha^{k+1}h^2$.  Provided \eqref{E:c-constraint} holds on all of $[0,t_K]$, we can continue until $\kappa^{-1}\sim h$, i.e. $K\sim \log h^{-1}$.   

Recall \eqref{eq:T0h}, 
and $A_j(T)$, $C_j(T)$ defined by
\eqref{eq:ODEsh}.  Let $\hat a_j(t)=h^{-1}A_j(ht)$, $\hat c_j(t)=C_j(ht)$.  Then $\hat a_j$, $\hat c_j$ solve 
$$
\left\{
\begin{aligned}
&\dot{\hat a}_j = \hat{c}_j^2 - b(\hat a_j,t) \\
&\dot{\hat c}_j = \hat c_j \partial_x b(\hat a_j,t)
\end{aligned}
\right.
$$
with initial data $\hat a_j(0) = \bar a_j$, $\hat c_j(0) = \bar c_j$.  We know that \eqref{E:c-constraint} holds for $\hat c_j$ on $[0,h^{-1}T_0]$.
Let $\tilde a_j = a_j-\hat a_j$, $\tilde c_j= c_j - \hat c_j$ denote the differences.  Let
$$\gamma(t) \defeq \frac{b(a_j,t)-b(\hat a_j,t)}{a_j-\hat a_j}$$
$$\sigma(t) \defeq \frac{\partial_x b(a_j,t)-\partial_x b(\hat a_j,t)}{a_j-\hat a_j}\,.$$
By the mean-value theorem, $|\gamma(t)|\lesssim h$ and $|\sigma(t)|\lesssim h^2$.
We have
\begin{equation}
\label{eq:ODEt}
\left\{
\begin{aligned}
&\dot{\tilde a}_j = \tilde c_j^2 + 2\hat c_j \tilde c_j -\gamma \tilde a_j + {\mathcal O}(h) \\
&\dot{\tilde c}_j = \tilde c_j (\partial_x b)(a_j,t) + \hat c_j \sigma \tilde a_j +{\mathcal O}(h^2) \,.
\end{aligned}
\right.
\end{equation}
We conclude that $|\tilde a_j| \lesssim e^{Cht}$ and $|\tilde c_j| \lesssim he^{Cht}$.  This is proved by Gronwall's method and a bootstrap argument.  Since \eqref{E:c-constraint} holds for $\hat c_j$ on $[0,h^{-1}T_0]$, it holds for $c_j$ on the same time scale if $T_0<\infty$, and up to the maximum time allowable by the above iteration argument, $\epsilon h^{-1}\log h^{-1}$, if $T_0=+\infty$.
\end{proof}

\appendix

\section{Local and global well-posedness}
\label{A:gwp}

\renewcommand\thefootnote{\ddag}%

In this appendix, we will prove that \eqref{E:pmKdV} is globally well-posed in $H^k$, $k\geq 1$ provided 
\begin{equation}
\label{E:b-size}
M(T) \defeq \sum_{j=0}^{k+1} \|\partial_x^j b(x,t) \|_{L_{[0,T]}^\infty L_x^\infty} < \infty \,.
\end{equation}
for all $T>0$.
This is proved for $k=1$ under the additional assumption that  $\|b\|_{L_x^2L_T^\infty} < \infty$ in the appendix of Dejak-Sigal \cite{DS}. \footnote{It is further assumed in \cite{DS} that $\|b\|_{L_T^\infty L_x^\infty}$ is small, although this appears to be unnecessary in their argument.}  The removal of the assumption  $\|b\|_{L_x^2L_T^\infty} < \infty$ is convenient since it allows for us to consider potentials that asymptotically in $x$ converge to a nonzero number, rather than decay. Moreover, our argument is self-contained.

Well-posedness for KdV (nonlinearity $\partial_x u^2$) with $b\equiv 0$ was obtained by Bona-Smith \cite{BS} via the energy method, using the vanishing viscosity technique for construction and a regularization argument for uniqueness.  Although their argument adapts to include $b\neq 0$ and to mKdV \eqref{E:pmKdV}, it applies only for $k > \frac{3}{2}$ due to the derivative in the nonlinearity.  Kenig-Ponce-Vega \cite{KPV, KPV2} reduced the regularity requirements (for $b\equiv 0$) below $k=1$ by introducing new local smoothing and maximal function estimates and applying the contraction method.   These estimates were obtained by Fourier analysis (Plancherel's theorem, van der Corput lemma).  At the $H^1$ level of regularity (and above) for mKdV, the full strength of the maximal function estimate in \cite{KPV,KPV2} is not needed.  Here, we prove a local smoothing estimate and a (weak) maximal function estimate (see \eqref{E:locsmooth} and \eqref{E:maximal} in Lemma \ref{L:locsmooth} below)  instead by the integrating factor method, which easily accomodates the inclusion of a potential term since integration by parts can be applied.   The estimates proved by Kenig-Ponce-Vega were directly applied by Dejak-Sigal, treating the potential term as a perturbation, which required introducing the norm $\|b\|_{L_x^2 L_T^\infty}$.  Our argument does not apply directly to KdV since we are lacking the (strong) maximal function estimate used by \cite{KPV,KPV2}.

Let $Q_n = [n-\frac12,n+\frac12]$ so that $\mathbb{R} = \cup Q_n$.  Let $\tilde Q_n = [n-1,n+1]$.  An example of our notation is:
$$\|u\|_{\ell^\infty_n L_T^2L_{Q_n}^2} = \sup_n \| u \|_{L_{(0,T)}^2L_{Q_n}^2} \,.$$
We will use variants like $\ell_n^2L_T^\infty L_{Q_n}^2$ etc.  Note that due to the finite incidence of overlap, we have
$$\|u\|_{\ell_n^\infty L_T^2L_{Q_n}^2} \sim \|u\|_{\ell_n^\infty L_T^2L_{\tilde Q_n}^2}$$

\begin{theorem}[local well-posedness]
\label{T:local}
Take $k\in \mathbb{Z}$, $k\geq 1$.  Suppose that 
$$M \defeq \sum_{j=0}^{k+1} \|\partial_x^j b(x,t) \|_{L_{[0,1]}^\infty L_x^\infty} < \infty \,.$$
For any $R\geq 1$, take
$$T \lesssim \min(M^{-1}, R^{-4}) \,.$$
\begin{enumerate}
\item If $\|u_0\|_{H^k} \leq R$, there exists a solution $u(t)\in C([0,T]; H_x^k)$ to \eqref{E:pmKdV} on $[0,T]$ with initial data $u_0(x)$ satisfying
$$\|u\|_{L_T^\infty H_x^k} + \|\partial_x^{k+1}u\|_{\ell_n^\infty L_T^2L_{Q_n}^2} \lesssim R \,.$$ 
\item This solution $u(t)$ is unique among all solutions in $C([0,T]; H_x^1)$.
\item The data-to-solution map $u_0 \mapsto u(t)$ is continuous as a mapping $H^k \to C([0,T];H_x^k)$.
\end{enumerate}
\end{theorem}

The main tool in the proof of Theorem \ref{T:local} is the local smoothing estimate \eqref{E:locsmooth} below.

\begin{lemma}
\label{L:locsmooth}
Suppose that 
$$v_t + v_{xxx} -(bv)_x = f \,.$$
We have, for 
$$T\lesssim (1+ \|b_x\|_{L_T^\infty L_x^\infty} + \|b\|_{L_T^\infty L_x^\infty})^{-1} \,,$$
the energy and local smoothing estimates
\begin{equation}
\label{E:locsmooth}
\|v\|_{L_T^\infty L_x^2} + \|v_x\|_{\ell_n^\infty L_T^2 L_{Q_n}^2} \lesssim \|v_0\|_{L_x^2} +
\left\{
\begin{aligned}
&\| \partial_x^{-1}f \|_{\ell_n^1 L_T^2 L_{Q_n}^2} \\
&\| f \|_{L_T^1L_x^2}
\end{aligned}
\right.
\end{equation}
and the maximal function estimate
\begin{equation}
\label{E:maximal}
\|v\|_{\ell_n^2 L_T^\infty L_{Q_n}^2} \lesssim \|v_0\|_{L_x^2} + T^{1/2}\|v\|_{L_T^2 H_x^1} + T^{1/2}\|f\|_{L_T^2 L_x^2} \,.
\end{equation}
The implicit constants are independent of $b$.  
\end{lemma}

\begin{proof}
Let $\varphi(x) = -\tan^{-1}(x-n)$, and set $w(x,t) = e^{\varphi(x)}v(x,t)$.  
Note that $0<e^{-\frac{\pi}{2}}\leq e^{\varphi(x)} \leq e^\frac{\pi}{2} < \infty$, so the inclusion of this factor is harmless in the estimates, although has the benefit of generating the ``local smoothing'' term in \eqref{E:locsmooth}.
We have
$$\partial_t w + w_{xxx} -3\varphi' w_{xx} + 3(-\varphi''+(\varphi')^2)w_x + (-\varphi'''+3\varphi''\varphi' - (\varphi')^3)w - (bw)_x + \varphi' bw = e^\varphi f \,.$$
This equation and manipulations based on integration by parts show that 
$$\partial_t \|w\|_{L_x^2} = 
\begin{aligned}[t]
&6\la \varphi', w_x^2 \ra - 3\la (-\varphi'' + (\varphi')^2)', w^2 \ra + 2\la -\varphi'''+3\varphi''\varphi' -(\varphi')^3, w^2 \ra \\
&- \la b_x, w^2 \ra + 2\la b\varphi', w^2 \ra + 2\la w, e^\varphi f\ra \,.
\end{aligned}
$$
We integrate the above identity over $[0,T]$, 
move the smoothing term $6\int_0^T \la \varphi', w_x^2 \ra_x \, dt$ over to the left side, and estimate the remaining terms to obtain:
\begin{align*}
\indentalign \|w(T)\|_{L_x^2}^2 + 6 \| \la x-n\ra^{-1}w_x\|_{L_T^2 L_x^2}^2 \\
&\leq  \|w_0\|_{L_x^2}^2 + C T (1+ \|b_x\|_{L_T^\infty L_x^\infty} + \|b\|_{L_T^\infty L_X^\infty}) \|w\|_{L_T^\infty L_x^2}^2  + C \int_0^T \left| \int e^{\varphi}f w \, dx \right| \, dt \,.
\end{align*}
Replacing $T$ by $T'$, and taking the supremum over $T'\in [0,T]$, we obtain, for $T\lesssim (1+ \|b_x\|_{L_T^\infty L_x^\infty} + \|b\|_{L_{[0,T]}^\infty L_x^\infty})^{-1}$, the estimate
$$\|w\|_{L_T^\infty L_x^2}^2 + \| \la x-n\ra^{-1}w_x\|_{L_T^2 L_x^2}^2 \lesssim \|w_0\|_{L_x^2}^2 +  \int_0^T \left| \int e^{\varphi}f w \, dx \right| \, dt$$
Using that $0< e^{-\pi/2} \leq e^\varphi \leq e^{\pi/2} <\infty$, this estimate can be converted back to an estimate for $v$:
$$\|v\|_{L_T^\infty L_x^2}^2 + \|v_x\|_{L_T^2 L_{Q_n}^2}^2 \lesssim \|v_0\|_{L_x^2}^2 +  \int_0^T \left| \int e^{2\varphi}f v \, dx \right| \, dt \,.$$
Estimating as
$$\int_0^T \left| \int e^{2\varphi}f v \, dx \right| \, dt \lesssim \|f\|_{L_T^1L_x^2}\|v\|_{L_T^\infty L_x^2} \,, $$
and then taking the supremum in $n$ yields the second bound in \eqref{E:locsmooth}.  Estimating instead as:
\[
\begin{split}
\int_0^T \left| \int e^{2\varphi}f v \, dx \right| \, dt  & = \int_0^T \left| \int e^{2\varphi}(\partial_x\partial_x^{-1}f) v \, dx \right| dt \\
& \leq \int_0^T \left| \int (\partial_x^{-1}f) \, \partial_x(e^{2\varphi}v) \,dx \,\right| dt  \\
&\leq \sum_m \|\partial_x^{-1}f\|_{L_T^2L_{Q_m}^2} \|\la \partial_x \ra v\|_{L_T^2L_{Q_m}^2}
\\
& \leq \|\partial_x^{-1}f\|_{\ell_m^1 L_T^2L_{Q_m}^2}  \|\la \partial_x \ra v\|_{\ell_m^\infty L_T^2L_{Q_m}^2}
\end{split}
\]
and taking the supremum in $n$ yields the second bound in \eqref{E:locsmooth}.

For the  estimate \eqref{E:maximal}, we take $\psi(x)=1$ on $[n-\frac12,n+\frac12]$ and $0$ outside $[n-1,n+1]$, set $w = \psi v$, and compute, similarly to the above, 
$$\|v\|_{L_T^\infty L_{Q_n}^2}^2 \lesssim \|v_0\|_{L_{\tilde Q_n}^2}^2 + T\|v_x\|_{L_T^2 L_{\tilde Q_n}^2}^2 + T\|f\|_{L_T^2L_{\tilde Q_n}^2}^2$$
The proof is completed by summing in $n$.
\end{proof}

\begin{proof}[Proof of Theorem \ref{T:local}]
We prove the existence by contraction in the space $X$, where
$$X = \{ \, u\, | \, \|u\|_{C([0,T]; H_x^k)} + \|\partial_x^{k+1} u\|_{\ell_n^\infty L_T^2 L_{Q_n}^2} + \sup_{\alpha \leq k-1} \|\partial_x^\alpha u\|_{\ell_n^2 L_T^\infty L_{Q_n}^2}\leq CR \,\} \,.$$
Here $C$ is just chosen large enough to exceed the implicit constant in \eqref{E:locsmooth}. 
Given $u\in X $, let $\varphi(u)$ denote the solution to
\begin{equation}
\label{E:Phi-def}
\partial_t \varphi(u) + \partial_x^3 \varphi(u) - \partial_x (b\varphi(u)) = -2\partial_x(u^3) \,.
\end{equation}
with initial condition $\varphi(u)(0) = u_0$.
A fixed point $\varphi(u)=u$ in $X$ will solve \eqref{E:pmKdV}. We separately treat the case $k=1$ for clarity of exposition.

\smallskip

\noindent\emph{Case $k=1$}.  Applying $\partial_x$ to \eqref{E:Phi-def} gives, with $v=\varphi(u)_x$,
$$ v_t + v_{xxx} - (bv)_x = - 2(u^3)_{xx} + (b_x\varphi(u))_x \,.$$
Now, \eqref{E:locsmooth} gives 
\begin{equation}
\label{E:1stest-1}
\begin{split}
& \|\varphi(u)_x \|_{L_T^\infty L_x^2} + \| \varphi(u)_{xx} \|_{\ell_n^\infty L_T^2L_{Q_n}^2} \lesssim \\
& \ \ \ \ \ \ \ \|u_0\|_{H_x^1} + \| (u^3)_x \|_{\ell_n^1 L_T^2 L_{Q_n}^2} + \|(b_x \varphi(u))_x \|_{L_T^1 L_x^2} \,.
\end{split} 
\end{equation}
Using that $\|u\|_{L_Q^\infty}^2 \lesssim (\|u\|_{L_{\tilde Q}^2}+\|u_x\|_{L_{\tilde Q}^2}) \|u \|_{L_{\tilde Q}^2}$, we also have 
$$\| (u^3)_x \|_{L_Q^2} \lesssim \|u_x \|_{L_Q^2} \|u\|_{L_Q^\infty}^2 \lesssim \|u_x\|_{L_Q^2} \|u\|_{L_{\tilde Q}^2} ( \|u\|_{L_{\tilde Q}^2} + \|u_x\|_{L_{\tilde Q}^2}) \,.$$
Taking the $L_T^2$ norm and applying the H\"older inequality, we obtain
$$\| (u^3)_x  \|_{L_T^2L_Q^2} \lesssim \|u_x\|_{L_T^\infty L_Q^2} \|u\|_{L_T^\infty L_{\tilde Q}^2} ( \|u\|_{L_T^2L_{\tilde Q}^2} + \|u_x\|_{L_T^2L_{\tilde Q}^2}) \,.$$
Taking the  $\ell_n^1$ norm and applying the H\"older inequality again yields
$$\| (u^3)_x \|_{\ell^1 L_T^2 L_{Q_n}^2} \lesssim \|u_x\|_{\ell_n^\infty L_T^\infty L_{Q_n}^2} \|u\|_{\ell_n^2 L_T^\infty L_{\tilde Q_n}^2} ( \|u\|_{\ell_n^2L_T^2L_{\tilde Q_n}^2} + \|u_x\|_{\ell_n^2 L_T^2 L_{\tilde Q_n}^2}) \,.$$
Using the straightforward bounds $\|u_x\|_{\ell_n^\infty L_T^\infty L_{Q_n}^2} \lesssim \|u_x\|_{L_T^\infty L_x^2}$,  
$$\|u\|_{\ell_n^2 L_T^2 L_{\tilde Q_n}^2} \lesssim \|u\|_{L_T^2L_x^2} \lesssim T^{1/2} \|u\|_{L_T^\infty L_x^2}$$ and $$\|u_x\|_{\ell_n^2 L_T^2 L_{\tilde Q_n}^2} \lesssim \|u_x\|_{L_T^2L_x^2}\lesssim T^{1/2} \|u_x\|_{L_T^\infty L_x^2} \,,$$
we obtain
$$\| (u^3)_x \|_{\ell_n^1 L_T^2 L_{Q_n}^2} \lesssim T^{1/2} \|u\|_{L_T^\infty H_x^1}^2 \|u\|_{\ell_n^2 L_T^\infty L_{Q_n}^2} \,.$$
Inserting these bounds into \eqref{E:1stest-1},
\begin{equation}
\label{E:1stest}
\|\varphi(u)_x \|_{L_T^\infty L_x^2} + \| \varphi(u)_{xx} \|_{\ell_n^\infty L_T^2L_{Q_n}^2} \lesssim 
\begin{aligned}[t]
&\|u_0\|_{H_x^1} + T^{1/2}\|u\|_{L_T^\infty H_x^1}^2 \|u\|_{\ell_n^2 L_T^\infty L_{Q_n}^2} \\
&+ T(\|b_x\|_{L_x^\infty} + \|b_{xx}\|_{L_x^\infty}) \|\varphi(u)\|_{H_x^1} \,.
\end{aligned}
\end{equation}
The local smoothing estimate \eqref{E:locsmooth} applied  
to $v=\varphi(u)$ (not $v=\varphi(u)_x$ as above),
and the estimate 
\[ \|(u^3)_x \|_{L_T^1 L_x^2} \lesssim T\|u\|_{L_T^\infty H_x^1}^3 \,, \]
provides the estimate
\begin{equation}
\label{E:2ndest}
\|\varphi(u)\|_{L_T^\infty L_x^2} \lesssim T\|u\|_{L_T^\infty H_x^1}^3
\end{equation}
The maximal function estimate 
\eqref{E:maximal} applied to 
$v=\varphi(u)$ and the estimate 
$$\|(u^3)_x \|_{L_T^2 L_x^2} \lesssim T^{1/2}\|u\|_{L_T^\infty H_x^1}^3\,,
$$ 
give the estimate
\begin{equation}
\label{E:3rdest}
\|\varphi(u)\|_{\ell_n^2 L_T^\infty L_{Q_n}^2} \lesssim \|u_0\|_{L_x^2} + T\|\varphi(u)\|_{L_T^\infty H_x^1} + T\| u \|_{L_T^\infty H_x^1}^3 \,. 
\end{equation}
Summing \eqref{E:1stest}, \eqref{E:2ndest}, \eqref{E:3rdest}, we obtain that $\|\varphi(u)\|_X \leq CR$ if $\|u\|_X \leq CR$ provided $T$ is as stated above.  Thus $\varphi:X\to X$. A similar argument establishes that $\varphi$ is a contraction on $X$.

\smallskip

\noindent\emph{Case $k\geq 2$}.
Differentiating \eqref{E:Phi-def} $ k$ times with respect to $ x $ 
we obtain, with $v=\partial_x^k\varphi(u)$,
$$\partial_t v + \partial_x^3 v - \partial_x(bv) = -2\partial_x^{k+1}(u^3) - 2\partial_x \sum_{\substack{\alpha+\beta \leq k+1 \\ \beta \leq k-1}} \partial_x^\alpha b \; \partial_x^\beta \varphi(u) \,. $$
Using \eqref{E:locsmooth} gives 
\begin{gather*} \|\partial_x^k \varphi(u)\|_{L_T^\infty L_x^2} + \|\partial_x^{k+1}\varphi(u)\|_{\ell_n^\infty L_T^2L_{Q_n}^2} \lesssim 
\\ \| \partial_x^k u^3\|_{\ell_n^1 L_T^2L_{Q_n}^2}  
+ \sup_{\substack{\alpha +\beta \leq k+1 \\ \beta\leq k-1}} \| \partial_x ( \partial_x^\alpha b \; \partial_x^\beta \varphi(u)) \|_{L_T^1L_x^2}
\,. \end{gather*}
Expanding, and applying Leibniz rule gives
$$\partial_x^k u = \sum_{\substack{\alpha+\beta+\gamma=k \\ \alpha\leq \beta \leq \gamma}} c_{\alpha\beta\gamma} \partial_x^\alpha u \;  \partial_x^\beta u \;  \partial_x^\gamma u\,, $$
which is then estimated as follows 
$$\| \partial_x^k u \|_{\ell_n^1L_T^2L_{Q_n}^2} \lesssim \sum_{\substack{\alpha+\beta+\gamma=k \\ \alpha\leq \beta \leq \gamma}} \| \partial_x^\alpha u\|_{\ell_n^2 L_T^\infty L_{Q_n}^\infty} \| \partial_x^\beta u\|_{\ell_n^2 L_T^2 L_{Q_n}^\infty} \| \partial_x^\gamma u\|_{\ell_n^\infty L_T^\infty L_{Q_n}^2} \,. $$
By the Sobolev embedding theorem (as in the $k=1$ case) we obtain
$$\| \partial_x^k u^3 \|_{\ell_n^1 L_T^2L_{Q_n}^2} \lesssim \sum_{\substack{\alpha+\beta+\gamma=k \\ \alpha\leq \beta \leq \gamma}} \left( \sup_{\sigma \leq \alpha+1} \|\partial_x^\sigma u \|_{\ell_n^2 L_T^\infty L_{Q_n}^2} \right)\left( \sup_{\sigma \leq \beta+1} \|\partial_x^\sigma u \|_{\ell_n^2 L_T^2 L_{Q_n}^2} \right) \|\partial_x^\gamma u\|_{L_T^\infty L_x^2} $$
When $k\geq 2$, we have $\alpha \leq [[\frac13k]] \leq k-2$ and $\beta \leq [[\frac12k]] \leq k-1$, and therefore
$$\| \partial_x^k u^3 \|_{\ell_n^1 L_T^2L_{Q_n}^2} \lesssim T^{1/2}\left( \sup_{\alpha\leq k-1} \|\partial_x^\alpha u\|_{\ell_n^2L_T^\infty L_{Q_n}^2}\right) \|u\|_{L_T^\infty H_x^k}^2 \,.$$
Also, 
$$\| \partial_x ( \partial_x^\alpha b \; \partial_x^\beta \varphi(u)) \|_{L_T^1L_x^2} \leq T \left( \sup_{\alpha \leq k+1} \|\partial_x^\alpha b\|_{L_T^\infty L_x^\infty} \right) \|\varphi(u)\|_{L_T^\infty H_x^k}$$
Combining these estimates, we obtain
\begin{gather}
\label{E:1stest-gen}
\begin{gathered}
\|\partial_x^k \varphi(u)\|_{L_T^\infty L_x^2}  + \|\partial_x^{k+1}\varphi(u)\|_{\ell_n^\infty L_T^2L_{Q_n}^2} 
\lesssim \|u_0\|_{H_x^k}  \\
 + \, T^{1/2}\left( \sup_{\alpha\leq k-1} \|\partial_x^\alpha u\|_{\ell_n^2L_T^\infty L_{Q_n}^2}\right) \|u\|_{L_T^\infty H_x^k}^2
+ 
T \left( \sup_{\alpha \leq k+1} \|\partial_x^\alpha b\|_{L_T^\infty L_x^\infty} \right) \|\varphi(u)\|_{L_T^\infty H_x^k}
\end{gathered}
\end{gather}
The local smoothing 
$\|(u^3)_x \|_{L_T^1 L_x^2} \lesssim T\|u\|_{L_T^\infty H_x^1}^3$ to obtain 
\begin{equation}
\label{E:2ndest-gen}
\|\varphi(u)\|_{L_T^\infty L_x^2} \lesssim T\|u\|_{L_T^\infty H_x^1}^3
\end{equation}
We apply the maximal function estimate \eqref{E:maximal} 
to $v=\partial_x^\alpha \varphi(u)$ for $\alpha\leq k-1$  and use that $\|\partial_x^{\alpha+1} u^3 \|_{L_T^1L_x^2} \leq T \|u\|_{L_T^\infty H_x^k}^3$ and 
$$\| \partial_x^{\alpha+1} (b\varphi(u)) \|_{L_T^1L_x^2} \leq T \left( \sup_{\beta \leq k} \|\partial_x^\beta b\|_{L_T^\infty L_x^\infty}\right) \|\varphi(u)\|_{L_T^\infty H_x^k}$$ 
to obtain
\begin{equation}
\label{E:3rdest-gen}
\| \partial_x^\alpha \varphi(u)\|_{\ell_n^2 L_T^\infty L_{Q_n}^2} \lesssim
\begin{aligned}[t]
&\|u_0\|_{H_x^{k-1}} + T\|\varphi(u)\|_{L_T^\infty H_x^k} + T \|u\|_{L_T^\infty H_x^k}^3  \\
&+ T \left( \sup_{\beta \leq k} \|\partial_x^\beta b\|_{L_T^\infty L_x^\infty}\right) \|\varphi(u)\|_{L_T^\infty H_x^k}
\end{aligned}
\end{equation}
Summing \eqref{E:1stest-gen}, \eqref{E:2ndest-gen}, \eqref{E:3rdest-gen}, we obtain that $\varphi:X\to X$, and a similar argument shows that $\varphi$ is a contraction.  This concludes the case $k\geq 2$.

To establish uniqueness within the broader class of solutions belonging merely to $C([0,T]; H_x^1)$, we argue as follows.  Suppose $u,v \in C([0,T]; H_x^1)$ solve \eqref{E:pmKdV}. By \eqref{E:maximal},
$$\|v\|_{\ell_n^2L_T^\infty L_{Q_n}^2} \lesssim \|v_0\|_{L^2} + T\|v\|_{L_T^\infty H_x^1} + T\|v\|_{L_T^\infty H_x^1}^3 \,.$$
By taking $T$ small enough in terms of $\|v\|_{L_T^\infty H_x^1}$, we have that
\begin{equation}
\label{E:mxmlbds1}
\|v\|_{\ell_n^2L_T^\infty L_{Q_n}^2} \lesssim \|v\|_{L_T^\infty H_x^1} \,.
\end{equation}
Similarly,
\begin{equation}
\label{E:mxmlbds2}
\|u\|_{\ell_n^2 L_T^\infty L_{Q_n}^2} \lesssim \|u\|_{L_T^\infty H_x^1} \,.
\end{equation}

Set $w=u-v$.  Then, with $g= (u^3-v^3)/(u-v)=u^2+uv+v^2$, we have
$$w_t + w_{xxx} -(bw)_x \pm (gw)_x=0 \,.$$
Apply \eqref{E:locsmooth} to $v=w_x$ to obtain
\begin{equation}
\label{E:uniq1}
\|w_x\|_{L_T^\infty L_x^2} +  \|w_{xx} \|_{\ell_n^\infty L_T^2L_{Q_n}^2} \lesssim  \| (gw)_x\|_{\ell_n^1 L_T^2 L_{Q_n}^2} + \|(b_xw)_x\|_{L_T^1L_x^2}
\end{equation}
The terms of $\|(gw)_x\|_{\ell_n^1 L_T^2 L_{Q_n}^2}$ are bounded following the method used above:
\begin{align*}
\|u_x v w \|_{\ell_n^1L_T^2L_{Q_n}^2} 
&\lesssim  \| u_x \|_{\ell_n^\infty L_T^\infty L_{Q_n}^2} \|vw \|_{\ell_n^1 L_T^2L_{Q_n}^\infty} \\
&\lesssim \| u_x \|_{\ell_n^\infty L_T^\infty L_{Q_n}^2}( \|vw \|_{\ell_n^1 L_T^2L_{Q_n}^1} + \|(vw)_x \|_{\ell_n^1 L_T^2L_{Q_n}^1})
\end{align*}
The term in parentheses is bounded by
$$\|v\|_{\ell_n^2 L_T^2L_{Q_n}^2}\|w\|_{\ell_n^2L_T^\infty L_{Q_n}^2} + \|v_x\|_{\ell_n^2L_T^2L_{Q_n}^2}\|w\|_{\ell_n^2L_T^\infty L_{Q_n}^2}+\|v\|_{\ell_n^2L_T^\infty L_{Q_n}^2}\|w_x\|_{\ell_n^2L_T^2L_{Q_n}^2}$$
which leads to the bound
\begin{equation}
\label{E:uniq2}
\|u_x v w \|_{\ell_n^1L_T^2L_{Q_n}^2} \lesssim T^{1/2}\|u\|_{L_T^\infty H_x^1}(\|v\|_{L_T^\infty H_x^1} \|w\|_{\ell_n^2L_T^\infty L_{Q_n}^2} + \|v\|_{\ell_n^2L_T^\infty L_{Q_n}^2}\|w\|_{L_T^\infty H_x^1} )
\end{equation}
 We now allow implicit constants to depend upon $\|u\|_{L_T^\infty H_x^1}$ and $\|v\|_{L_T^\infty H_x^1}$.  Appealing to \eqref{E:uniq1}, \eqref{E:uniq2} (and analogous estimates for other terms in $gw$), \eqref{E:mxmlbds1}, \eqref{E:mxmlbds2} to obtain
$$\|w\|_{L_T^\infty H_x^1} \lesssim T^{1/2}(\|w\|_{\ell_n^2L_T^\infty L_{Q_n}^2} + \|w\|_{L_T^\infty H_x^1})$$
Combining this estimate with the maximal function estimate 
\eqref{E:maximal} applied to $w$ yields
$$\|w\|_{\ell_n^2L_T^\infty L_{Q_n}^2} \lesssim T^{1/2}\|w\|_{L_T^\infty H_x^1} + T\|g\|_{L_T^\infty H_x^1}\|w\|_{L_T^\infty H_x^1} \,.$$
This gives $w\equiv 0$ for $T$ sufficiently small.
The continuity of the data-to-solution map is proved using similar arguments.
\end{proof}

Next, we prove global well-posedness in $H^k$ by proving \emph{a priori} bounds.
Theorem \ref{T:local} shows that doing it suffices for global well-posedness

\begin{theorem}[global well-posedness]
\label{T:global}
Fix $k\geq 1$ and suppose $M(T)<\infty$ for all $T \geq 0$, where $M(T)$ is defined in \eqref{E:b-size}.  For $u_0\in H^k$, there is a unique global solution $u\in C_{\text{loc}}([0,+\infty); H_x^k)$ to \eqref{E:pmKdV} with $\|u\|_{L_T^\infty H_x^k}$ controlled by $\|u_0\|_{H^k}$, $T$, and $M(T)$.
\end{theorem}
\begin{proof}
Before beginning, we note that by the Gagliaro-Nirenberg inequality,
 $\|u\|_{L^4}^4 \lesssim \|u\|_{L^2}^3\|u_x\|_{L^2}$, 
we have (in the focusing case)
$$\|u_x\|_{L^2}^2 - \|u_x\|\|u\|_{L^2}^3 \leq  I_3(u) \leq \|u_x\|_{L^2}^2 \,.$$
With $\alpha = \|u_x\|_{L^2}^2/ \|u\|_{L^2}^6$ and $\beta = I_3(u)/\|u\|_{L^2}^6$, this is $\alpha - \alpha^{1/2} \leq \beta \leq \alpha$, which implies that $\la \alpha \ra \sim \la \beta \ra$, i.e. 
$$\|u_x\|_{L^2}^2 + \|u\|_{L^2}^6  \sim I_3(u) + \|u\|_{L^2}^6$$
The same statement holds in the defocusing case.

Another fact we need is based on the 
\begin{align*}
\frac{d}{dt} I_j(u) &= \la I_j'(u), \partial_t u \ra \\
&= \la I_j'(u),  -u_{xxx}-2(u^3)_x +(bu)_x \ra  \\
&= \la I_j'(u), \partial_x I_3'(u) \ra + \la I_j'(u), (bu)_x \ra\\
&= \la I_j'(u), (bu)_x \ra
\end{align*}

For $u(t)\in L^2$, we compute near conservation of momentum and energy from Lemma \ref{L:near-conserved}:
$$\frac{d}{dt} I_1(u) = \la b_x, A_1(u)\ra$$
Estimate $|\la b_x, A_1(u)\ra| \leq \|b_x\|_{L^\infty} I_1(u)$, and apply Gronwall to obtain a bound on $\|u\|_{L_T^\infty L_x^2}$ in terms of $\|b_x\|_{L_T^\infty L^\infty}$ and $\|u_0\|_{L^2}$.
For $u(t)\in H^1$, we compute near conservation of energy from Lemma \ref{L:near-conserved}:
$$\frac{d}{dt} I_3(u) = 3\la b_{x}, A_3(u)\ra - \la b_{xxx},A_1(u)\ra \,.$$
We have
\begin{align*}
| \la b_{x},A_3(u)\ra| 
&\lesssim \|b_{x}\|_{L^\infty}(\|u_x\|_{L^2}^2 + \|u\|_{L^4}^4)  \\
&\lesssim  \|b_{x}\|_{L^\infty}(\|u_x\|_{L^2}^2 + \|u_x\|_{L^2}\|u\|_{L^2}^3) \\
&\lesssim \|b_{x}\|_{L^\infty} (\|u_x\|_{L^2}^2 + \|u\|_{L^2}^6) \\
&\lesssim \|b_{x}\|_{L^\infty} (I_3(u) + \|u\|_{L^2}^6)
\end{align*}
and
$$|\la b_{xxx},A_1(u)\ra | \lesssim \|b_{xxx}\|_{L^\infty} \|u\|_{L^2}^2 \,.$$
Combining these gives
$$\left|\frac{d}{dt} I_3(u) \right| \lesssim \|b_x\|_{L^\infty} I_3(u) + \|b_x\|_{L^\infty}\|u\|_{L^2}^6+ \|b_{xxx}\|_{L^\infty} \|u\|_{L^2}^2$$
Gronwall's inequality, combined with the previous bound on $\|u\|_{L^2}$, gives the bound on $I_3(u)$ and hence $\|u\|_{H^1}$.

For $u(t)\in H^2$, we apply Lemma \ref{L:near-conserved} to obtain
\begin{align*}
\frac{d}{dt} I_5(u) 
&=  \la I_5'(u), (bu)_x \ra \\
&= 5\la b_x, A_5(u)\ra - 5 \la b_{xxx}, A_3(u)\ra + \la b_{xxxxx},A_1(u)\ra 
\end{align*}
We have
\begin{align*}
|\la b_x, A_5(u) \ra | &\lesssim \|b_x\|_{L^\infty}( \|u_{xx}\|_{L^2}^2 + \|u\|_{H^1}^4 + \|u\|_{H^1}^6) \\
&\lesssim \|b_x\|_{L^\infty} I_5(u) + \|b_x\|_{L^\infty}(\|u\|_{H^1}^4 + \|u\|_{H^1}^6)
\end{align*}
Also,
$$|\la b_{xxx},A_3(u)\ra | \lesssim \|b_{xxx}\|_{L^\infty}(\|u\|_{H^1}^2+\|u\|_{H^1}^4)$$
and
$$| \la b_{xxxxx}, A_1(u) \ra | \lesssim \|b_{xxx}\|_{L^\infty} \|(u^2)_{xx}\|_{L^2} \lesssim \|b_{xxx}\|_{L^\infty} \|u\|_{H^2}\|u\|_{L^2} $$
Combining, applying Gronwall's inequality, and appealing to the bound on $\|u\|_{H^1}$ obtained previously, we obtain the claimed \emph{a priori} bound in the case $k=2$.

Bounds on $H^k$ for $k\geq 3$ can be obtained by the above method appealing to higher-order analogues of the identities in Lemma \ref{L:near-conserved}.  However, starting with $k=3$, we do not need such refined information.  By direct computation from \eqref{E:pmKdV},
$$\frac{d}{dt} \|\partial_x^k u\|_{L^2}^2 = -\int \partial_x^{k+1}(bu) \, \partial_x^k u + 2\int \partial_x^{k+1}u^3 \; \partial_x^k u$$
In the Leibniz expansion of $\partial_x^{k+1}u^3$, we isolate two cases:
$$\partial_x^{k+1}u^3 = 3u^2 \partial_x^{k+1}u + \sum_{\substack{\alpha+\beta+\gamma = k+1 \\ \alpha \leq \beta \leq \gamma\leq k}} c_{\alpha\beta\gamma} \partial_x^\alpha u \; \partial_x^\beta u \; \partial_x^\gamma u$$
For the first term,
$$ \left| \int u^2 \, \partial_x^{k+1}u \, \partial_x^k u \right| = \left|\int (u^2)_x (\partial_x^k u)^2 \right| \lesssim \|u\|_{H^2}^2 \|u\|_{H^k}^2$$ 
By the H\"older's inequality and interpolation, if $\alpha+\beta+\gamma=k+1$ and $\gamma\leq k$,
$$\| \partial_x^\alpha u \; \partial_x^\beta u \; \partial_x^\gamma u \|_{L^2} \lesssim \|u\|_{H^2}^2\|u\|_{H^k}$$
Thus we have
$$\left| \int \partial_x^{k+1}u^3 \; \partial_x^k u \right| \lesssim \|u\|_{H^2}^2 \|u\|_{H^k}^2$$
Similarly, we can bound
$$\left|\int \partial_x^{k+1}(bu) \, \partial_x^k u \right| \lesssim M(t)\|u\|_{H^k}^2$$
by separately considering the term $b\, \partial_x^{k+1}u \, \partial_x^ku$ and integrating by parts.  We obtain
$$\left| \frac{d}{dt} \|\partial_x^k u \|_{L^2}^2 \right| \lesssim (M+\|u\|_{H^2}^2)\|u\|_{H^k}^2$$
and can apply the Gronwall inequality to obtain the desired \emph{a priori} bound.
\end{proof}

\section{Comments about the effective ODEs}
\label{A:ode}
Here we make some comments about the differential equations
for the parameters $ a $ and $ c$.

\subsection{Conditions on $ T_0 $.}
\label{AS:noh}
First we give a reason for replacing $ T_0 ( h ) $ 
in the definition of $ T( h ) $ \eqref{eq:T0h} by $ T_0 $ 
defined by \eqref{eq:ODE}. In \eqref{eq:oded} we have seen
that the $ a $ and $ c$ solving the system \eqref{E:eom}
give the following equations for $ \widetilde A = h a $, 
$ \widetilde C = c $, $ T = h t $:
\[ \left\{
\begin{aligned}
& \partial_T \widetilde A_j = \widetilde C_j^2 - b_0( \widetilde A_j,T)  + 
{\mathcal O} ( h ) \\
& \partial_T \widetilde C_j = \widetilde C_j \partial_x b_0 (\widetilde A_j, T)
+ {\mathcal O} ( h ) 
\end{aligned}
\right. \,, 
\qquad \widetilde A(0) = \bar a h \,, \quad \widetilde  C(0) = \bar c \,, \ \ \ j=1,2\,.
\]
This can also be seen by analysing \eqref{eq:dynh} 
using Lemma \ref{L:q-asymp}.

As in \eqref{eq:ODEt} we can write the equations for 
$ \widetilde A_j - A_j $ and $ \widetilde C_j - C_j $:
\[ \left\{
\begin{aligned}
& \partial_T (\widetilde {A}_j - A_j) = (\widetilde C_j - C_j) ^2 +
 2 C_j (\widetilde C_j - C_j ) + \gamma_0 ( \widetilde A_j  - A_j ) + {\mathcal O}(h) \\
& \partial_T ( \widetilde C_j - C_j ) = (\widetilde C_j  - C_j ) 
(\partial_x b_0)(A_j,t) +  C_j \sigma_0 (\widetilde A_j - A_j ) +
{\mathcal O}(h) \,,
\\ 
& \widetilde A_j ( 0 ) - A_j ( 0 ) = 0 \,, \ \ \widetilde C_j ( 0 ) 
- C_j ( 0) = 0 \,, 
\end{aligned}
\right. \]
where $ \gamma_0 , \sigma_0 = {\mathcal O} ( 1 ) $. This implies
that 
\[ \left\{
\begin{aligned}
& \widetilde A_j ( T ) - A_j ( T ) = {\mathcal O} ( h ) e^{CT} \,, \\
& \widetilde C_j ( T ) - C_j ( T) = {\mathcal O} ( h ) e^{CT}  \,. 
\end{aligned}
\right. \]
This means that for $ T < \delta \log ( 1/h ) $, 
we have $ C_j ( T ) = \widetilde C_j ( T ) + {\mathcal O} (h^{ 1- \delta C} ) $.
Hence, if $ \delta $ is small enough, then for small $ h $ we have
that $ T_0 ( h ) $ defined in \eqref{eq:T0h} and $ T_0 $ in 
\eqref{eq:ODE} can be interchanged.

\subsection{Examples with $ C_j $  going to $ 0 $}
In the decoupled equations \eqref{eq:ODEsh} we can have 
\[ C_j ( T ) \rightarrow 0 \, , \ \  T \rightarrow \infty \,, \]
which implies that $ T_0 < \infty $ in the definition \eqref{eq:ODE}.
That prevents $ \log ( 1/h ) / h $ lifespan of the approximation 
\eqref{eq:Th}. 

Let us put 
\[
a = A_j\,, \ \  c = C_j \,, \] 
so that the 
system  \eqref{eq:ODEsh} becomes
\begin{equation}
\label{eq:taua}  a'_T = c^2 ( T ) - b_0 ( a , T ) \,, \ \ 
c_T' = c \, \partial_a b_0 ( a , T) 
\,. \end{equation}
For simplicity we consider the case of $ b_0 ( a, T ) = b_0 ( a ) $.
In that case the Hamiltonian 
\[ E (a , c ) = - \frac 1 3 c^3 + c b_0 ( a  )  \]
is conserved in the evolution and we have
\begin{equation}
\label{eq:obv}
  \exp (  T \min \partial_a b )
\leq  | c ( T ) | \leq 
  \exp (  T \max \partial_a b ) \,.
\end{equation}
In particular this means that $ c > \delta > 0 $ if $ T< T_1(\delta)  $. 

We cannot improve on \eqref{eq:obv}, and in general we may have 
\[ | c  (  T) | \leq e^{ -  \gamma  T } \,, \  \ T \rightarrow \infty \,, \]
but this behaviour is rare. First we note that the conservation of
$ E $ shows that if 
$  c ( T_j ) \rightarrow 0 $ for some sequence $ T_j \rightarrow \infty $,
then $ E = 0 $.
We can then solve for $ c $, and the equation reduces to  
$ { da}/{ dT}  = 2 {b_0}(a )$, $c^2 = 3 {b_0} ( a ) $,
that is to 
\begin{equation}
\label{eq:int}
\frac 12 \int_{a_0}^{a}  \frac{ d \tilde a } {{b_0} ( \tilde a) }
= T \,, \ \ b ( a ( 0 ) ) > 0 \,.  \end{equation}
If $ {b_0} ( a ) > 0 $ in this set of values $ a $ then 
\begin{equation}
\label{eq:alpha}  a ( T ) \rightarrow \infty \,, \ T \rightarrow
\infty \,, \end{equation}
and $ c ( T ) = ( 3 {b_0} ( a ( T ) ))^{\frac12}  $.

If $ {b_0} ( a ) = 0 $ for some $ a > a ( 0 ) $ ($ a'_T = 
2 b_0 > 0 $), then we denote
$ a_1 $, 
the smallest such $ a $ and assume that the order of vanishing
of $ {b_0} $ there is $ \ell_1 $. The analysis of \eqref{eq:int}
shows that
\[  a (T) = a_1 + {\mathcal O}(1)
\left\{ \begin{array}{ll} K e^{- \gamma T } & 
\ell_1 = 1\,, \\
\ & \ \\
K T^{- 1 /(\ell_1 - 1)  }  & \ell_1 > 1 \,, \end{array} \right. \]
which gives the rate of decay of $ c ( T ) $.

Hence we have shown the following statement which is almost
as long to state as to prove:

\begin{lemma}
\label{l:ode}
Suppose that in \eqref{eq:taua} $ {b_0} = {b_0} ( a ) $. Then 
\[ E \neq 0 \,, \ \ | c(0) | > \delta_0 > 0  \ 
\Longrightarrow \ \exists \, \delta > 0 \ \forall \, T > 0 \,, \ \ 
| c ( T ) | > \delta\,. \]

If $ E = 0 $, let 
\[  a_1 = \min \{ a \; : \;  a > a (0) \,,
\ {b_0}( a ) = 0 \} \,,\]
with $ a_1 $ not defined if the set is 
empty (note that $ c(0) \neq 0 $ and $ E = 0 $ imply 
that $ {b_0} ( a(0)) >  0 $). 
Now suppose that $ a_1 $ exists, and that 
\[   \partial^\ell {b_0} ( a_1 ) = 0 \,, \ \ell < \ell_1 \,,\ \ 
\partial^{\ell_1} {b_0} ( a_1 ) \neq 0 \,. \]
Then as $ T  \rightarrow \infty $,
\[ |c ( T )| \leq \left\{ \begin{array}{ll} K e^{- \gamma T } & 
\ell_1 = 1\,, \\
\ & \ \\
K T^{- \ell_1 /(\ell_1 - 1)  }  & \ell_1 > 1 \,, \end{array} \right.
\]
for some constants $ \gamma $ and $ K$, 
and $  a ( T ) \rightarrow a_1 $.

 If $ a_1 $ does not 
exist then $  c ( T ) = (3 {b_0} ( a ( T ) ) )^{\frac12} $, 
$ a ( T ) \rightarrow \infty $, $ T \rightarrow \infty$.
\end{lemma}

We excluded the case of infinite order of vanishing since it is very
special from our point of view.

The lemma suggests that $ c \rightarrow 0 $ is highly nongeneric 
but it can occur for our system. Since for the original time $ t $ in 
\eqref{E:pmKdV} we would like to go up to time $ \delta \log ( 1/h ) / h $ 
we cannot do it in some cases as then 
\[ c (t)|_{t =\delta \log ( 1/h ) / h } \sim \left\{ \begin{array}{ll}  h^{\gamma \delta/2}  & 
\ell_1 = 1\,, \\
\ & \ \\
\log^{- \frac 12 \ell_1 / (\ell_1 - 1)  } (1/h)  & \ell_1 > 1 \,. \end{array} \right. 
\]

\subsection{Avoided crossing for the effective equations of motion.}
Here we make some comments about the puzzling avoided crossing
which needs further investigation.

For the decoupled equations it is easy to find examples
in which 
\begin{equation}
\label{eq:crossC}  c_1 ( T_0 ) = c_2 ( T_0 ) \,.
\end{equation} 
One is shown in 
Fig.\ref{f:A1}. We take $ b_0 $ independent of $ T $ and
equal to $ \cos^2 x $. If we choose the initial 
conditions so that $ c_j^2 = 3 \cos^2 A_j $, $ A_j = h a_j $ as
in \eqref{eq:ODEsh}, and $ -\pi/2 < A_1 < -A_2 < 0 $, then 
when $ A_1 ( T_0 ) = - A_2 ( T_0 )$ we have \eqref{eq:crossC}
 (this also provides an 
example of $ c_2 ( T ) \rightarrow 0$ as $ T \rightarrow \infty $).

\begin{figure}
\begin{center}
\includegraphics[width=6in]{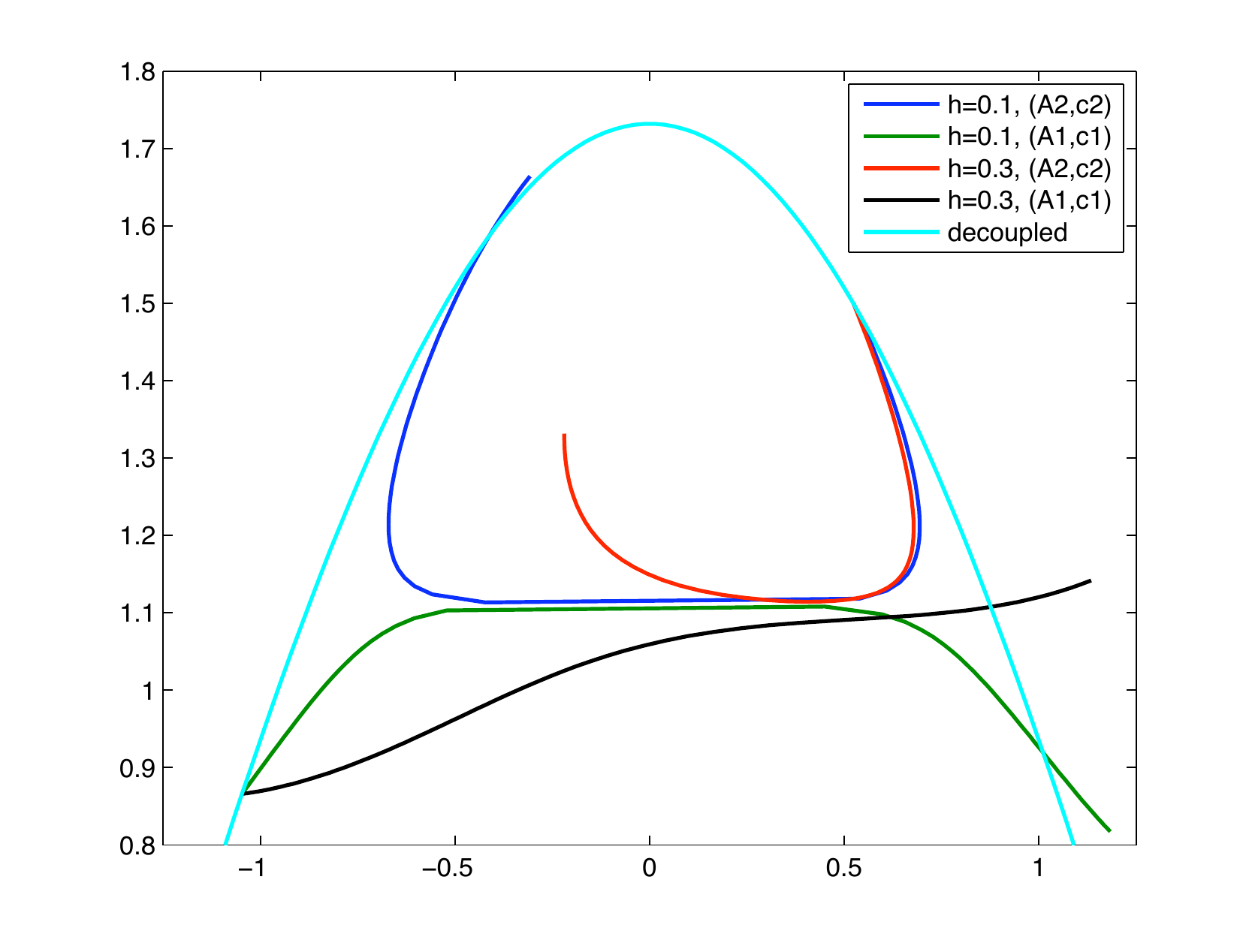}
\end{center}
\caption{The plots of $ (A_j , c_j ) $, $ j = 1, 2 $, solving
\eqref{eq:dynh} for
for $ b_0 ( x , t) = \cos^2 x $ and initial data 
$ A_1 (0) = - \pi/3 $, $ A_2 (0 ) = \pi/6 $, and 
$ c_1 (0) = \sqrt 3 \cos(\pi/3) $, $  c_2 (0) = \sqrt 3 \cos(\pi/6)$.
The ``decoupled'' curve corresponds to solving \eqref{eq:ODEsh}. Because
of the choice of initial conditions, $ (A_j , c_j ) $, $ j = 1,2 $ 
line on the same curve.}
\label{f:A1}
\end{figure}

The decoupled equations \eqref{eq:ODEsh} 
should be compared the rescaled version of \eqref{E:eom}:
\begin{gather}
\label{eq:dynh}
\begin{gathered}
\partial_T c_j = \partial_{x_j} B_0 ( c, A , h ) \,, \ \
\partial_T A_j = c_j^2 - \partial_{c_j} B_0 ( c , A , h ) \,, \\
B_0 ( c, A , h ) \defeq \frac 12 \int q_2 ( x/h , c, A/h ) b_0 (x) dx
\,. 
\end{gathered}
\end{gather}
For the example above the comparison between the 
solutions of the decoupled $h$-independent equations
and solutions to the equation \eqref{eq:dynh} are shown 
in Fig.\ref{f:A1} (the solutions \eqref{eq:ODEsh}
are shown as a single curve which both solutions with
these initial data follow).

The dramatic avoided crossings shown in Fig.\ref{f:A1} (and 
also, for a different, time dependent $ b_0 $ in 
Fig.\ref{f:cros}) are not seen  
in the behaviour of $ q_2 ( x, c, A/h ) $ which is the 
approximation of the solution to \eqref{E:pmKdV} -- see Fig.\ref{f:A2}. 
The masses of the right and left solitons are switched and that
corresponds to the switch of positions of $ A_1 $ and $ A_2$. 
It is possible that a different parametrization of double solitons
would resolve this problem. Another possibility is 
to study the decomposition 
\eqref{eq:Qdec} in the proof of Lemma \ref{L:q-asymp} 
uniformly $ \alpha \rightarrow 0 $ (corresponding to $ a_2 - a_1
\rightarrow 0 $).

\begin{figure}
\begin{center}
\includegraphics[width=2.6in]{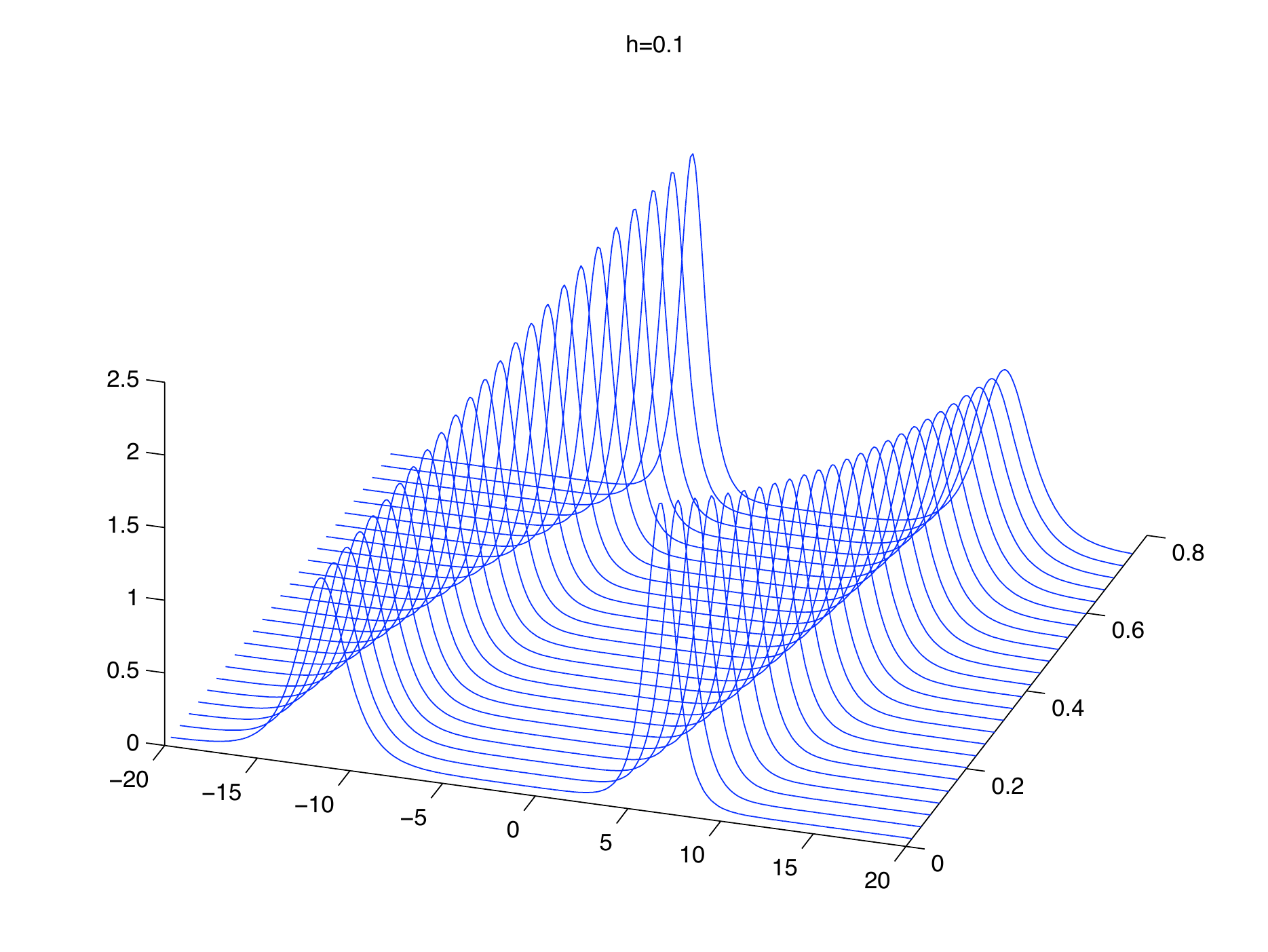} \includegraphics[width=2.6in]{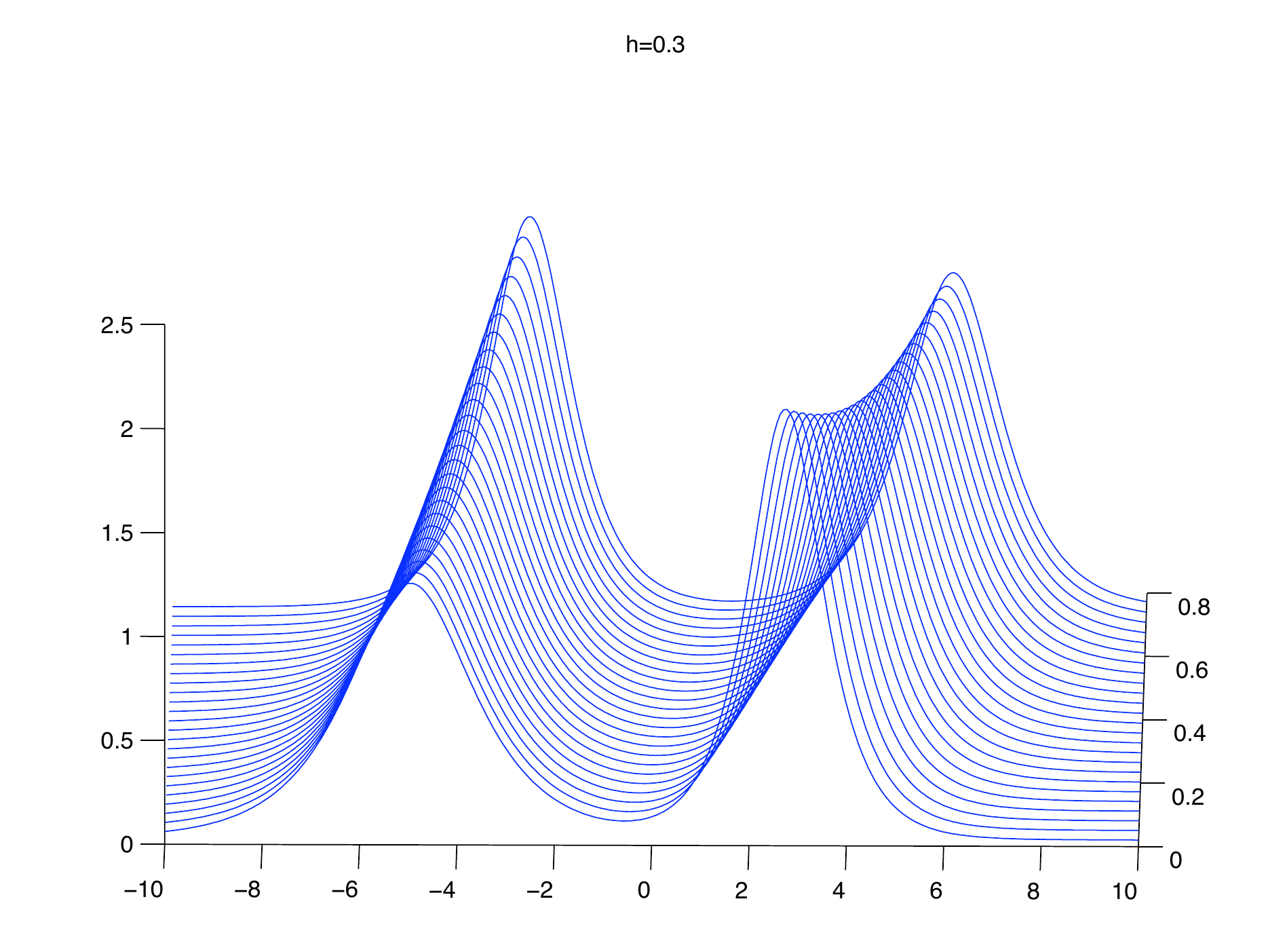}
\end{center}
\caption{The plots of $ q_2 ( x , c, A/h ) $ for 
$ (A_j , c_j ) $, $ j = 1, 2 $, solving
\eqref{eq:dynh} for
for $ b_0 ( x , t) = \cos^2 x $ and initial data 
$ A_1 (0) = - \pi/3 $, $ A_2 (0 ) = \pi/6 $, and 
$ c_1 (0) = \sqrt 3 \cos(\pi/3) $, $  c_2 (0) = \sqrt 3 \cos(\pi/6)$.
On the left $ h = 0.1 $ and on the right $ h = 0.3 $.}
\label{f:A2}
\end{figure}

We conclude with two heuristic observations. If the decoupled
equations lead to \eqref{eq:crossC} and 
$ | A_1 - A_2 | > \epsilon > 0 $ (which is the case when we 
approach the crossing in Fig.\ref{f:A1}) then equations
\eqref{eq:dynh} differ from \eqref{eq:ODEsh} by terms of 
size 
\[  h \log \left( \frac {c_2 -c_1 }{c_1 + c_2 } \right) \,, \]
see Lemma \ref{L:q-asymp}. 
For this to affect the motion of trajectories on finite
time scales in $ T $ we need 
\begin{equation}
\label{eq:hexp}
  c_2 - c_1 \simeq \exp \left(- \frac \gamma h \right)  \,. 
\end{equation}
This means that $ c_j $'s have to get exponentially close
to each other (but does not explain avoided crossing).

On the other hand if $ |a_1 - a_2 | > \epsilon > 0 $, where
$ a_j $'s are the original variables in \eqref{E:eom}, $A_j(0)=ha_j(0)$,
then we can use the decomposition in Lemma \ref{L:q-asymp} and
variables $ \hat a_j $ defined by \eqref{eq:hata}. The remark
after the proof of Lemma \ref{E:sf} shows that the equations
of motion take essentially the same form written in terms
of $ \widehat a_j $'s and $ c_j$'s and hence $ \hat a_j $
has to stay bounded. And that means that $ c_2 - c_1 $ is 
bounded away from $ 0 $. Hence, when $ c_2 - c_1 \rightarrow 0 $
we must also have $ a_2 - a_1 \rightarrow 0 $ as seen in 
Fig.\ref{f:cros} and Fig.\ref{f:A1}.

\section{Alternative proof of Lemma \ref{l:sP} (with Bernd Sturmfels)}

We note that the standard substition reduces the equation 
$ P ( c ) u = 0 $, where $ P ( c ) $ is defined in 
\eqref{eq:defP},  to an equation with rational coefficients:
\[ z = \tanh x \,, \ \ \partial_x = ( 1 - z^2 ) \partial_z \,,
\ \ \eta^2 = 1 - z^2 \,.  \]
This means that $ P ( c ) u = 0 $ is equivalent to $ Q ( c ) v = 0 $,
$ u (x ) = v ( \tanh x ) $, where
\[ Q ( c )  =  (L^2 + 1 ) 
 (L^2 + c^2 )  -
10 L  R ( z )  L  
+ 10 ( 3 R ( z )  - 2 R( z ) ^2 ) - 6 ( 1 + c^2) R( z) 
\,, \] 
and
\[ L = \frac 1 i ( 1 - z^2) \partial_z \,, \ \  \
R ( z ) = 1 - z^2 \,, \ \ - 1 < z < 1 \,. \]

Lemma \ref{l:sP} will follow from finding a basis of solutions
of $ Q ( c ) v = 0 $ and from seeing that the only bounded
solution is the one corresponding to $ \partial_x \eta $, that is,
to 
$$ v ( z ) = z ( 1 -z^2)^{\frac12} \,. $$
Remarkably, and no doubt 
because of some deeper underlying structure due to complete
integrability, this can be achieved using {\tt MAPLE} package
{\tt DEtools}.

First, the operator $ Q ( c) $ is brought to a convenient 
form 
\[ \begin{split} 
Q  = & \,  (z-1)^4 (z+1)^4 \frac{d^4} {dz^4} f ( z )  +12z(z-1)^3(z+1)^3
\frac{d^3}{dz^3} f ( z ) \\
& \, + (z-1)^2(z+1)^2(26z^2-c^2+1)\frac{d^2} {dz^2} f ( z ) 
\\
&  \,  -2z
(z-1)(z+1)(8z^2-11+c^2)\frac{d}{dz} f (z ) \\
& \,    + (4-20z^2 + 6 c^2z^2 - 5c^2 +16 z^2)f(z) 
\end{split}\]

Applying the {\tt MAPLE} command {\tt DFactorsols(Q,f(z))} gives the
following explicit basis of solutions to $ Q ( c) v = 0 $, $ c \neq 1 $:
\[ \begin{split}
& v_1 ( z ) = ( 1 - z^2)^{\frac 12} z \,, \\
& v_2 ( z ) = (1 + z)^{-\frac c2 } ( 1 - z)^{\frac c2 } ( ( c + z)^2 +
z^2 - 1 ) \,, \\
& v_3 ( z) = v_2 ( -z ) = ( 1 + z)^{\frac c2 } ( z -1)^{-\frac c2 } 
( ( c - z)^2 + z^2 - 1 )  \,, \\
& v_4 ( z ) = (1-z^2)^{-\frac12} 
 \left( -3z c^2 + 3z^3 c^2
- 7z^3 + 7 z \right) \log \frac{z+1}{z-1} \\
& \ \ \ \ \ \ \ \ \ \ \ + \, (1-z^2)^{-\frac12} \left(
 4c^2 - 6 c^2 z^2 + 14z^2 -12 \right) \,. 
\end{split} \]
For $ c \neq 1 $ these solutions are linearly independent and only $ v_1 $ 
vanishes at $ z = \pm 1 $ (or is bounded). Hence 
$ \ker_{L^2}  P( c)  $ is one dimensional proving Lemma \ref{l:sP}.


\begin{thebibliography}{00}

\bibitem{AKNS} M. Ablowitz, D. Kaup, A. Newell, and H. Segur, \emph{Nonlinear evolution equations of physical significance}, Phys. Rev. Lett. 31 (1973) pp. 125--127.

\bibitem{AW} W. Abou-Salem, 
{\em 
Solitary wave dynamics in time dependent potentials}, J.~Math.~Phys. {\bf 49},
032101 (2008).

\bibitem{AFS} W. Abou-Salem, J. Fr\"ohlich, and I.M. Sigal, 
{\em Colliding solitons for the nonlinear Schr\"odinger equation,} 
Comm. Math. Physics, {\bf 291}(2009), 151--176.

\bibitem{Kas}
N. Benes, A. Kasman, and K. Young, 
{\em On decompositions of the KdV 2-Soliton,}
J. of Nonlinear Science, {\bf 2}(2006),
179--200.

\bibitem{BS} J.L. Bona and R. Smith, \emph{The initial-value problem for the Korteweg-de Vries equation},  Philos. Trans. Roy. Soc. London Ser. A  278  (1975), no. 1287, pp. 555--601.

\bibitem{BSS} J.L. Bona, P.E. Souganidis, and W.A. Strauss, 
\emph{Stability and instability of solitary waves of Korteweg de Vries type},
Proc. Roy. Soc. London Ser. A 411 (1987), no. 1841, pp. 395--412. 

\bibitem{BR} A. Bouzouina and D. Robert,
{\em Uniform semiclassical estimates for the propagation 
of quantum observables,}  Duke Math. J. {\bf 111}(2002), 223--252.

\bibitem{BP} V. Buslaev and G. Perelman, \emph{On the stability of solitary waves for nonlinear Schr\"odinger equations}, Nonlinear evolution equations, editor N.N. Uraltseva, Transl. Ser. 2, 164, Amer. Math. Soc., pp. 75Ã¢98, Amer.
Math. Soc., Providence (1995).


\bibitem{DaVe} K.~Datchev and I.~Ventura, {\em Solitary waves
for the nonlinear Hartree equation with an external potential.}
{\tt arXiv:0904.0834}, to appear in Pacific.~J.~Math.

\bibitem{DJ} S.I. Dejak and B.L.G Jonsson, Long time dynamics of variable coefficient mKdV solitary waves, J. Math. Phys., 47, 2006. 

\bibitem{DS} S.I. Dejak and I.M. Sigal, \emph{Long-time dynamics of KdV solitary waves over a variable bottom},  Comm. Pure Appl. Math.  59  (2006), pp. 869--905.


\bibitem{FT} L.D. Faddeev and L.A. Takhtajan, \emph{Hamiltonian methods in the theory of solitons}, Springer-Verlag Berlin Heidelberg 2007, translated from the Russian by A.G. Reyman.

\bibitem{FrSi} J. Fr\"ohlich, S. Gustafson, B.L.G. Jonsson, and
I.M. Sigal, {\em Solitary wave dynamics in an external potential,}
Comm. Math. Physics, {\bf 250}(2004), 613--642.

\bibitem{ZS} Z. Gang and I.M. Sigal, \emph{On soliton dynamics in nonlinear Schr\"odinger equations}, Geom. Funct. Anal.  16  (2006),  no. 6, pp. 1377--1390. 

\bibitem{ZW} Z. Gang and M.I. Weinstein,
{\em Dynamics of nonlinear Schr\"odinger/Gross–Pitaevskii equations:  mass transfer in systems with solitons and degenerate neutral modes,}
Analysis \& PDE, {\bf 1}(3)(2008), 267--322.


\bibitem{HP} J. Holmer, \emph{Dynamics of KdV solitons in the presence of a slowly varying potential}, arxiv.org preprint \texttt{arXiv:1001.1583 [math.AP]}.

\bibitem{Codes} J. Holmer, G. Perelman, and M. Zworski, 
{\em 2-solitons in external fields}, on-line presentation with
{\tt MATLAB} codes,
{\tt http://math.berkeley.edu/$\sim$zworski/hpzweb.html}.

\bibitem{HZ1} J. Holmer and M. Zworski, \emph{Slow soliton interaction with delta impurities}, J. Modern Dynamics 1 (2007), pp. 689--718.

\bibitem{HZ2} J. Holmer and M. Zworski, \emph{Soliton interaction with slowly varying potentials}, IMRN Internat. Math. Res. Notices 2008 (2008), Art. ID runn026, 36 pp. 

\bibitem{KPV2} C.E. Kenig, G. Ponce, L. Vega, \emph{Well-posedness and scattering results for the generalized Korteweg-de Vries equation via the contraction principle},  Comm. Pure Appl. Math.  46  (1993), pp. 527--620.

\bibitem{KPV} C.E. Kenig, G. Ponce, and L. Vega, \emph{Well-posedness of the initial value problem for the Korteweg-de Vries equation},  J. Amer. Math. Soc.  4  (1991),  no. 2, pp. 323--347. 

\bibitem{KMR}  J. Krieger, Y. Martel, and P. Raphael, \emph{Two soliton 
solutions to the three dimensional gravitational Hartree equation},
Comm. Pure Appl. Math.
{\bf 62}(2009), 1501--1550.

\bibitem{Lax} P. Lax, \emph{Integrals of nonlinear equations of evolution and solitary waves}, Comm. Pure Appl. Math, 21 (1968) 467--490.

\bibitem{MS} J. Maddocks, R. Sachs, \emph{On the stability of KdV multi-solitons},  Communications on Pure and Applied Mathematics, 46 (1993) pp. 867--901.


\bibitem{MM}  Y. Martel and F. Merle, \emph{Description of two soliton collision for the quartic gKdV equation}, \\ arxiv.org preprint \texttt{arXiv:0709.2677}.

\bibitem{MMT}  Y. Martel, F. Merle, and T.-P. Tsai,  
\emph{Stability and asymptotic stability in the energy space of the sum of N solitons for subcritical gKdV equations},  Comm. Math. Phys.  231  (2002),  no. 2, pp. 347--373.

\bibitem{Miura} R. Miura, \emph{Korteweg-de Vries equation and generalizations. I. A remarkable explicit nonlinear transformation},  J. Math. Phys. 9 (1968) 1202.

\bibitem{Munoz} C. Mu$\tilde{\text n}$oz, \emph{On the soliton dynamics under a slowly varying medium for generalized KdV equations}, arxiv.org preprint \texttt{arXiv:0912.4725 [math.AP]},  to appear in Analysis \& PDE.


\bibitem{Olver} P. Olver, \emph{Applications of Lie groups to differential equations }.  



\bibitem{P} G. Perelman, \emph{Asymptotic stability of multi-soliton solutions for nonlinear Schr\"odinger equations},  Comm. Partial Differential Equations  29  (2004),  no. 7-8, 1051--1095.

\bibitem{Po} T. Potter, 
{\em Effective dynamics for $N$-solitons of the Gross-Pitaevskii equation,}
in preparation.


\bibitem{RSS}  I. Rodnianski, W. Schlag, A. Soffer, 
\emph{Asymptotic stability of N-soliton states of NLS},
arxiv.org preprint \texttt{arXiv:math/0309114}.

\bibitem{phys-paper} K.E.~Strecker et. al., \emph{Formation and propagation of matter wave soliton trains}, Nature {\bf 417}(2002), 150--154.




\bibitem{Tr} L.N.~Trefethen, {\em Spectral Methods in MATLAB,}
SIAM, Philadelphia, 2000

\bibitem{W} M. Wadati, \emph{The modified Korteweg-de Vries equation}, J. Phys. Soc. Jpn. 34 (1973) pp. 1289-1296  
\end{thebibliography}
\end{document}